\documentclass[12pt]{amsart}
\usepackage{amssymb}
\usepackage{a4wide}
\usepackage{subcaption}
\usepackage{tikz}
\usepackage{url}
\usetikzlibrary{arrows.meta}
\tikzset{
  >={Latex[width=2mm,length=2mm]},
            base/.style = {rectangle, rounded corners, draw=black,
                           minimum width=3.5cm, minimum height=1cm,
                           text centered},
}

\newtheorem{lemma}{Lemma}[section]
\newtheorem{theorem}[lemma]{Theorem}
\newtheorem{definition}[lemma]{Definition}
\newtheorem{corollary}[lemma]{Corollary}
\theoremstyle{remark}
\newtheorem{remark}[lemma]{Remark}
\newtheorem{example}[lemma]{Example}

\newcommand{\eps}{{\varepsilon}}
\newcommand{\R}{{\mathbb R}}
\newcommand{\Z}{{\mathbb Z}}
\newcommand{\N}{{\mathbb N}}

\renewcommand{\P}{\mathcal{P}}
\newcommand{\Int}{\mathrm{int}}
\newcommand{\Ext}{\mathrm{ext}}

\renewcommand{\theenumi}{\roman{enumi}}\renewcommand{\labelenumi}{(\theenumi)} 

\title{Symplectic billiards for pairs of polygons}

\author{Peter Albers, Fabian Lander, and Jannik M. Westermann}
\address{Institut f\"ur Mathematik, Universit\"at Heidelberg,
Im Neuenheimer Feld 205, 69120 Heidelberg, Germany}
\email{peter.albers@uni-heidelberg.de}
\email{fabian.lander@mis.mpg.de}
\email{jmwesterm@gmail.com}

\date{\today}

\begin{document}

\begin{abstract} 
We introduce symplectic billiards for pairs of possibly non-convex polygons. After establishing basic properties, we give several criteria on pairs of polygons for the symplectic billiard map to be fully periodic, i.e.~\textit{every} orbit is periodic. The first fully periodic examples were discovered by Albers--Tabachnikov \cite{Albers_Tabachnikov_Introducing_symplectic_billiards} and Albers--Banhatti--Sadlo--Schwartz--Tabachnikov in \cite{Albers_Banhatti_Sadlo_Schwartz_Tabachnikov_2025}. Our criteria allow us to construct a plethora of new examples. Moreover, we provide an example of a pair of polygons where the symplectic billiard map is fully periodic while having orbits of arbitrarily large period. After giving a class of examples which  provably have isolated periodic orbits (and are thus not fully periodic) we exhibit the first example without any periodic orbits at all. It is open whether having no periodic orbits at all is possible in the single polygon setting. Finally, we prove that if one replaces polygons by smooth, strictly convex curves then there are always infinitely many periodic orbits.
\end{abstract}

\maketitle

\section{Introduction}

In this article we introduce symplectic billiards for pairs of polygons which, in addition, are allowed to be non-convex. We briefly treat the case of pairs of smooth, strictly convex curves at the end.  Symplectic billiards was studied first in \cite{Albers_Tabachnikov_Introducing_symplectic_billiards} for smooth, strictly convex curves and convex polygons in the plane. The higher dimensional case was treated there as well. In \cite{Albers_Banhatti_Sadlo_Schwartz_Tabachnikov_2025}, the polygonal case was investigated in great detail, in particular, using computer simulations. In \cite{Schwartz_Symplectic_tiling} Schwartz introduced a version of tiling billiards stemming from the symplectic reflection rule.

There are several motivations to extend the definition of symplectic billiards from one billiard table to a pair of tables. It was already recognized in \cite{Albers_Tabachnikov_Introducing_symplectic_billiards} that the even and odd parts of a symplectic billiard trajectory separately have interesting properties. For instance, recording the even / odd parts of a symplectic billiard trajectory in two copies of the table exhibits hidden symmetries, see the proof of Theorems 8 and 9 in \cite{Albers_Tabachnikov_Introducing_symplectic_billiards}. This is exactly symplectic billiards on two copies of the same table. Allowing a pair of (in general different) tables is a natural generalization as we will explain below. 

Another motivation comes from Minkowski billiards, first studied by Gutkin--Tabachnikov in \cite{Gutkin_Tabachnikov_Billiards_in_Finsler}, i.e.~billiards where the reflection rule on the ``table'' is with respect to a Minkowski metric, that is, a Finsler metric which is constant. Thus a Minkowski metric is described by one ``unit ball''. Minkowski billiards has a very symmetric description where the role of the table and that of the unit ball are interchangeable, see Definition 2.9, Remark 2.10 and Figure 1 in the article \cite{Artstein_Avidan_Ostrover_Bounds_for_Minkowski_billiard} by Artstein-Avidan--Ostrover. 

Finally, if one considers symplectic billiards on an arbitrary table in $\R^2$ together with the standard unit disk (as second table) then it turns out that the induced map on the arbitrary table is simply Euclidean billiards. In other words, Euclidean billiards is a very special case of symplectic billiards on a pair of tables. This observation will be pursued in future work.

The above reasons motivated us to extend the definition of symplectic billiards to pairs of tables. At the same time we explore an extension to non-convex tables. This new point of view allows us to systematically prove certain experimentally made observations from \cite{Albers_Banhatti_Sadlo_Schwartz_Tabachnikov_2025} but also to  construct many interesting new examples. The definition for pairs of tables and non-convex polygons works verbatim in the higher dimensional case but is outside the scope of this article.

In \cite{Albers_Tabachnikov_Introducing_symplectic_billiards}  the rule for the symplectic billiard reflection is derived from a variational point of view, namely extremizing enclosed area. As illustrated in Figure \ref{fig:symplectic_billiard_reflection}, the pair $(x,y)$ of points on the curve is reflected by the symplectic billiard map to the pair $(y,z)$ if and only if the line $xz$ is parallel to the tangent line at $y$. 
This continues to hold if we consider two distinct polygons $P_-$ and $P_+$ (or curves) and impose the same billiard reflection rule, that is, the segment $xz$ for points $x,z\in P_-$ is parallel to the side containing $y\in P_+$, see Figure \ref{fig:symplectic_billiard_reflection}, and similarly with the roles of $P_-$ and $P_+$ reversed. The previous, single table case is included by simply considering $P_-=P_+$. 

\begin{figure}[ht]
\includegraphics[width=.9\linewidth]{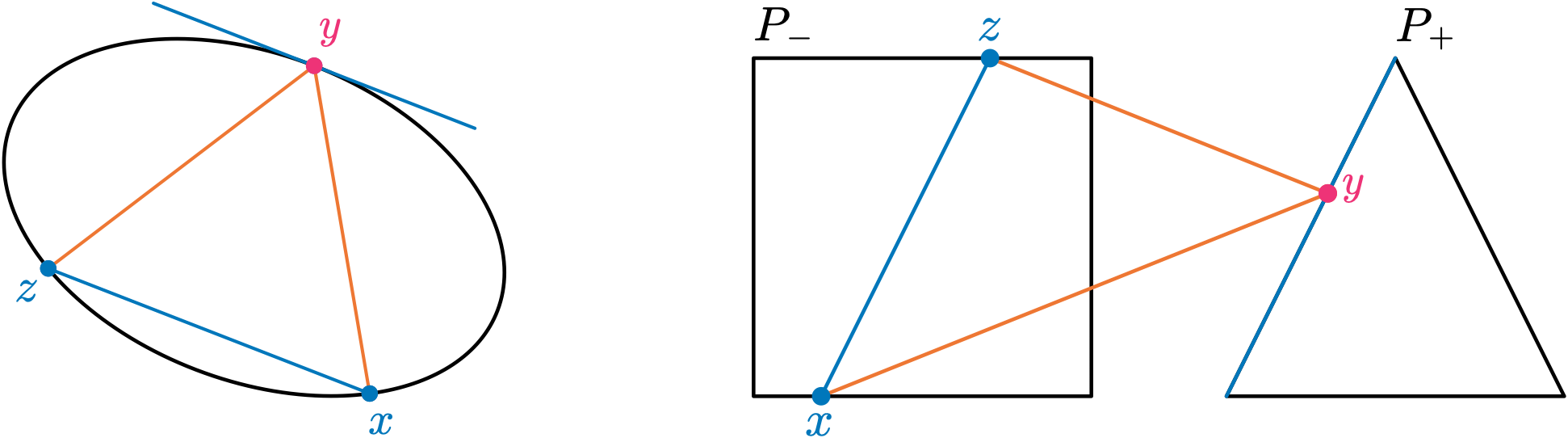} \qquad\qquad
\caption{The symplectic billiard reflection in a curve and in two polygons.}
\label{fig:symplectic_billiard_reflection}
\end{figure}


In Section \ref{sec:symplectic_billiard_map_on_pairs_of_polygons} we provide the necessary notions and definitions. The non-convexity makes it necessary to deal with certain subtleties. For instance, proving that the symplectic billiard map can be iterated infinitely often forward and backward  on a set of full measure requires quite some more work than in the convex case. As part of this discussion we introduce the subsets $C\subset F\subset P_- \sqcup P_+$ of critical points resp.~filled set of vertices and $N\subset C^\#\subset (P_-\times P_+) \sqcup (P_+\times P_-)$, the discontinuity set resp.~the $C$-grid. 

In Section \ref{section:criteria} we then formulate and prove the main assertions, Theorems \ref{theorem:F_finite} and \ref{theorem:periodicity-criterion}, concerning periodicity criteria for pairs of polygons in terms of the sets $F$, $C$ and $C^\#$. This goes back to the discovery in \cite{Albers_Banhatti_Sadlo_Schwartz_Tabachnikov_2025} of several polygons for which \textit{every} symplectic billiard trajectory is periodic. For very few cases, e.g.~the Quad, an ad-hoc argument was provided in \cite{Albers_Banhatti_Sadlo_Schwartz_Tabachnikov_2025}. Full periodicity of the symplectic billiard map was experimentally observed for a number of other cases, too. Theorems  \ref{theorem:F_finite} and \ref{theorem:periodicity-criterion} give a systematic treatment of this phenomenon:

\begin{theorem}[\text{cf. Theorem }\ref{theorem:periodicity-criterion} (III)]
Fix a pair of polygons $P_-$ and $P_+$. If there are only finitely many classes of trajectories starting in vertices, more precisely, if the set $C$ of critical points is finite, then every symplectic billiard orbit is periodic with a uniform period bound. 
\end{theorem}

The generalization to pairs of polygons was instrumental in arriving at our proofs and the statements of the theorems. In turn, understanding why every symplectic billiard orbit on certain pairs of polygons is periodic makes it possible to give big classes of examples, see e.g.~Corollary \ref{corollary:three_directions}. Moreover, we exhibit the first examples where every symplectic billiard orbit is periodic but there is no uniform upper period bound:

\begin{theorem}[\text{cf. Theorem }\ref{theorem:FP_but_not_BP}]
There is a pair of polygons $P_-$ and $P_+$ for which every symplectic billiard orbit is periodic; however, their periods are not uniformly bounded. 
\end{theorem}

In contrast to the fully periodic cases, in Section \ref{section:kite} we prove in detail that there are quadrilaterals, called (north-east) ne-quadrilaterals, where not every orbit is periodic. In fact, we show that certain periodic orbits are isolated:

\begin{theorem}[\text{cf. Section }\ref{section:kite}]
There is an open family of single table polygons each having an isolated periodic symplectic billiard orbit. 
\end{theorem}

The set of these ne-quadrilaterals forms an open set in the space of quadrilaterals up to affine transformations. While it was ``evident''  using computer simulations that there are examples with non-periodic orbits there was, so far, no proof of this fact. The proof of the existence of isolated periodic orbits uses an interesting geometric observation about ne-quadrilaterals. 

In Section \ref{section:necktie} we discuss a pair of convex polygons called the necktie, see Figure \ref{fig:necktie_intro}. 
\begin{figure}[ht]
    \centering
    \includegraphics[width=0.06\linewidth]{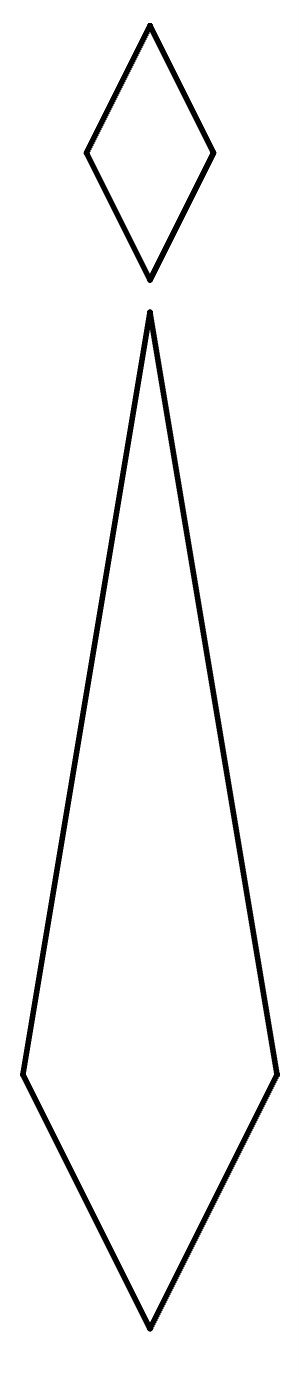}
    \caption{The necktie -- a pair of polygons for which the symplectic billiard map has no periodic orbits at all, see Section \ref{section:necktie}.}
    \label{fig:necktie_intro}
\end{figure}
We prove that in this example the symplectic billiard map possesses not a single periodic orbit:

\begin{theorem}[\text{cf. Section \ref{section:necktie}}]
    There is a pair of convex polygons for which the symplectic billiard map has no periodic orbits at all.
\end{theorem}

This is in stark contrast to the following, where we prove that for pairs of smooth, strictly convex curves there always exist infinitely many periodic orbits:

\begin{theorem}[\text{cf. Theorem }\ref{theorem:smooth_existence}]
    Every pair of smooth, strictly convex tables has infinitely many periodic symplectic billiard orbits.
\end{theorem}

The necktie is the first example of a pair of polygons without any periodic orbits. It is an open question whether such examples exist in the single table setting. The proof that the necktie does not have periodic orbits involves exhibiting a subset of its phase space which every orbit visits and on which the symplectic billiard map induces a return map. This return map is (conjugate to) the so-called dyadic odometer resp.~the von Neumann-Kakutani transformation which has no periodic points.

\subsection*{Acknowledgements} The authors are partially supported by the Deutsche Forschungsgemeinschaft under Germany’s Excellence Strategy EXC2181/1 - 390900948 (the Heidelberg STRUCTURES Excellence Cluster), the Collaborative Research Center SFB/TRR 191 - 281071066 (Symplectic Structures in Geometry, Algebra and Dynamics), and the Research Training Group RTG 2229 - 281869850 (Asymptotic Invariants and Limits of Groups and Spaces).

This article tremendously benefited from discussions with Rich Schwartz and Sergei Tabachnikov for which we warmly thank both of them.

We thank the anonymous referee for a very thorough report and many helpful suggestions!

\section{The symplectic billiard map on pairs of polygons}\label{sec:symplectic_billiard_map_on_pairs_of_polygons}

Let us begin by describing the setting in certain detail. For two points $x,y\in\R^2$ we denote by $xy$ resp.~$\overline{xy}$ the open resp.~closed segment between $x$ and $y$.


\begin{definition}
Let $v_1, \dots, v_n \in \R^2$ be pairwise distinct such that any two open segments $v_i v_{i+1}$, $v_j v_{j+1}$, $i\neq j$, do not intersect, nor does an open segment $v_jv_{j+1}$ contain any $v_i$. The piecewise linear, closed curve comprised of $\overline{v_iv_{i+1}}$, with $v_{n+1} := v_1$, is called an (embedded) polygon $P$. We call $V = \{v_1, \dots, v_n\} $ the set of vertices and $v_1 v_{2}, \dots, v_n v_{n+1}$ (open) edges of $P$.
\end{definition}

\begin{remark} \label{remark_P_as_S1}
    An alternative description of a polygon $P$ is as a continuous injective map $P:S^1\to \R^2$ whose image is contained in a finite union of lines. In the following we always assume that three consecutive points $v_i, v_{i+1}, v_{i+2}$ are not co-linear. Moreover, we read indices cyclically. 
\end{remark}

\begin{definition}
    Let $P$ be a polygon. Then $\R^2 \setminus P$ is divided into exactly two connected components, one of which is bounded. The bounded component $\Int(P)$ is called the interior of $P$, the other component $\Ext(P)$ is called the exterior of $P$.
\end{definition}

\begin{remark}\label{remark:outside_inside}
The interior of a polygon is well-defined by the polygonal version of the Jordan curve theorem, see for instance \cite{Courant_Robbins}. In particular, locally near a point on an edge of $P$ there is a unique ``outside'' and ``inside''.
For a point $x$ on an edge $v_i v_{i+1}$ we denote by $\nu_x$ the outer unit normal vector, that is, the vector of unit length that is orthogonal to $v_i v_{i+1}$ and points into the exterior of $P$ at $x$, i.e.~$x + \eps \nu_x \in \Ext(P)$ for any sufficiently small $\eps > 0$.
\end{remark}

\begin{remark}\label{remark:billiard_rule}
We will extend the definition of symplectic billiards from \cite{Albers_Tabachnikov_Introducing_symplectic_billiards} to non-convex polygons and to pairs of polygons. Most of this article is concerned with pairs of convex polygons. Therefore, the reader may choose to skip the non-convex generalization. Before giving rigorous definitions and statements we will showcase the idea.

Recall that for a convex polygon $P$, the reflection law for the symplectic billiard map is the following rule. Let $x$ and $y$ be points on $P$ that lie on non-parallel edges. Given such points $x$ and $y$ on $P$, choose $z$ as the unique intersection of $P$ and the line $x + T_yP$ other than $x$, see Figure \ref{fig:convex_rule}. 

\begin{figure}[ht]
    \centering
    \includegraphics[width=.35\linewidth]{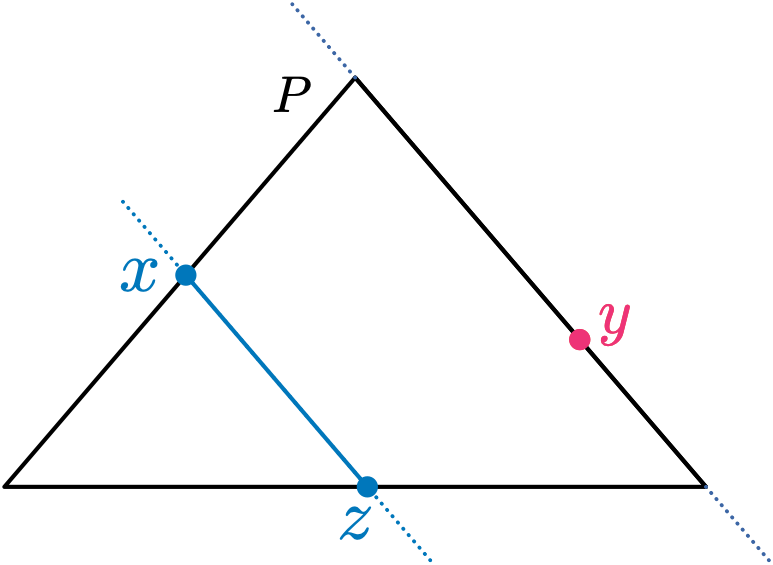}
    \caption{The symplectic billiard map on a convex polygon.}
    \label{fig:convex_rule}
\end{figure}

%
%
This setting 
can be generalized to two, possibly non-convex polygons (also called tables) as follows.
    Given two polygons $P_-$ and $P_+$ and two points $x,y$ on non-parallel edges and not on the same polygon, i.e.~$x\in P_\pm$ and $y\in P_\mp$, 
    we modify the above rule to choose $z$ as the unique intersection (other than $x$) of $P_\pm$ with the line $x + T_yP_\mp$ that satisfies $xz \subset \Int(P_\pm)$, see Figure \ref{fig:disjoint_rule}. 

\begin{figure}[ht]
    \centering
    \includegraphics[width=.75\linewidth]{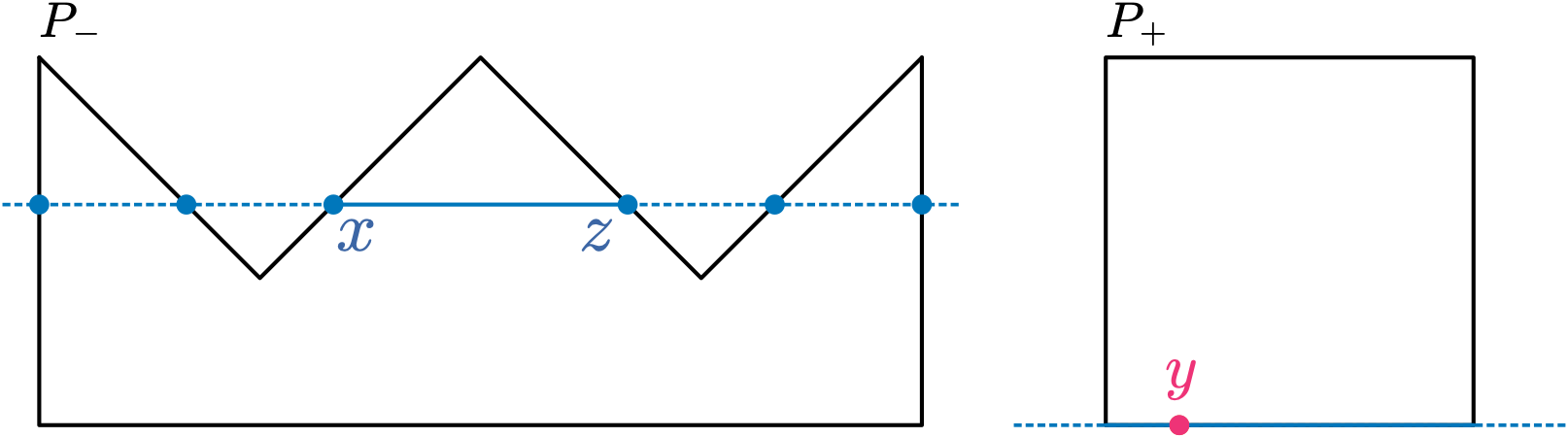}
    \caption{The symplectic billiard map on two polygons.}
    \label{fig:disjoint_rule}
\end{figure}

The plus and minus signs are only for convenience of notation and have no deeper meaning. It is useful to assume them to be disjoint for visualization, as the symplectic billiard dynamics is not changed when translating the polygons. Writing $P_\pm$ and $P_\mp$ indicates a sign choice and its opposite choice.

Note that for simplicity we draw the polygons disjointly next to each other. The previous case of a single polygon $P$ is subsumed (and generalized to non-convex tables) by setting $P_+:=P_-:=P$. We will focus on the case of two tables.

Setting $x_0:= x\in P_\pm$, $x_1:= y\in P_\mp$ and $x_2:= z\in P_\pm$, we can repeat applying this rule by starting at $x_1\in P_\mp$, moving along the tangent of $x_2\in P_\pm$ into the interior of $P_\mp$, until hitting the boundary, which determines a new point $x_3\in P_\mp$. Iterating this gives a sequence $(x_k)_{k\in\N_0}$ which is the forward part of the symplectic billiard trajectory of the pair $(x_0,x_1)\in P_\pm\times P_\mp$. See Figure \ref{fig:first_iterations} for the first few iterations. The even trajectory $(x_{2k})_{k\in\N_0}$ stays in $P_\pm$ (blue in Figure \ref{fig:first_iterations}) and the odd trajectory $(x_{2k-1})_{k\in\N}$ stays in $P_\mp$ (red in Figure \ref{fig:first_iterations}). Note that inserting $(x_{i+1},x_i)$ into our rule (mind the order), we obtain $x_{i-1}$. Using this reversibility we will consider $x_{-1},\,x_{-2},\ldots$ as the backward part of the trajectory.

\begin{figure}[ht]
    \centering
    \includegraphics[width=.5\linewidth]{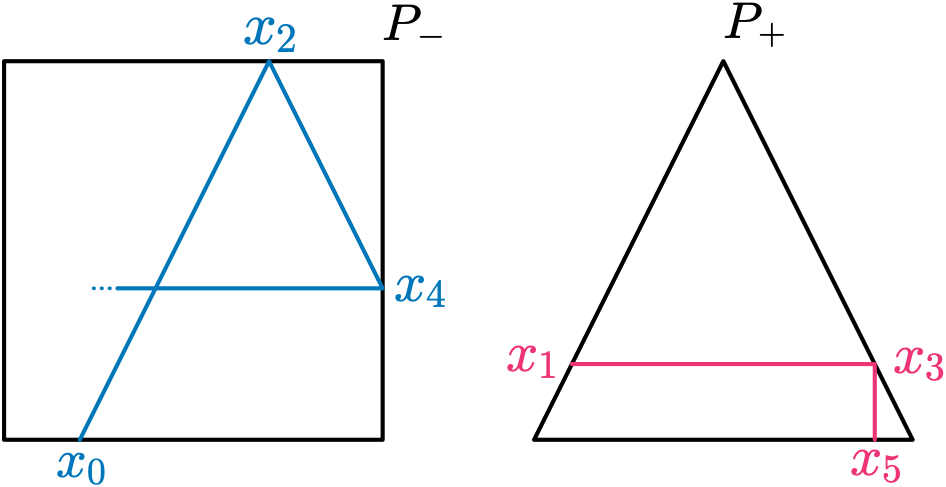}
    \caption{A part of a symplectic billiard trajectory on two polygons.}
    \label{fig:first_iterations}
\end{figure}

\end{remark}

The symplectic billiard rule is not well-defined for all $(x,y)$, however. For example $y$ should not be a vertex or $x$ and $y$ should not lie on parallel edges. When $x$ is a vertex, the rule may or may not be well-defined, see Figure \ref{fig:map_not_defined} for examples.

\begin{figure}[ht]
    \centering
    \includegraphics[width=.6\linewidth]{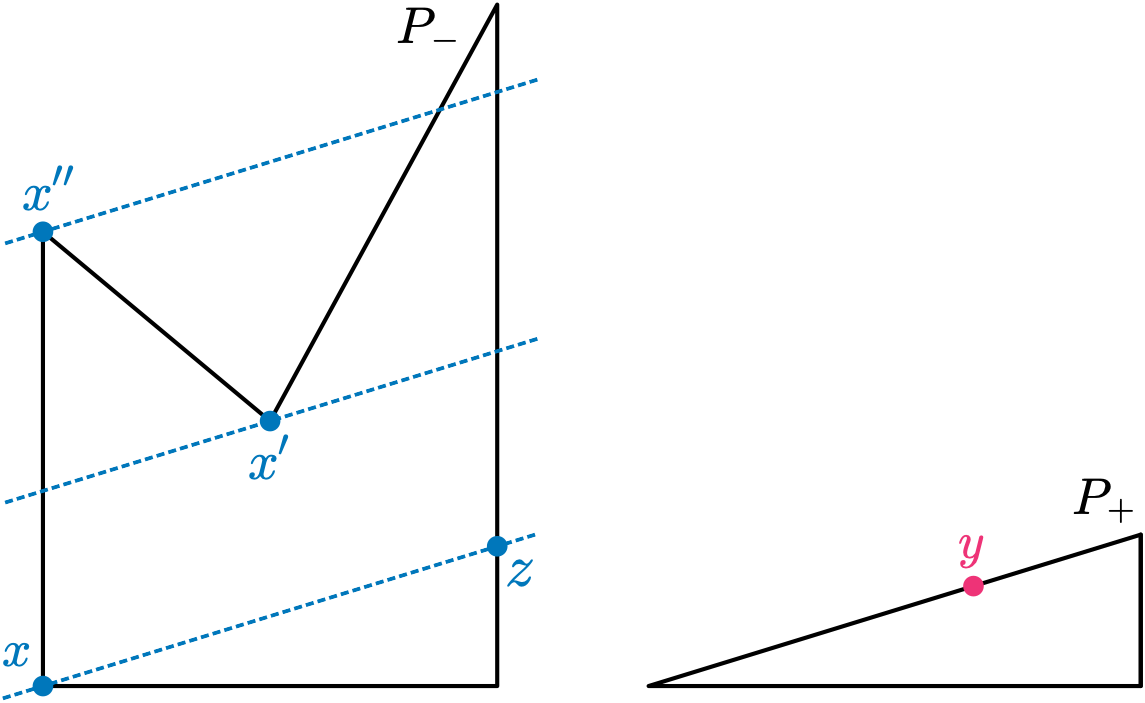}
    \caption{The symplectic billiard map sends $(x,y)$ to $(y,z)$, however it is not defined at the points $(x',y)$ and $(x'',y)$
    }
    \label{fig:map_not_defined}
\end{figure}

We say that a pair $(x,y)\in P_\pm \times P_\mp$ satisfies \textbf{condition A} if 
$$
y \notin V_\mp \text{ and }\exists!\, z \in P_\pm \cap (x + T_yP_\mp \setminus \{0\}) \text{ such that }xz \subset \Int(P_\pm). 
$$
We define the maximal phase space by
$$
\P_\mathrm{max}:= \{(x,y) \in (P_- \times P_+)\sqcup (P_+\times P_-)\mid  (x,y) \text{ satisfies condition A}\}. 
$$
This is the biggest set on which the symplectic billiard map is defined (for one iteration).

\begin{definition}\label{def:sympl._billiard_map}
    The symplectic billiard map is defined as
    \begin{equation}\nonumber
    \begin{aligned}
        \phi: \P_\mathrm{max} &\longrightarrow (P_-\times P_+) \sqcup (P_+\times P_-)\\
        (x,y) &\mapsto (y,z)
    \end{aligned}
    \end{equation}
    where $z$ is the unique point from condition A.
    \end{definition}

\begin{remark}\label{remark:single_tables_are_a_special_case}
The definition of the symplectic billiard map on two tables agrees with the existing notion of symplectic billiard on single table convex polygons and extends this implicitly to single table non-convex polygons as well (taking a non-convex polygon as $P_- = P_+$). However, the single table setting allows for orbits of odd periods, and in the two table counterpart the respective orbit has twice the period. That is, a three-periodic trajectory $(\ldots,x,y,z,x,y,z,\ldots)$ in a single polygon $P$ becomes $(\ldots,x_+,y_-,z_+,x_-,y_+,z_-,\ldots)$ in $(P\times P)\sqcup (P\times P)$ with $x_\pm=x$ etc.
\end{remark}

We want to show that the maximal phase space is a ``fat'' set. For that let us prove that it contains the set $\P'$ of pairs of points on non-parallel edges, i.e.~contains
$$
    \P' := \{(x,y) \in (P_- \setminus V_- \times P_+\setminus V_+) \sqcup (P_+\setminus V_+ \times P_-\setminus V_- ) \mid\det(\nu_x,\nu_y)  \neq  0\}
$$
where we recall that $V_\pm\subset P_\pm$ is the respective vertex set.

\begin{lemma}\label{lemma:unique_intersection_point}
    For any $(x,y) \in  \P' $ there exists a unique $z \in P_\pm \cap (x + T_y P_\mp \setminus \{0\})$ such that $xz \subset \Int(P_\pm)$. Hence $\P'$ is contained in $\P_\mathrm{max}$. Moreover, if $z$ is not a vertex then $\det(\nu_y,\nu_z) \neq 0$, i.e.~$(y,z) \in \P'$. 
\end{lemma}

\begin{proof}
    Let $(x,y)\in \P'$. Since $\det(\nu_x,\nu_y) \neq 0$ there exists $z\in P_\pm\cap(x+T_yP_\mp\setminus \{0\})$ with $xz\subset\Int(P_\pm)$. Now, we assume that there exists another point $z'\in P_\pm\cap(x+T_yP_\mp\setminus \{0\})$ with $xz'\subset\Int(P_\pm)$. Locally near $x$ there is a unique ``outside'' and ``inside'' of $P$, see Remark \ref{remark:outside_inside}. Therefore, $z-x$ and $z'-x$ point  in the same direction. We may assume without loss of generality $xz\subset xz'$. If $z\neq z'$ then $z$ lies in the interior of $P_\pm$ and not on $P_\pm$, therefore $z=z'$.
    
    For the second claim assume that $z$ is not a vertex and that $\nu_y$ and $\nu_z$ are parallel. Since $z$ is not a vertex and $x-z$ is orthogonal to $\nu_y=\pm\nu_z$, the segment $xz$ must be partially contained in the edge of $P_\pm$ containing $z$, which contradicts $xz \subset \Int(P_\pm)$.
\end{proof}

In general, the image of $\P_\mathrm{max}$ under $\phi$ is not contained in $\P_\mathrm{max}$ again since $z$ might be a vertex of $P_\pm$. Note that if $z$ is not a vertex, then $(y,z) \in \P'\subset \P_\mathrm{max}$ by Lemma \ref{lemma:unique_intersection_point}. Moreover, given $y\in P_\mp$ the condition 
\begin{equation}\nonumber
z \in P_\pm \cap (x + T_yP_\mp \setminus \{0\}) \text{ and }xz \subset \Int(P_\pm)
\end{equation}
on points $x,z\in P_\pm$ is symmetric in $x$ and $z$, i.e.~is equivalent to
\begin{equation}\nonumber
x \in P_\pm \cap (z + T_yP_\mp \setminus \{0\}) \text{ and }zx \subset \Int(P_\pm).
\end{equation}
Thus, if $z$ is not a vertex, then not only $(y,z) \in \P'\subset \P_\mathrm{max}$ but also $(z,y) \in \P'\subset \P_\mathrm{max}$, again by Lemma \ref{lemma:unique_intersection_point}. Using uniqueness in condition $\mathrm{A}$  the next Lemma follows.

\begin{lemma}\label{lemma:billiard_map_reversible_on_bigger_set}
    For any $(x,y) \in \P'$ with $\phi(x,y) = (y,z)$ such that $z$ is not a vertex, we have $(z,y)\in \P'$ and $\phi(z,y) = (y,x)$. 
\end{lemma}

\begin{remark}\label{remark:bigger_set_reversible}
It is easy to see that Lemma \ref{lemma:billiard_map_reversible_on_bigger_set} continues to hold even for points $(x,y) \in \P_\mathrm{max}$ or for $z$ being a vertex but not a ``non-convex vertex'', see Figure \ref{fig:non-reversible}. In particular, the issue in Figure \ref{fig:non-reversible} never arises in the convex case. 
\end{remark}

\begin{figure}[ht]
    \centering
    \includegraphics[width=.6\linewidth]{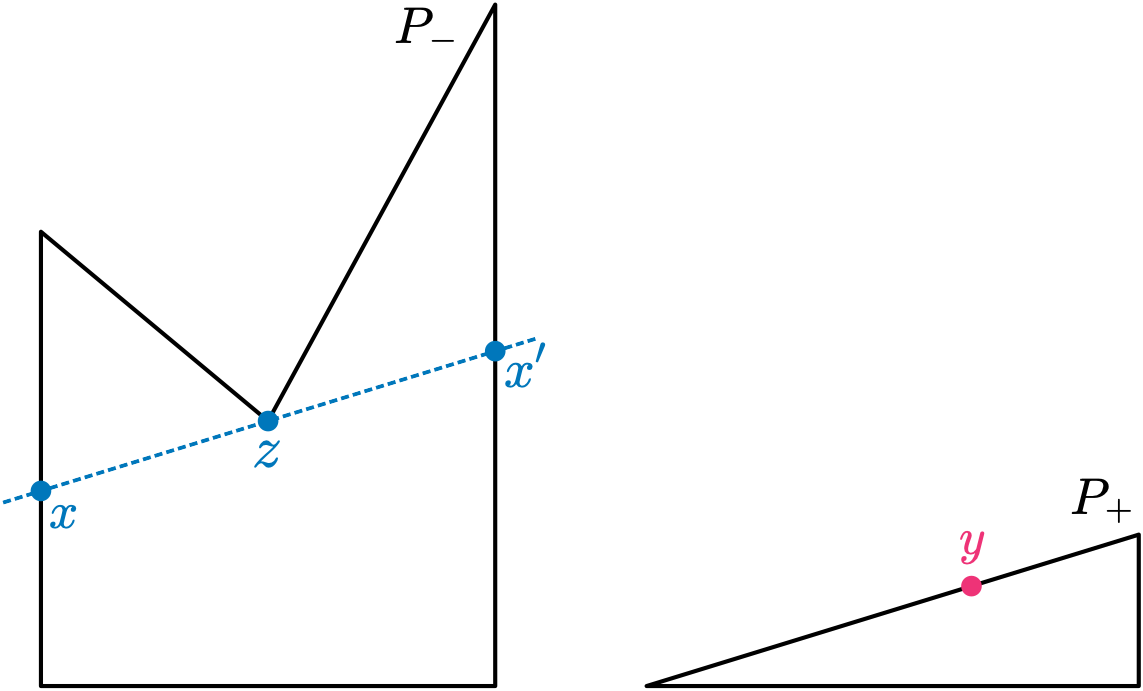}
    \caption{We have both $\phi(x,y) = \phi(x',y) = (y,z)$. The vertex $z$ is a ``non-convex'' vertex and $\phi(z,y)$ is not defined.}
    \label{fig:non-reversible}
\end{figure}

We can iterate the billiard map forwards and backwards infinitely often on the phase space $\P$ which we get if we remove a null set $N$ from $\P_\mathrm{max}$. This null set $N$ is the set of points for which the forward or backward trajectory hits a vertex. 
We give a formal definition for completeness. For ease of notation we introduce the swap map $\sigma (x,y) = (y,x)$. 

In the remainder of this section we define various sets and analyze their properties. For convenience we collect all these sets at the end of this section in Remark \ref{rmk:glossary}.

\begin{definition}\label{definition:iterable_phase_space}
    Let $(P_-,P_+)$ be a pair of polygons and $(V_-,V_+)$ their sets of vertices. Set $N_0 = (P_-\times V_+) \sqcup (P_+\times V_-)$ and recursively $N_{i+1} := \phi^{-1}(N_i)$, $i\in \N_0 $.
    Then the set
    $$
    N:=\bigcup_{i \in \N_0} N_i \cup \sigma\Big(\bigcup_{i \in \N_0} N_i\Big)
    $$
 is called the discontinuity set. 
 We define the phase space as 
    $$
    \P := \P_\mathrm{max} \setminus N = \P' \setminus N,
    $$
    see Remark \ref{rmk:P_m_setminus_N=P_prime_setminus_N} for the equality.
    Moreover, we define the forward phase space 
    $$
    \P_f := \{(x,y) \in \P \mid \det(\nu_x,\nu_y) > 0\}
    $$
    as well as the backward phase space
    $$
    \P_b := \{(x,y) \in \P \mid \det(\nu_x,\nu_y) < 0\}.
    $$
\end{definition}

\begin{remark}\label{rmk:P_m_setminus_N=P_prime_setminus_N}
We point out that $N_0$ is a finite union of lines and thus so is $N_i$. In particular, $N$ is indeed a null set.

The equality $\P_\mathrm{max} \setminus N = \P' \setminus N$ follows from $\P' \subset \P_\mathrm{max}$ and $\P_\mathrm{max}\setminus\P' \subset \sigma(N_0)$. In particular, we have that $\phi(x,y)=(y,z)$ for $(x,y)\in\P$ implies that $z$ is not a vertex. We also point out
$$
\P = \P_f \sqcup \P_b
$$
and $\sigma(\P_f)=\P_b$ resp.~$\sigma(\P_b)=\P_f$.
\end{remark}

\begin{lemma}\label{lemma:billard_map_maps_phase_space_into_itself}
   The symplectic billiard map $\phi$ maps $\P$, $\P_f$ resp.~$\P_b$ into itself, i.e.
    $$
    \phi(\P) \subset \P ,\,\phi(\P_f) \subset \P_f \text{ and } \phi(\P_b) \subset \P_b.
    $$
    In particular, the property of being in the forward resp.~backward phase space is preserved under $\phi$.
\end{lemma}

\begin{proof}
By symmetry of the backward and forward phase space, it suffices to show $\phi(\P_f)\subset \P_f$. 
Let $(x,y)\in \P_f$ and $(y,z) := \phi(x,y)$, in particular, $\det(\nu_x,\nu_y)>0$ holds. First let us argue that $\det(\nu_y,\nu_z)>0$ (note that $z$ is not a vertex since $(x,y)\notin N$). For a smooth, convex curve, this statement has been established in  \cite[Lemma 2.1]{Albers_Tabachnikov_Introducing_symplectic_billiards}. Here is an alternate proof that also works in the polygonal case. 
Write $J = \begin{pmatrix}  0 & -1 \\ 1 & 0\end{pmatrix}$ for the rotation by $\pi/2$. Note that for $v,w \in \R^2$ we have $\det(v,Jw) = \langle v, w\rangle = \det(-Jv,w)$. Now write $t= \tfrac{z-x}{\|z-x\|}$, which is parallel to the tangent at $y$ and thus $t = \pm J \nu_y$, respectively $\mp J t = \nu_y.$ 
Now for $(x,y)\in \P_f$ we have
$$ 0 < \det(\nu_x,\nu_y) = \det(\nu_x,\mp J t) = \mp \det(\nu_x, J t) = \mp \langle \nu_x, t\rangle. $$
Since $t$ is pointing into the interior of the polygon at $x$, we have $\langle \nu_x,t\rangle < 0$. Thus $ \nu_y = - J t$. Hence
$$\det(\nu_y,\nu_z) = \det(- J t, \nu_z) = \langle t, \nu_z \rangle > 0$$
since $-t$ is pointing into the interior at $z$. 

The inequality $\det(\nu_y,\nu_z)>0$ implies in particular that $(y,z)\in \P'$. It remains to show $(y,z)\notin N$. Assume otherwise, i.e.~$(y,z)\in N$. Then, by Definition \ref{definition:iterable_phase_space}, there is an $n \in \N_0$ such that either $\phi(x,y)=(y,z) \in \phi^{-n}(N_0)$ or $(z,y) \in \phi^{-n}(N_0)$. In the first case we conclude $(x,y) \in \phi^{-(n+1)}(N_0)$ and thus $(x,y)\in N$. This contradicts $(x,y) \in \P$. In the second case we recall that $z$ is not a vertex. Therefore, Lemma \ref{lemma:billiard_map_reversible_on_bigger_set} asserts $\phi(z,y) = (y,x)$. Hence $(y,x) \in \phi^{-(n-1)}(N_0)$ and, in particular, $(y,x)\in N$. Since $\sigma(N)=N$ we arrive at the same contradiction $(x,y) \in N$. This concludes the proof.
\end{proof}

\begin{corollary}
\label{corollary:invertible_on_phase_space}
The symplectic billiard map $\phi$ is invertible on $\P$ with $\phi^{-1} = \sigma \circ \phi \circ \sigma$. In particular, we have $\phi(\P) = \P$, $\phi(\P_f) = \P_f$ and $\phi(\P_b) = \P_b$.
\end{corollary}

\begin{proof}
First note that $\sigma(\P) = \P$ because of the invariance of $\P'$ and $N$ under $\sigma$.  From Lemma \ref{lemma:billiard_map_reversible_on_bigger_set} we obtain the equalities $\sigma\circ\phi\circ\sigma\circ\phi=\mathrm{id}$ and $\phi\circ\sigma\circ\phi\circ\sigma = \mathrm{id}$ on $\P$. Hence $\phi$ is invertible on $\P$ with $\phi^{-1} = \sigma\circ\phi\circ\sigma$. Lemma \ref{lemma:billard_map_maps_phase_space_into_itself} implies $\phi(\P_f) \subset \P_f$ and $\phi(\P_b) \subset \P_b$. Combining this with $\sigma(\P_f)=\P_b$ resp.~$\sigma(\P_b)=\P_f$ and $\phi^{-1} = \sigma\circ\phi\circ\sigma$ implies $\phi^{-1}(\P_f) \subset \P_f$ and $\phi^{-1}(\P_b) \subset \P_b$. This proves the Corollary.
\end{proof}

It is helpful to distinguish between an orbit of the symplectic billiard map $\phi$, which is a sequence of pairs $(x_k,x_{k+1})_{k\in\Z} \subset \P$, and a trajectory, i.e.~the corresponding sequence of points $(x_k)_{k\in\Z} \subset P_- \sqcup P_+$. In order to formalize this we denote by
\begin{equation}\nonumber
\begin{aligned}
\pi_1: (P_-\times P_+) \sqcup (P_+\times P_-) &\to P_- \sqcup P_+\\
\pi_2: (P_-\times P_+) \sqcup (P_+\times P_-) &\to P_- \sqcup P_+
\end{aligned}
\end{equation}
the projection to the first resp.~second factor.

\begin{definition}\label{definition:symplectic_billiard_trajectory} $ $
\begin{enumerate}  \itemsep=1ex 
\item The symplectic billiard orbit of $(x_0,x_1) \in \P$ is the sequence $(\phi^k(x_0,x_1))_{k\in\Z} \subset\P$.
\item The symplectic billiard trajectory $(x_k)_{k\in\Z}\subset P_- \sqcup P_+$ starting at $(x_0,x_1)\in \P$ is obtained by setting $x_k:= \pi_1(\phi^k(x_0,x_1))$. In particular,
    $$
    (x_k,x_{k+1}) = \phi(x_{k-1},x_k) ,\, k\in \Z, \text{ and } (x_k,x_{k+1}) = \phi^{-1}(x_{k+1},x_{k+2}) ,\, k\in \Z.
    $$
\item We call $(x_{2k})_{k\in\Z} \subset P_\pm$ the even trajectory and $(x_{2k-1})_{k\in\Z}\subset P_\mp$ the odd trajectory of $(x_0,x_1)$.
\item We denote by $(x_k^\sigma)_{k\in\Z} := (x_{-k+1})_{k\in\Z}$ the backward symplectic billiard trajectory starting at $(x_0,x_1)$. 
\end{enumerate}
 
\end{definition}

The notation of the backward trajectory is motivated by the following lemma.

\begin{lemma}
The backward trajectory $(x_k^\sigma)_{k\in\Z}$ starting at $(x_0,x_1)\in \P$ coincides with the forward trajectory $(y_k)_{k\in \Z}$ starting at $(y_0,y_1):=  (x_1,x_0)=\sigma(x_0,x_1) $. 
\end{lemma}

\begin{proof}
Combining $\pi_1 = \pi_2 \circ \sigma $ and $\pi_2\circ\phi^{k} = \pi_1\circ\phi^{k+1}$ with $\phi^{-k} = \sigma \circ \phi^k\circ\sigma $, $k\in \Z$, we see
    \begin{equation*}
        \begin{split}
            y_k &= \pi_1(\phi^k(y_0,y_1)) = \pi_1(\phi^k(\sigma(x_0,x_1))) = \pi_2(\sigma(\phi^k(\sigma(x_0,x_1)))) \\
            &= \pi_2(\phi^{-k}(x_0,x_1)) = \pi_1(\phi^{-k+1}(x_0,x_1)) = x_{-k+1}.
        \end{split}
    \end{equation*}
\end{proof}

\begin{remark}\label{remark:(semi-)finite_trajectories}
We recall that for $(x_0,x_1) \in \P_\mathrm{max}$ the symplectic billiard trajectory might only be defined for a finite number of (forward and/or backward) iterations. Nevertheless, we still call $(x_k)_{k_1\leq k\leq k_2}$ the symplectic billiard trajectory starting at $(x_0,x_1)$ where the number  $k_1\leq0$ of backward iterations resp.~$k_2-1\geq1$ of forward iterations is chosen to be minimal resp.~maximal. Note that $(x_0,x_1) \in \P$ if and only if $k_1 = -\infty$ and $k_2 = \infty$, by definition of $\P$.
\end{remark}

\begin{remark}
We point out that the symplectic billiard map commutes with affine transformations of the plane, as in \cite{Albers_Tabachnikov_Introducing_symplectic_billiards}. Of course, the same affine transformation has to be applied to both polygons $P_\pm$ at the same time in the two table setting.

Moreover, individually translating / scaling one polygon does not change the symplectic billiard rule.
\end{remark}

Next, we recall two lemmas from the single table setting in \cite{Albers_Banhatti_Sadlo_Schwartz_Tabachnikov_2025} that immediately generalize to two polygons. 
Firstly there is a piecewise constant area form on $\P'$ which amounts to weighting its various rectangular components by numbers coming from the angles of the polygons. The explicit weights can be found in an earlier version \cite{albers2024symplecticbilliardspairspolygons} of this article.

\begin{lemma}[Lemma 2.1 in \cite{Albers_Banhatti_Sadlo_Schwartz_Tabachnikov_2025}]\label{lemma:T_piecewise_affine}
The map $\phi$ is area preserving with respect to the measure induced from restricting the area form above to $\P$.
\end{lemma}

As in \cite{Albers_Banhatti_Sadlo_Schwartz_Tabachnikov_2025} we assign to a symplectic billiard trajectory the bi-infinite sequence of the edges that this trajectory hits. This sequence is called the \textit{symbolic trajectory}. Equivalence classes of points in the phase space with the same symbolic trajectory are called \textit{tiles}. In particular, $\phi$ maps tiles to tiles.
It turns out, see Lemma \ref{lemma:positive_tiles_are_periodic}, that tiles are \textit{phase rectangles}, that is, subsets of $\P$ of the form $xy\times wz$, $\{x\}\times wz$, $xy\times \{w\}$, or $\{x\}\times \{w\}$. 

\begin{lemma}\label{lemma:connected_is_contained_in_tile}
Every arc-wise connected component of $\P$ is contained in a tile. 
\end{lemma}

\begin{proof}
The symbolic trajectory can only change along a path of starting points if one of the corresponding symplectic billiard trajectories hits a vertex. We recall that for points in $\P$ we can iterate the symplectic billiard map infinitely many times and never hit a vertex of $P_-\sqcup P_+$. Thus the symbolic trajectory is constant along a path in $\P$.
\end{proof}

\begin{lemma}[Lemma 2.3 in \cite{Albers_Banhatti_Sadlo_Schwartz_Tabachnikov_2025}]
\label{lemma:positive_tiles_are_periodic}
Tiles are phase rectangles. If a tile is a genuine phase rectangle, that is, has a non-zero area, then its symbolic trajectory is periodic. Furthermore, every orbit in this tile is periodic. More precisely, let $M$ be a tile of positive area with a periodic symbolic trajectory of period $n$. Then $\phi^n$ maps $M$ to itself, and the return map $\phi^n$ has either order 4, or order 2, or it is the identity.
\end{lemma}

The above lemmas immediately imply the following corollary.

\begin{corollary}\label{cor:connected_components_are_tiles}
The  arc-wise connected components of $\P$ are precisely the tiles.
\end{corollary}

As in \cite{Albers_Banhatti_Sadlo_Schwartz_Tabachnikov_2025} we call a periodic symplectic billiard orbit \textit{isolated} if its tile has zero area in phase space. Note that an isolated periodic orbit does not mean that there exists a neighborhood of its starting points $(x_0,x_1) \in \P$ in which none of the orbits are periodic with the same period.

\subsection*{Critical points and filled set of vertices}
For our criteria for periodicity in the next section, we need the following sets. 
We will define the set of critical points $C$. First we give an explanation, then a formal definition. Let $(x_k)_{k\in\Z}$ be any symplectic billiard trajectory hitting a vertex, i.e.~ $x_K\in V_\pm$ for some $K\in\N$. Then $x_{K-2\ell}\in C$ for all $\ell\geq 0$. This description simplifies for convex polygons $P_-$ and $P_+$ since then we may reverse symplectic billiard trajectories hitting a vertex, see Remark \ref{remark:bigger_set_reversible}. Thus, for convex polygons the set $C$ consists of all points along even trajectories that start in vertices. Finally, we need to add all vertices in which no trajectories end. This happens for instance when $P_-=P_+$ is a square.
\begin{remark}\label{remark:intuition_for_C}
This description also leads to the following iterative procedure for determining the set $C$ of critical points for convex polygons. Start with an initial vertex $v_i^\pm \in P_\pm$ and an (open!) edge $v_j^\mp v_{j+1}^\mp \subset P_\mp$. Map the entire segment $\{v_i^\pm\} \times v_j^\mp v_{j+1}^\mp$ repeatedly by the symplectic billiard map $\phi$ until the first time the image contains a vertex. Thus, this image is split by vertices $v_{k_1}^\mp,\dots,v_{k_n}^\mp$ into segments. Trace these vertices back to points $c_1,\ldots,c_n$ in the initial segment $v_j^\mp v_{j+1}^\mp$. Note $c_1,\ldots,c_n \in C$. Now start over with the subsegments $\{v_i^\pm\} \times c_k c_{k+1}$, $k=0,\ldots, n$ where $c_0:= v_j^\mp$ and $c_{n+1}:= v_{j+1}^\mp$. This iterative procedure stops only if the even trajectory hits a vertex. We repeat this process for all initial vertices. This process is demonstrated in Examples \ref{example:quad}  and \ref{example_unbounded_periodic}.

For a non-convex polygon $P$ there is a ``non-convex vertex'', i.e.~a vertex that lies in the interior of the convex hull of the polygon, as for example in Figure \ref{fig:map_not_defined}. In this case there may exist trajectories that end in this vertex but for which we cannot consider a trajectory starting in this vertex in the reversed direction because the billiard map is not well-defined in this case, that is, the map is not reversible. In this situation we need to consider both choices of starting a billiard trajectory and continue as in the convex case.
\end{remark}

Here is the formal definition of $C$.

\begin{definition}\label{def:critical_points}
We define the set of critical points
    $$
    C := \bigcup_{i \in \N_0} \pi_2(N_{2i}) \subset P_- \sqcup P_+. 
    $$
    Furthermore we call 
    $$
    C^\# := \big((C \cap P_-) \times P_+\big) \cup \big((C\cap P_+) \times P_-\big) \cup \big(P_-\times (C\cap P_+)\big) \cup \big(P_+ \times (C\cap P_-)\big)
    $$
    the $C$-grid.
\end{definition}
We point out that, by definition, $C^\#$ contains the discontinuity set $N$, i.e.
$$
N\subset C^\#.
$$ 
It is also helpful to consider the filled set of vertices $F$ defined next. This set will be of importance also in the periodicity criterion Theorem \ref{theorem:F_finite} below. The idea is to go from all vertices of $P_\pm$ in all directions tangent to $P_\mp$ and collect the intersections with $P_\pm$ and repeat the process. In contrast, in $C$ we only take directions of actual symplectic billiard trajectories into account. From the above description it is then clear that the filled set of vertices contains the set $C$ of critical points, i.e.~$C\subset F$.

\begin{definition}\label{def:filled_set_of_vertices}
We set $V_0^\pm:=V_\pm$ and define recursively for $i \in \N$
\begin{equation*}
    V_i^\pm := \Big\{ v \in P_\pm \Big| \begin{array}{c}
    \exists w\in V_{i-1}^\pm \text{ \rm such that the segment } vw \text{ \rm satisfies}\\ vw \subset \Int(P)
    \text{ \rm and } vw \text{ \rm is parallel to an edge of } P_\mp 
\end{array}\Big\}.
\end{equation*}
The filled set of vertices is
\begin{equation*}
    F := F_- \sqcup F_+ := \bigcup_{i\in\N_0}V_i^- \sqcup \bigcup_{i\in\N_0} V_i^+.
\end{equation*}
\end{definition}

\begin{lemma}\label{lemma:F_C_countable_N_null set}
The set $C$ of critical points and the filled set of vertices $F$  are countable. The set $C^\#$ is a null set.
\end{lemma}

\begin{proof}
There are countably many points in the filled set of vertices $F$ since each subset of the form $V_i^\pm$ has finitely many elements. Thus, $C\subset F$ implies that $C$ is countable, too. This, in turn, implies that the $C$-grid, that is, the set $C^\#$, is a null set. Note that from $N\subset C^\#$, we also again obtain that $N$ is a null set.
\end{proof}

\subsection*{Three examples}

\begin{example}\label{example:quad}
The first example is the Quad which has been studied in \cite{Albers_Banhatti_Sadlo_Schwartz_Tabachnikov_2025}. In particular, this is the single table setting; however, we adopt the two table perspective for visualization. 

We use the algorithm from Remark \ref{remark:intuition_for_C} to determine the set of critical points $C$. That is, we pick a vertex $x_0$ and consider all trajectories starting in $x_0$. For the chosen example after applying the symplectic billiard map once the image contains a vertex. Tracing back leads to the point $c\in C$, see Figure \ref{fig:quad_c_bestimmen_0} at the top. We repeat this process with the subsegments until the even trajectory hits a vertex, see Figure \ref{fig:quad_c_bestimmen_0} at the bottom.

\begin{figure}[ht]
    \centering
    \includegraphics[width=.65\linewidth]{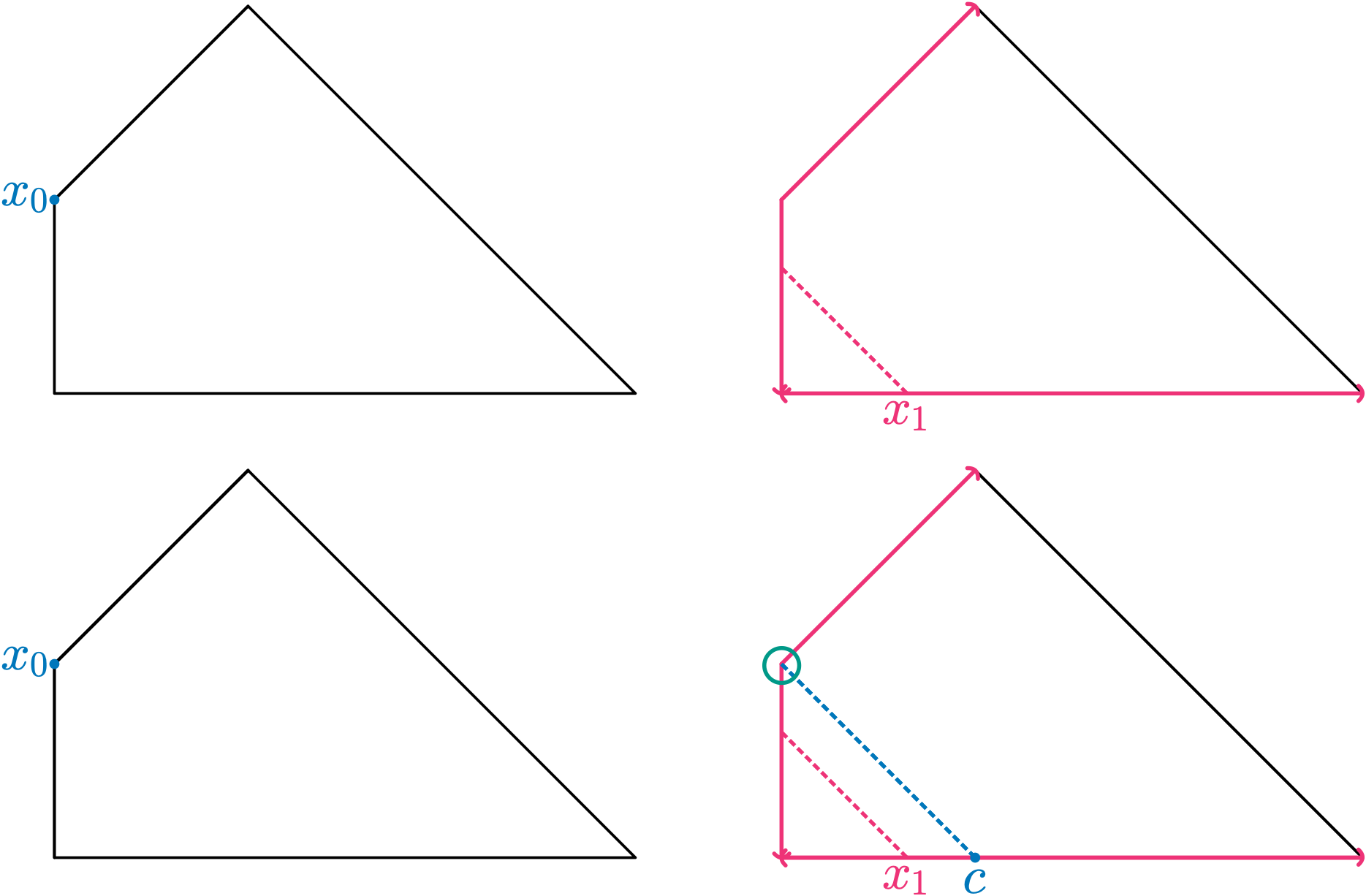}\\[4ex]
    \includegraphics[width=.65\linewidth]{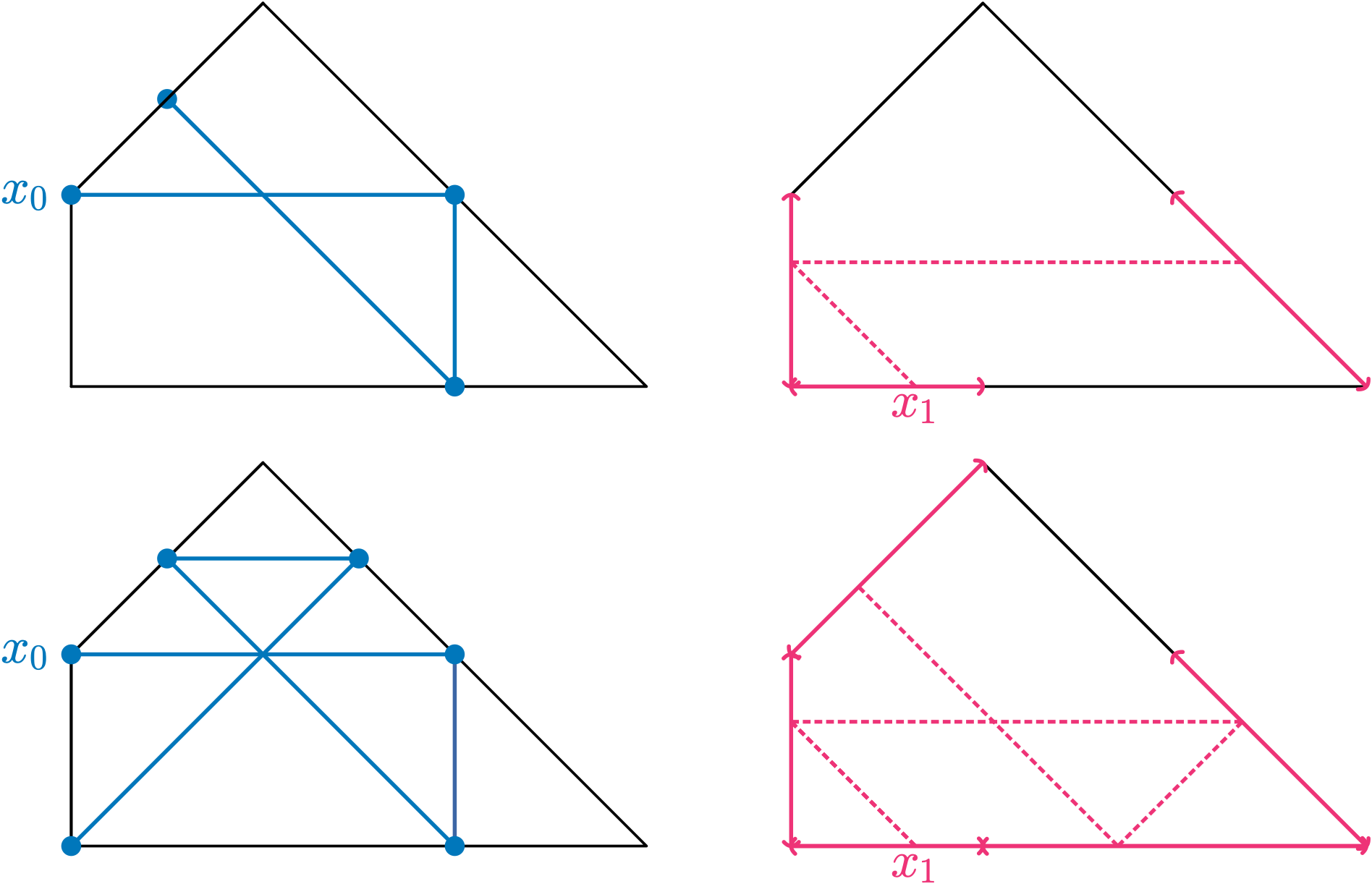}
\caption{Determining critical points of the Quad by starting trajectories in vertices.}
    \label{fig:quad_c_bestimmen_0}
\end{figure}

To determine $C$ we collect all the points on the even trajectory. It is not necessary to also collect the points that split segments as these will be added  when starting at other vertices. 

\begin{figure}[h]
    \centering
    \includegraphics[width=0.55\linewidth]{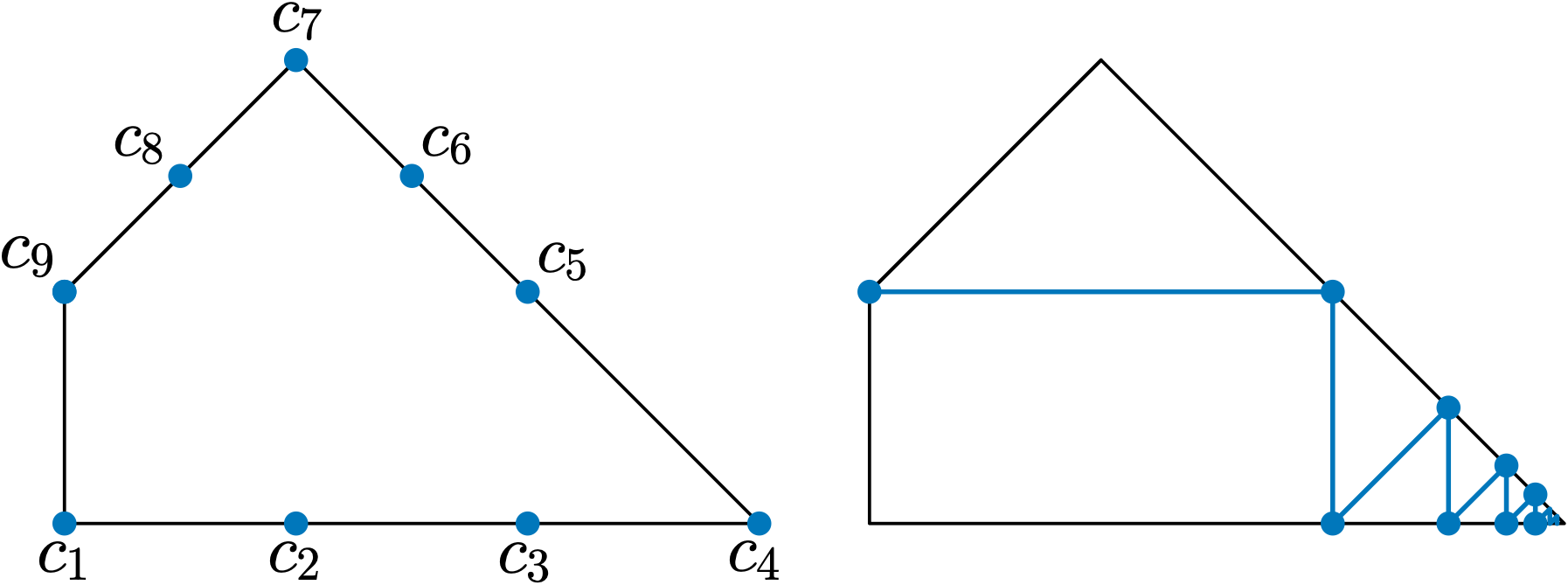}
\caption{\small On the left hand is displayed the set of critical points of the Quad. On the right hand we show how, iteratively following those directions, to generate an infinite subset of the filled set of vertices of the Quad. Note that this is not a trajectory.}
\label{fig:C_grid_quad_sub}
\end{figure}

Repeating this for all edges and vertices yields the set $C$ of critical points, displayed in Figure \ref{fig:C_grid_quad_sub} for the Quad. In contrast, Figure \ref{fig:C_grid_quad_sub} shows that the filled set of vertices $F$ of the Quad is infinite.

Furthermore, the collection of black lines in Figure \ref{fig:C_grid_quad} on the left displays the discontinuity set $N$, the white ``tiles'' are the actual tiles and their union is the phase space $\P$. Finally, the collection of black lines in Figure \ref{fig:C_grid_quad} on the right displays the $C$-grid. In particular, the Quad shows that, in general, the inclusions $C\subset F$ and $N\subset C^\#$ are strict.

\begin{figure}[ht]
    \begin{subfigure}{.38\linewidth}
        \centering
        \includegraphics[width=1.1\linewidth]{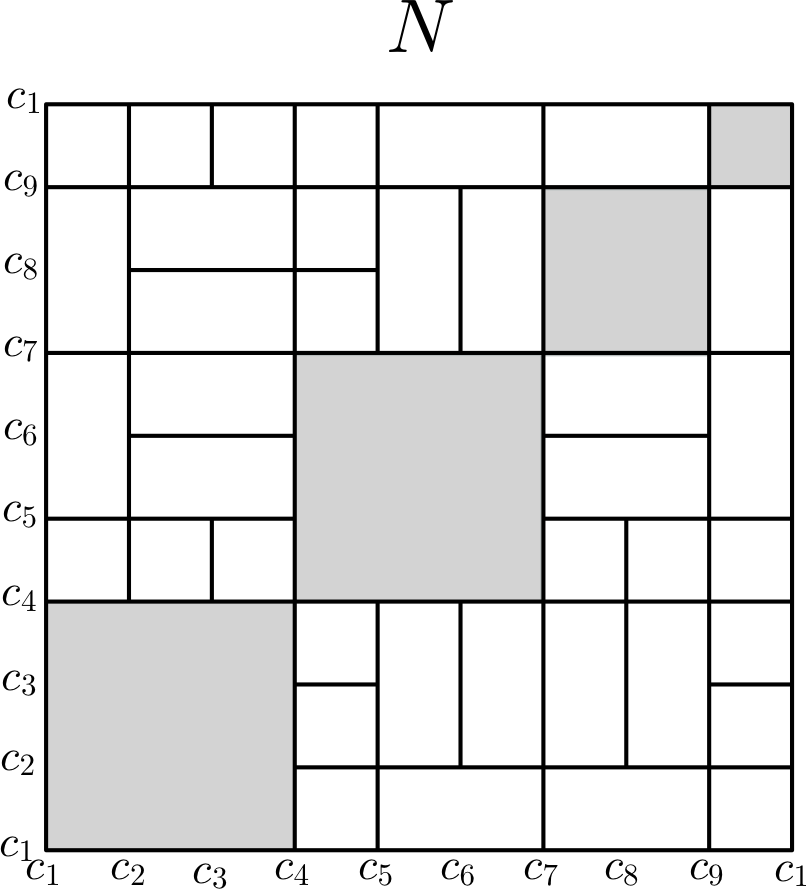}
    \end{subfigure}%
    \hspace{10mm} 
    \begin{subfigure}{.38\linewidth}
        \centering
        \includegraphics[width=1.1\linewidth]{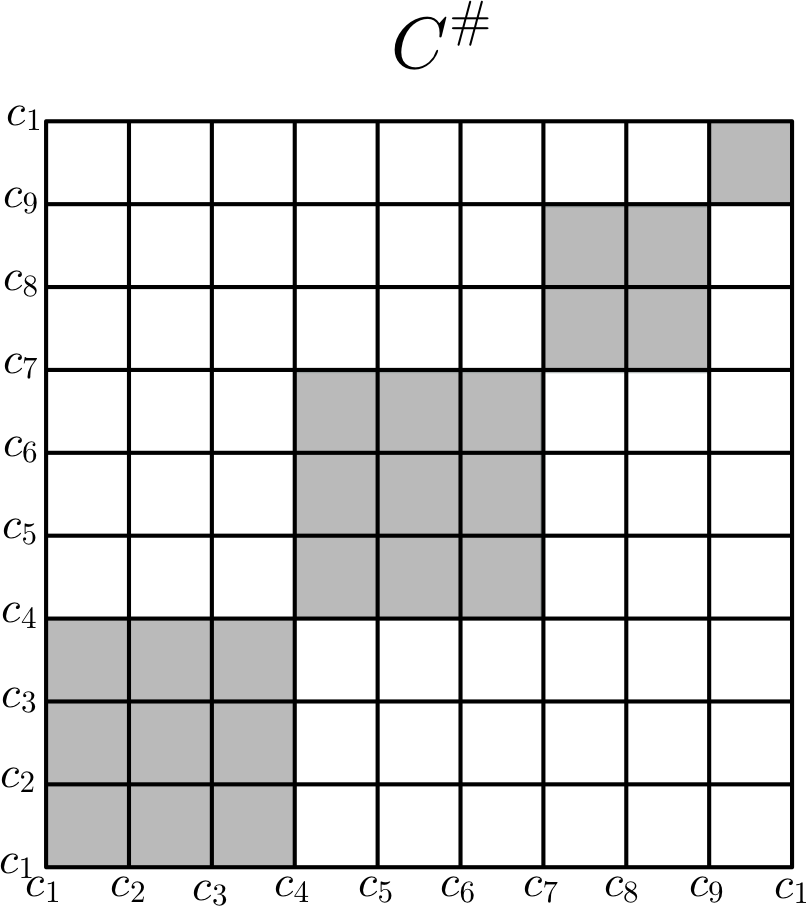}
    \end{subfigure}
    \caption{\small The discontinuity set $N$ (black lines) and the phase space of the Quad (white part), made up of tiles (connected white regions), and the $C$-grid (black lines). Note the strict inclusion $N\subset C^\#$.} 
    \label{fig:C_grid_quad} 
\end{figure}

\end{example}

\begin{example}
As a second example we look at the pair of polygons $(P_-,P_+)$ displayed in Figure \ref{fig:disjoint_C_grid_sub1}. In this particular example, the set of critical points $C\subset P_-\sqcup P_+$ and the filled set of vertices $F$ coincide. They are labeled by $c_1,\dots,c_4$ and $d_1,\dots,d_5$, see Figure \ref{fig:disjoint_C_grid_sub1}. 

\begin{figure}[ht]
    \centering
    \centering
    \includegraphics[width=0.6\linewidth]{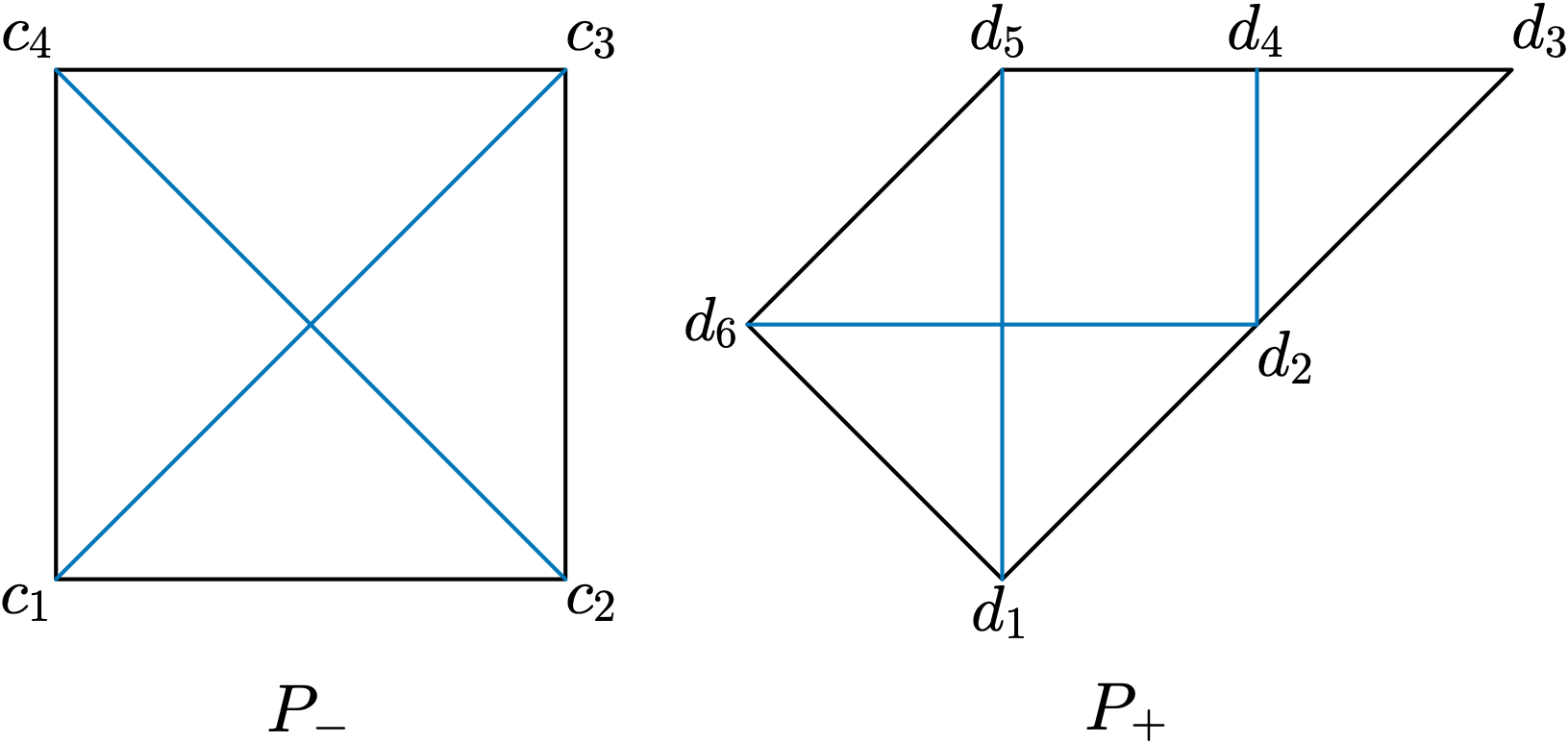}
    \caption{The filled set of vertices and the set of critical points coincide in this example.}
    \label{fig:disjoint_C_grid_sub1}
\end{figure}

The collection of black lines in Figure \ref{fig:disjoint_C_grid_sub3} on the left is the discontinuity set $N$, the white ``tiles'' are again the actual tiles and their union is the phase space $\P$. The grey shaded area is excluded since the corresponding sides are parallel and the symplectic billiard map is not defined. The collection of black lines in Figure \ref{fig:disjoint_C_grid_sub3} on the right displays the $C$-grid.

\begin{figure}[ht]
    \centering
    \includegraphics[width=0.85\linewidth]{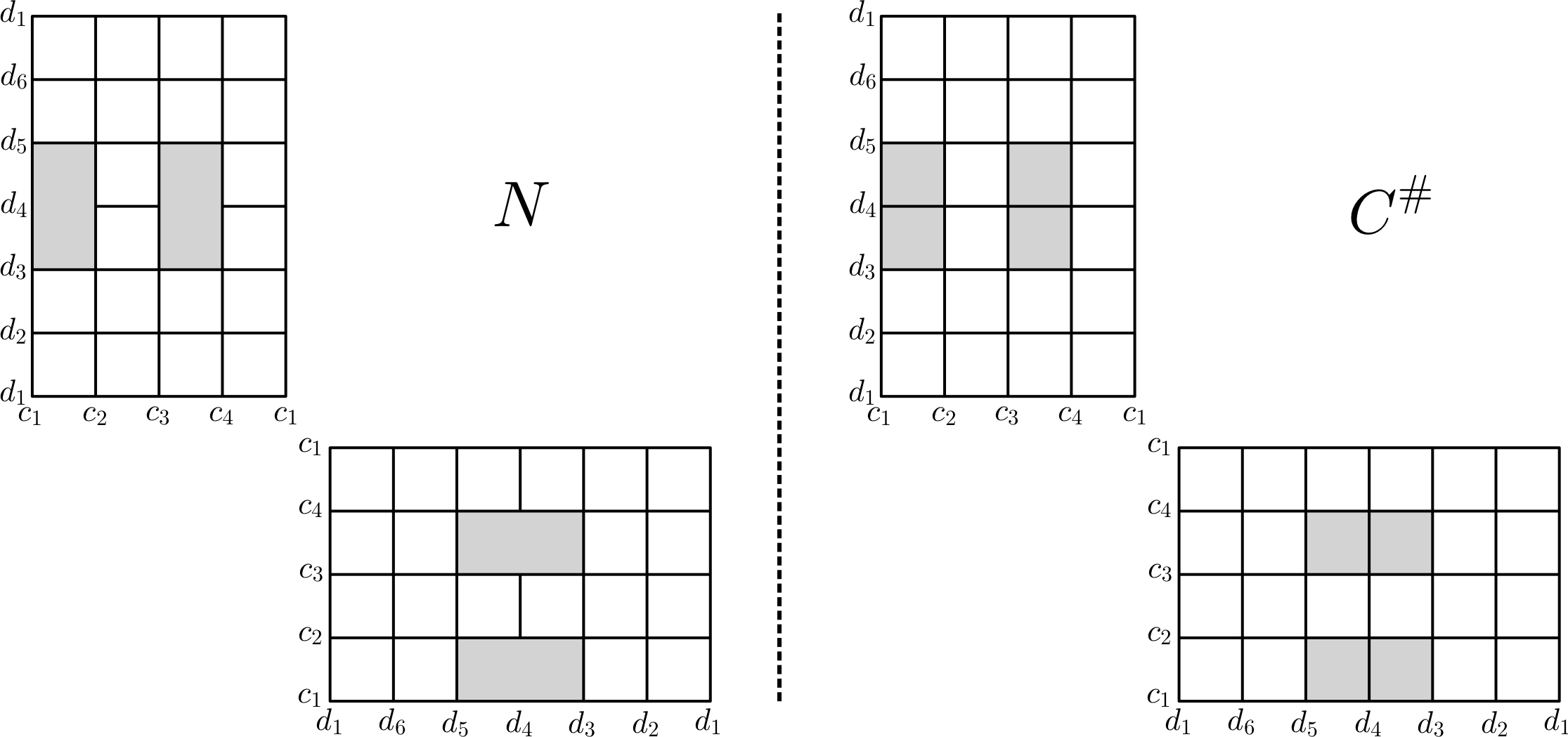}
    \caption{The discontinuity set and the phase space $\P$, made up of tiles, together with the $C$-grid.}
    \label{fig:disjoint_C_grid_sub3}
\end{figure}

\end{example}

\begin{example}\label{example_unbounded_periodic}
The third example consists of a square $P_-$ with vertices $(0,0)$, $(4,0)$, $(4,4)$, $(0,4)$ and a rhombus $P_+$ with  vertices $(6,5)$, $(8,1)$, $(12,-1)$, $(10,3)$.  We apply Remark \ref{remark:intuition_for_C} to determine the set of critical points. To set up our argument, consider the sequences $(a_i)_{i\in\N}$ and $(b_i)_{i\in\Z}$ as in Figure \ref{fig:C_infinite_explanation_sub1}. 

\begin{figure}[ht]
        \includegraphics[width=0.75\linewidth]{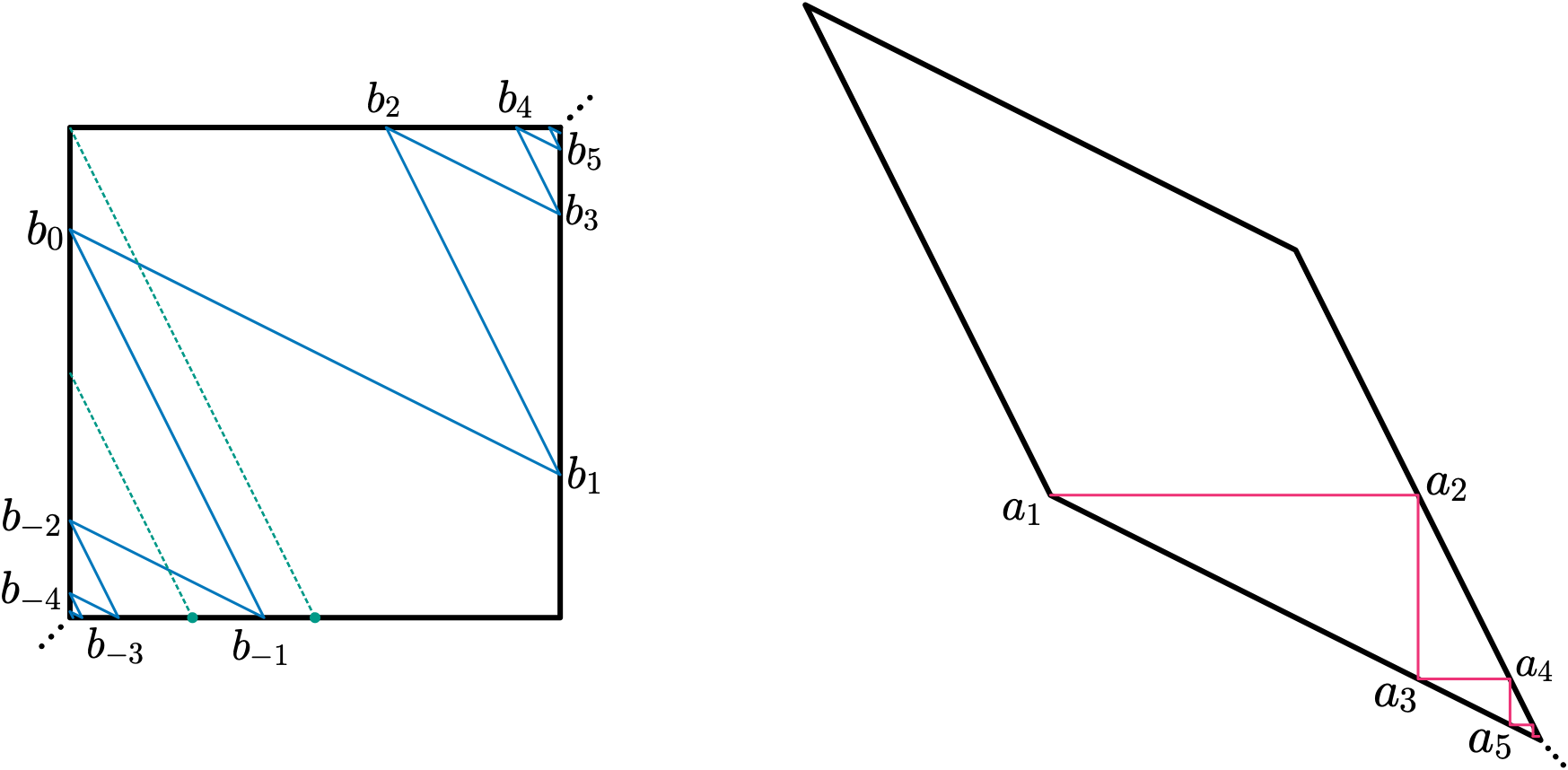}
        \caption{The sequences $(a_i)_{i\in\N}$ and $(b_i)_{i\in\Z}$.}
        \label{fig:C_infinite_explanation_sub1}
\end{figure}

We consider a trajectory starting in the vertex $a_1 = x_0$. The symplectic billiard trajectory starting in $(x_0,x_1) = (a_1, b_{-1})$ is the (in this case finite) sequence 
$$
(x_0,x_1,x_2,x_3,x_4,x_5,x_6,x_7,x_8) = (a_1,b_{-1},a_2,b_0,a_3,b_1,a_2,b_2,a_1),
$$ 
see Figure \ref{fig:C_infinite_explanation_sub2}. 

\begin{figure}[ht]
        \includegraphics[width=0.75\linewidth]{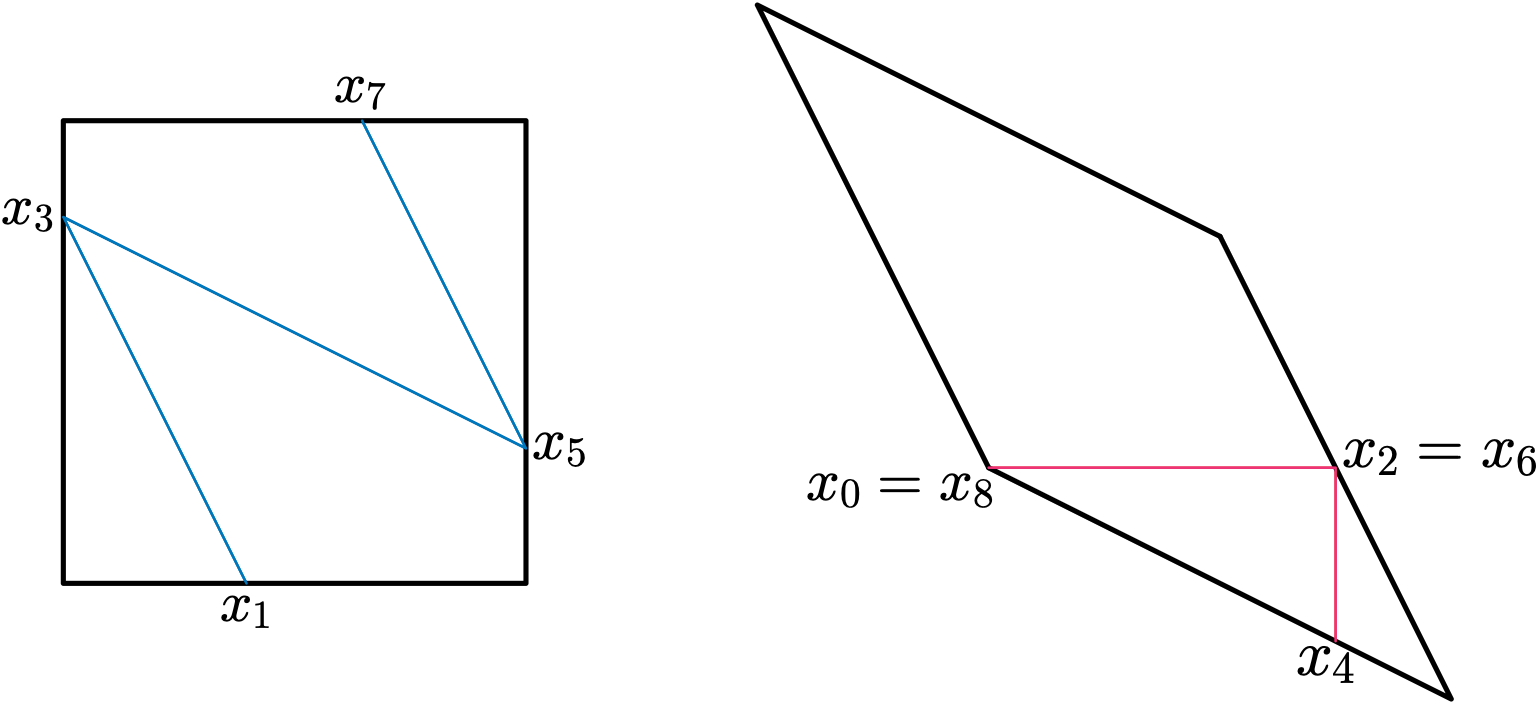}
    \caption{The finite trajectory $(x_n)$.}
            \label{fig:C_infinite_explanation_sub2}
\end{figure}

Note that altering $x_1$ within the green points in Figure \ref{fig:C_infinite_explanation_sub1} at $b_{-1}$ does not change the symbolic trajectory nor the even trajectory. If we instead choose $x_1=b_{-3}$ then the (again finite) trajectory has the following form 
$$
(x_n)_{n=0,\ldots,16} = (a_1,b_{-3},a_2,b_{-2},a_3,b_{-1},a_4,b_0,a_5,b_1,a_4,b_2,a_3,b_3,a_2,b_4,a_1).
$$
The symbolic and even trajectories do not change when altering $x_1$ within a range of $b_{-3}$ that is similar to that with the green points before. Repeating arguments like this yields the set of critical points $C$ as shown in Figure \ref{fig:C_infinite_periodic_sub1}.

\begin{figure}[ht]
        \includegraphics[width=.70\linewidth]{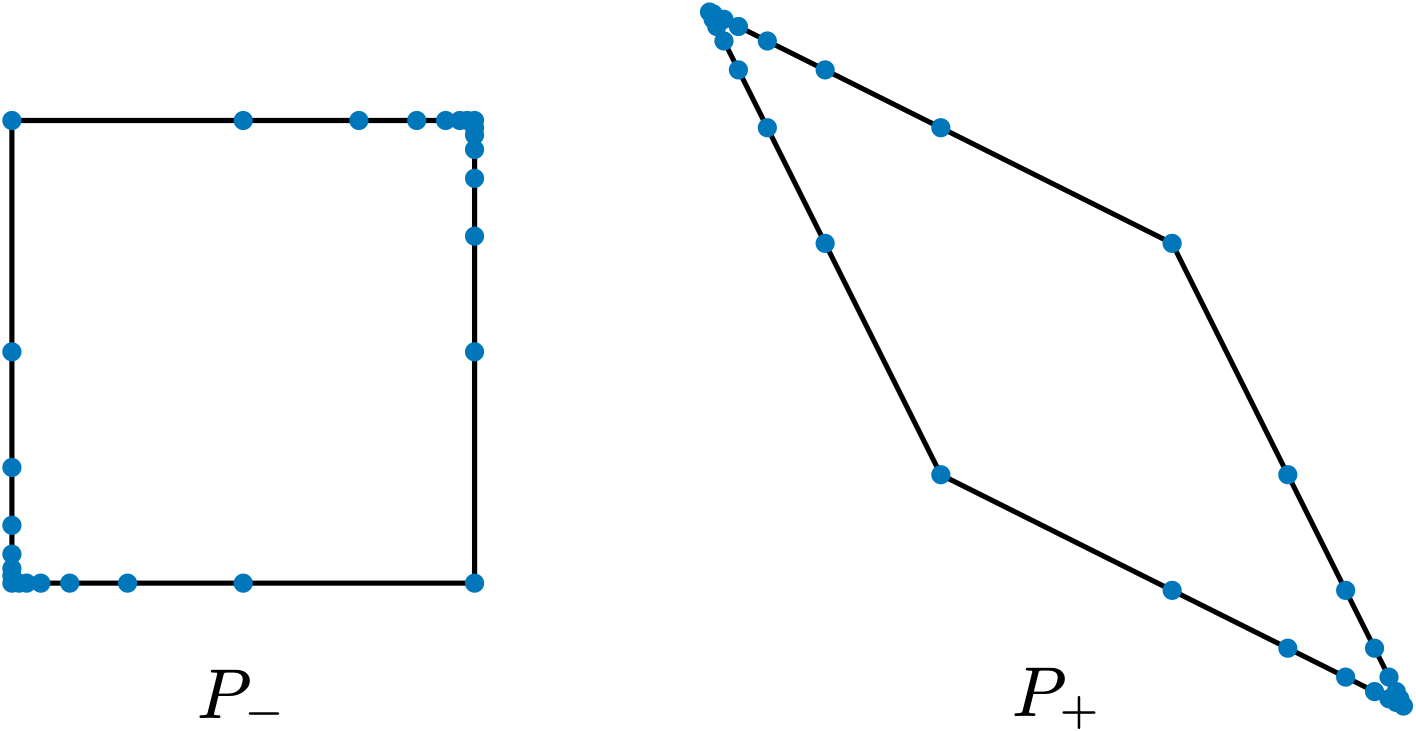}
        \caption{The set of critical points.}
        \label{fig:C_infinite_periodic_sub1}
\end{figure}

The $C$-grid is displayed in Figure \ref{fig:C_infinite_periodic_sub2}.
  
\begin{figure}[ht]
        \includegraphics[width=.6\linewidth]{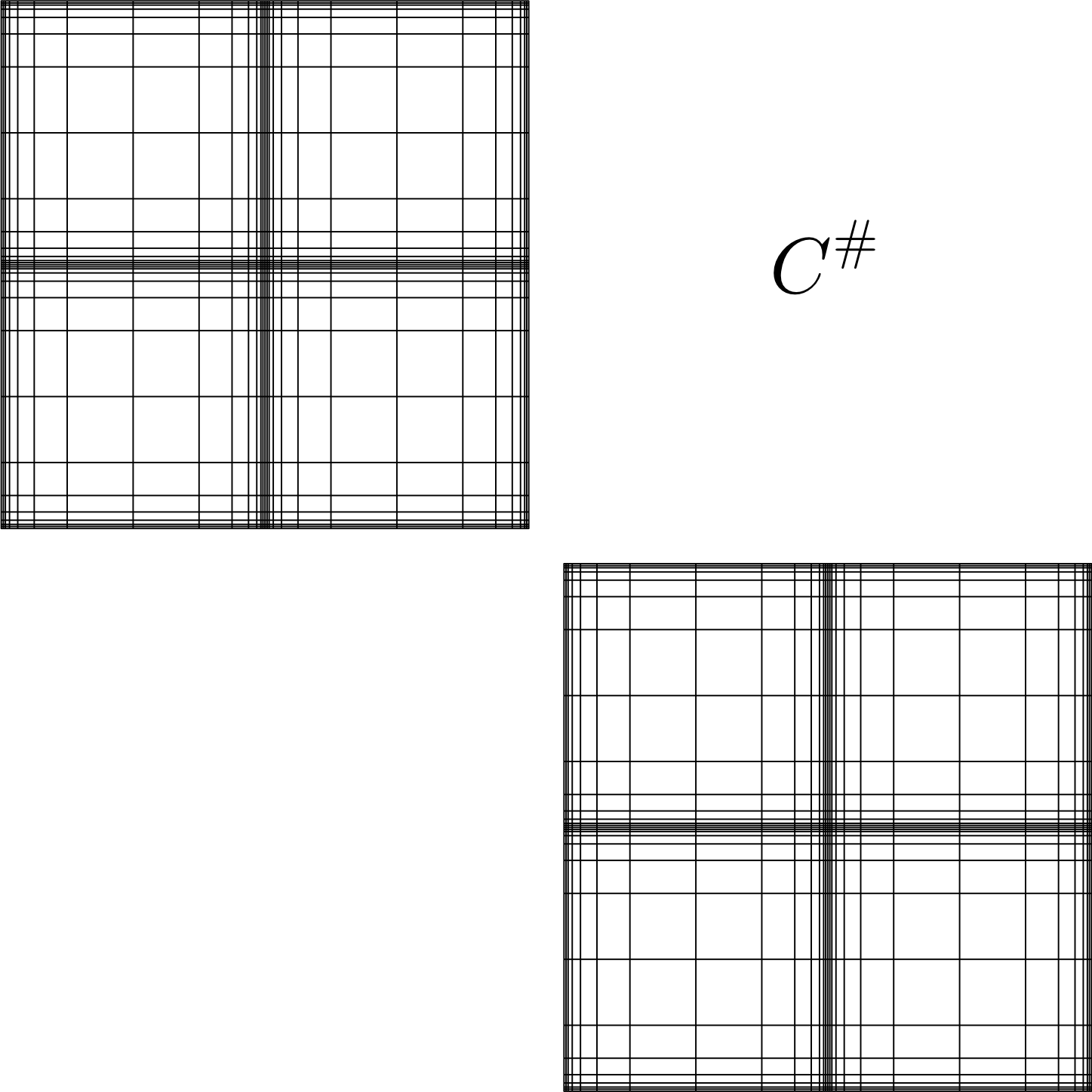}
        \caption{The $C$-grid.}
        \label{fig:C_infinite_periodic_sub2}
\end{figure}     

\end{example}

In the next section we will prove that if $C$ is finite then the symplectic billiard map is uniformly bounded fully periodic, i.e.~every orbit is periodic with a uniform period bound, see Theorem \ref{theorem:periodicity-criterion}.  In the current example the set $C$ is infinite but has the property that every limit point is a vertex. We will see that this implies that the symplectic billiard map is still fully periodic, however without a uniform period bound. In this specific example this easily follows directly from the above considerations. 

\begin{remark}\label{rmk:glossary}
For convenience we collect here all pertinent sets and their interpretation. 
\begin{itemize}   \renewcommand{\labelitemi}{-}  
\item $\P_\mathrm{max}=\{(x,y) \in (P_- \times P_+)\sqcup (P_+\times P_-)\mid  (x,y) \text{ satisfies condition A}\}$. This is the set of pairs on which one iteration of the symplectic billiard map is well-defined. 
\item $\P'=\{(x,y) \in (P_- \setminus V_- \times P_+\setminus V_+) \sqcup (P_+\setminus V_+ \times P_-\setminus V_- ) \mid\det(\nu_x,\nu_y)  \neq  0\}$. This is the set of pairs of non-vertices on which one iteration of the symplectic billiard map is well-defined. Note that for $(x,y)\in\P'$ the point $z$ in $\phi(x,y)=(y,z)$ may be a vertex.
\item $N$ is the discontinuity set, see Definition \ref{definition:iterable_phase_space}. This is the set of points which under forward or backward iteration of the symplectic billiard map hit a vertex.
\item $\P=\P_\mathrm{max}\setminus N=\P'\setminus N$ is the phase space, i.e.~the maximal set on which the symplectic billiard map is a bijection. In particular, we can iterate the symplectic billiard map forward and backward arbitrarily often.
\item $\P_f := \{(x,y) \in \P \mid \det(\nu_x,\nu_y) > 0\}$ is the forward phase space.
\item $\P_b := \{(x,y) \in \P \mid \det(\nu_x,\nu_y) < 0\}$ is the backward phase space.
\item $C$ is the set of critical points, see Definition \ref{def:critical_points}. These are the points in $P_\pm$ which are part of a trajectory hitting a vertex in $P_{\pm}$, i.e.~after an odd number of iterations of the symplectic billiard map $\phi$.
\item $C^\#$ is a grid generated by $C$, see Definition \ref{def:critical_points}. It contains $N$.
\item $F$ is the filled set of vertices. This is the subset of $P_+\sqcup P_-$ obtained as follows. We start from all vertices of $P_\pm$ in all directions tangent to $P_\mp$ and collect the intersections with $P_\pm$ and repeat the process. It contains $C$.
\end{itemize}

\end{remark}


\section{Criteria for periodicity}\label{section:criteria}

We begin by defining the following sets of pairs of polygons $P_-$ and $P_+$ and their corresponding phase space $\P=\P(P_-,P_+)$ equipped with the symplectic billiard map. We consider the sets of nowhere periodic, fully periodic, uniformly bounded fully periodic symplectic billiards, and those with isolated periodic orbits.
\begin{equation}\nonumber
\begin{aligned}
NP &:= \{(P_-,P_+)\mid\text{There is no periodic orbit in }\P(P_-,P_+).\}\\[.5ex]
FP &:= \{(P_-,P_+)\mid\text{Every orbit in } \P(P_-,P_+) \text{ is periodic}.\}\\[.5ex]
BP &:= \{(P_-,P_+)\in FP\mid  \text{There is a uniform upper bound on all periods.}\}\\[.5ex]
IP &:= \{(P_-,P_+)\mid \text{There is an isolated periodic orbit in }\P(P_-,P_+).\}
\end{aligned}
\end{equation}
Clearly, $BP \subset FP$ and the sets $NP$, $FP$ and $IP$ are pairwise disjoint. Each of these sets is non-empty. 
Indeed, in \cite{Albers_Tabachnikov_Introducing_symplectic_billiards} and \cite{Albers_Banhatti_Sadlo_Schwartz_Tabachnikov_2025} there are several examples of convex polygons for which every symplectic billiard orbit on the phase space is periodic and the period is uniformly bounded. One of these is the Quad, see Figure \ref{fig:C_grid_quad}. Hence $BP$ is non-empty. 
In Section \ref{section:kite} we will see that there is a family of convex polygons that have isolated periodic orbits, so $IP$ is non-empty.
In Section \ref{section:necktie} we will see that $NP$ is also non-empty. 
Finally, Example \ref{example_unbounded_periodic} provides an example in $FP\setminus BP$.
Here is an illustration.

\begin{center}
\begin{tikzpicture}
  \begin{scope}[blend group = soft light]
    \fill[red!30!white,xshift=8cm]  circle (1.5);
    \fill[green!30!white,xshift=12cm] circle (1.5);
    \fill[blue!30!white,xshift=4cm]  circle (1.5);
    \fill[white,xshift=4cm] circle (.75);
  \end{scope}
  \node at (4,0)  {BP};
  \node at (5.125,0) {FP};
  \node at (8,0)   {IP};
  \node at (12,0)   {NP};
\end{tikzpicture} 
\end{center}

In this section we give sufficient criteria for (possibly two and non-convex) polygons that guarantee that every orbit on their phase space is periodic with or without period bounds, i.e.~$(P_-,P_+) \in FP$ or $(P_-,P_+) \in BP$.

Here is a diagram of the results from this section where we use the filled set of vertices $F$ and the set $C\subset F$ of critical points for a pair $(P_-,P_+)$ of polygons, see Definitions \ref{def:filled_set_of_vertices}  and \ref{def:critical_points}.

\begin{center}
\begin{tikzpicture}[node distance=2.0cm, 
    every node/.style={fill=white}, align=center]

  \node (three)         []              {Corollary \ref{corollary:three_directions}};
  \node (F_finite)      [base, below of=three]               
  {$F$ finite};
  \node (C_finite)      [base, below of=F_finite]   
  {$C$ finite};
  \node (limit)         [base, below of=C_finite]   
  {limit points of $C$\\ are vertices};
  \node (closed)        [base, below of=limit]     
  {$C$ closed};
  \node (BP)            [base, right of=C_finite, xshift=7.5cm]    
  {$(P_-,P_+) \in BP$};
  \node (FP)            [base, below of=BP]     
  {$(P_-,P_+) \in FP$};
  \node (P_minus_C)     [base, below of=FP]   
  {every orbit on\\ $\P\setminus C^\#$ is periodic};
  \node (almost_e)          [base, below of=P_minus_C]   
  {almost every orbit\\ on $\P_\mathrm{max}$ is periodic};
                                                  
  \draw[->, double equal sign distance]     (F_finite) -- node[text width=1.5cm, xshift=1cm]
                                  {$C\subset F$}
                                  (C_finite);
  \draw[->, double equal sign distance]     (C_finite) -- (limit);
  \draw[->, double equal sign distance]     (limit) -- (closed);
  \draw[->, double equal sign distance]     (BP) -- (FP);
  \draw[->, double equal sign distance]     (FP) -- (P_minus_C);
  \draw[->, double equal sign distance]     (C_finite) -- node [text width=3cm]
                                    {Theorem \ref{theorem:periodicity-criterion} }
                                    (BP);
  \draw[->, double equal sign distance]     (F_finite)  -| node [xshift=-4.9cm, text width=3cm]
                                    {Theorem \ref{theorem:F_finite}}
                                    (BP);
  \draw[->, double equal sign distance]     (three) -- 
                                    (F_finite);
  \draw[->, double equal sign distance]     (P_minus_C) -- node[text width=2.5cm, xshift =-1.5cm]
                                    {Lemma \ref{lemma:F_C_countable_N_null set}}
                                    (almost_e);
  \draw[->, double equal sign distance]     (limit) -- node[text width = 3cm]{Theorem \ref{theorem:periodicity-criterion} }(FP);                                   
  \draw[->, double equal sign distance]     (closed) -- node[text width = 3cm]{Theorem \ref{theorem:periodicity-criterion} }(P_minus_C);
  \draw[->, double equal sign distance] (C_finite) to[bend left=280] node[midway,xshift=2.75cm]{Example \ref{example:quad} }(F_finite);
  \draw[->, double equal sign distance] (limit) to[bend left=280] node[midway,xshift=2.75cm]{Example \ref{example_unbounded_periodic} } (C_finite);
   \draw[->, double equal sign distance] (FP) to[bend left=80] (BP);
    \draw(1.8,-3.3) -- (2.5,-2.65);
    \draw(1.8,-5.3) -- (2.5,-4.65);
    \draw(7.0,-5.3) -- (7.7,-4.65);
    \draw[<-] (2.6,-3.0) -- (3.4,-3.0);
     \draw[<-] (2.6,-4.98) -- (3.4,-4.98);
     \draw[->] (6.3,-4.98) -- (7.1,-4.98);
  \end{tikzpicture}
\end{center}

Before stating the first theorem we recall the notation 
\begin{equation}\nonumber
F_\pm=F\cap P_\pm.
\end{equation}

\begin{theorem}\label{theorem:F_finite}
If the filled set of vertices $F$ is finite then every symplectic billiard orbit on the phase space $\P$ is periodic. All periods are bounded by $4|F_-||F_+|$.
\end{theorem}

\begin{remark}
The theorem above is actually an immediate corollary of Theorem \ref{theorem:periodicity-criterion} below because the set $C$ of critical points is contained in the filled set of vertices $F$. However, Theorem \ref{theorem:F_finite} is somewhat simpler to prove but gives a weaker period bound. To briefly outline the idea of the proof, recall the sets $V_i^\pm$ from Definition \ref{def:filled_set_of_vertices}. The connecting lines from points in $V_i^\pm$ to the respective points in $V_{i+1}^\pm$ act as ``guide rails'' for the even and odd trajectory. If  $F$ is finite this restricts a trajectory to only visit finitely many points, see Figure \ref{fig:f_trapez}. In particular, every trajectory is periodic. A counting argument gives the upper bound. A full proof can be found in an earlier version \cite{albers2024symplecticbilliardspairspolygons} of this article.

\begin{figure}[ht]
    \centering
    \includegraphics[width=.8\linewidth]{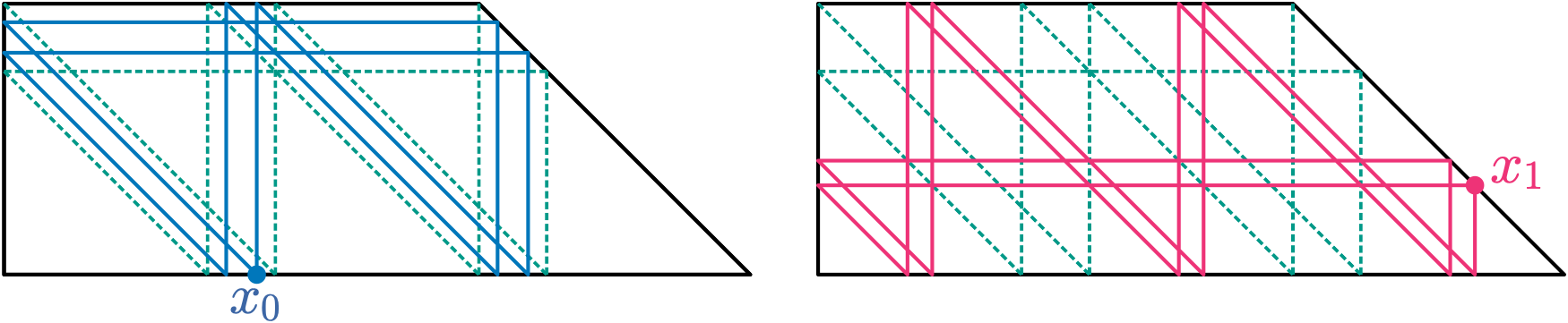}
    \caption{The connecting lines (dashed) from points in $V_i^\pm$ to the respective points in $V_{i+1}^\pm$ act as ``guide rails'' for the billiard trajectory (blue and red) starting at $(x_0,x_1)$.}
    \label{fig:f_trapez}
\end{figure}

In the example of a triangle almost every trajectory visits every edge twice whereas in the case of a square every trajectory visits  edges only once, compare the two cases in Figure \ref{fig:overkill_triangle_quad}. We point out that these examples are in the single table setting; however, in Figure \ref{fig:overkill_triangle_quad} we show even and odd trajectory separately, as usual. These examples show that in certain cases the period bound in Theorem \ref{theorem:periodicity-criterion} may be improved.

\begin{figure}[ht]
        \centering
        \includegraphics[width=.65\linewidth]{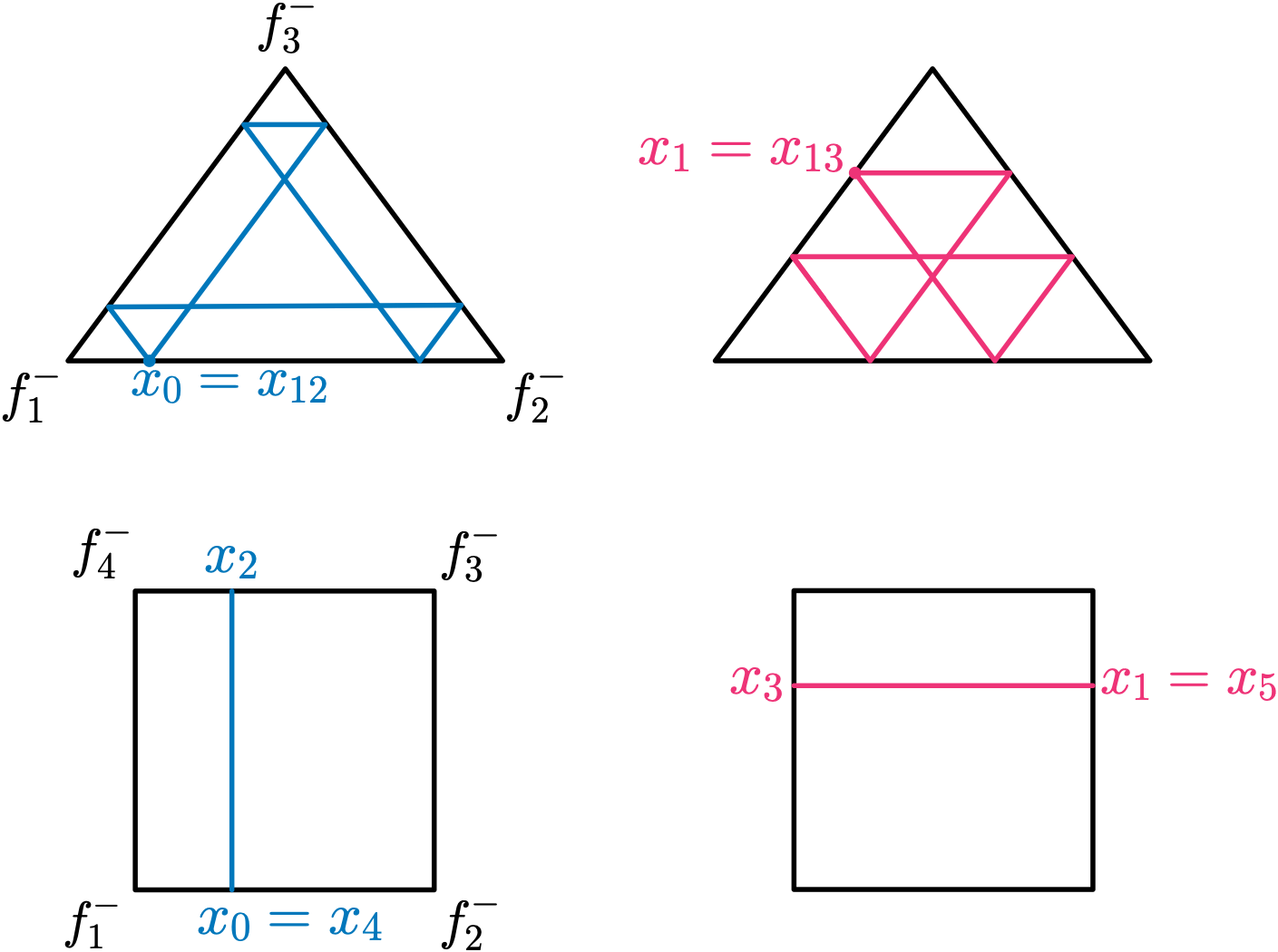}
        \caption{Periodic orbits in a  triangle resp.~in a square.}
        \label{fig:overkill_triangle_quad}
\end{figure}

\end{remark}

Theorem \ref{theorem:F_finite} has the following direct consequence.

\begin{corollary}\label{corollary:three_directions}
Let $(P_-,P_+)$ be a pair of polygons whose vertices lie on the integer lattice, i.e.~ $V_-,V_+ \subset \Z^2$. Assume in addition that $P_-$ and $P_+$ together have at most three distinct tangent directions. Then every symplectic billiard orbit on the phase space is periodic and there is a global bound on the period.
\end{corollary}

\begin{proof}
As pointed out in Remark \ref{remark:variational_principle_and_invariant_area_form}, symplectic billiards commutes with applying affine transformations to both polygons. If there are three distinct tangent directions we may assume without loss of generality that these directions are 
$$
\begin{pmatrix}
1\\0
\end{pmatrix}, 
\begin{pmatrix}
0\\1
\end{pmatrix}
\text{ and }
 \begin{pmatrix}
1\\1
\end{pmatrix}.
$$
It is clear from Definition \ref{def:filled_set_of_vertices} that $F_\pm$ must now be a subset of the integer lattice $\Z^2$. The case of two tangent directions follows exactly the same way.
\end{proof}

The corollary allows us to construct a plethora of examples of polygonal symplectic billiards that carry only periodic symplectic billiard trajectories
\begin{figure}[ht]
    \centering
    \includegraphics[width = .45\linewidth]{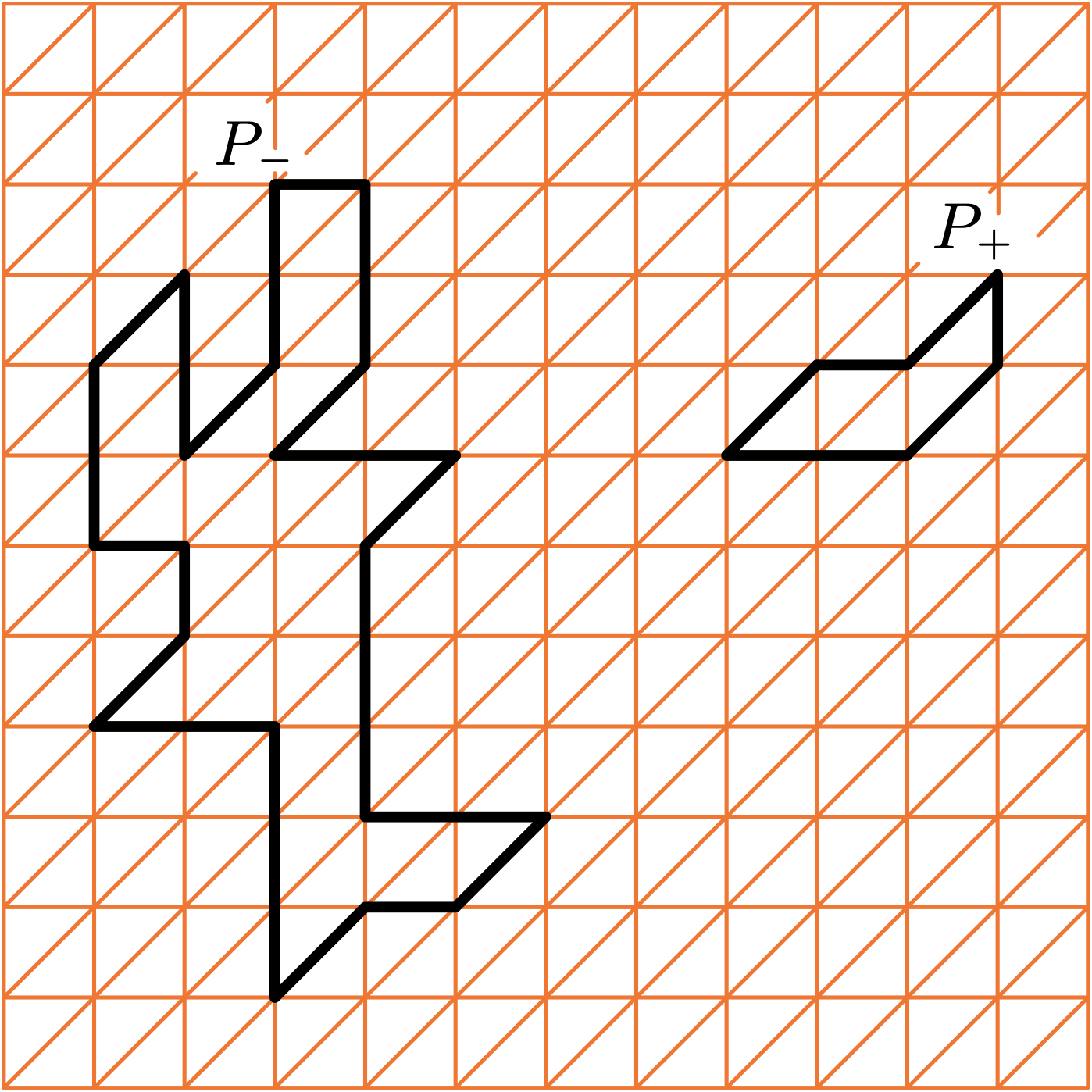}
    \caption{An example for Corollary \ref{corollary:three_directions}}
    \label{fig:three_directions_construction}
\end{figure}
with a uniform bound on the period. For example, any pair of (possibly non-convex) polygons ``inscribed'' into
$$
\left[\Z^2 + \R \begin{pmatrix} 1\\0\end{pmatrix}\right] \cup \left[\Z^2 + \R \begin{pmatrix}0\\1\end{pmatrix}\right] \cup \left[\Z^2 + \R\begin{pmatrix}1\\1\end{pmatrix}\right],
$$ 
see Figure \ref{fig:three_directions_construction}, is fully periodic.

More generally, Theorem \ref{theorem:F_finite} allows us to find examples in BP which are not restricted to the integer lattice and/or three directions. E.g.~in Figures \ref{fig:f_trapez} and \ref{fig:baguette_diamond_F} - \ref{fig:double_egg} we have a couple more examples, in both the single and the two table settings, each having a finite filled set of vertices $F$. In dashed green resp.~orange we show the connecting lines from points in $V_i^\pm$ to the respective points in $V_{i+1}^\pm$.

\begin{figure}[ht]
\centering
\begin{minipage}{\textwidth}
  \centering
  \includegraphics[width=.6\linewidth]{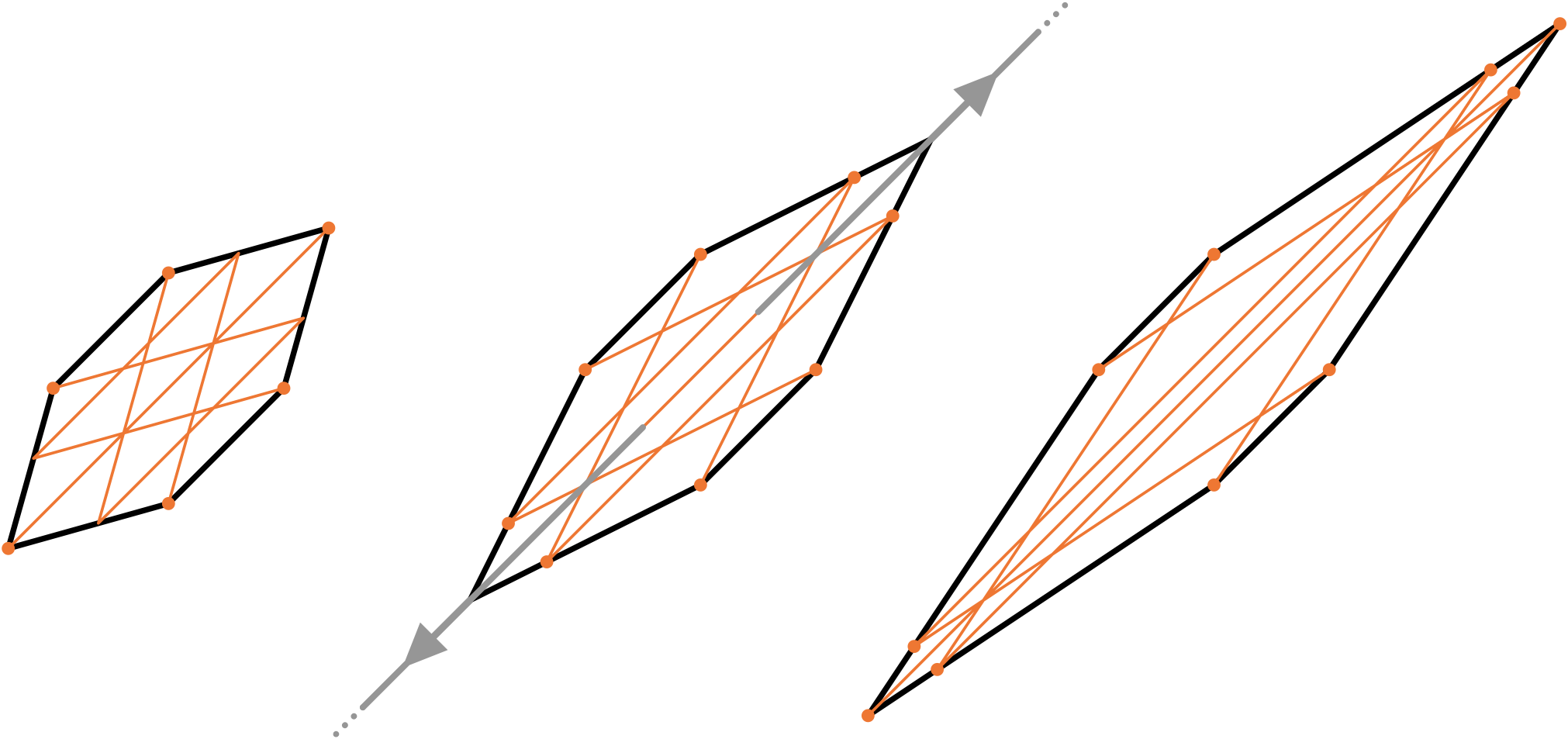}
  \caption{Baguette diamonds. Moving both tips simultaneously along the grey line gives a 1-parameter family of polygons all having a finite filled set of vertices. All of them are in BP.}
    \label{fig:baguette_diamond_F}
\end{minipage}%
\end{figure}

\begin{figure}[ht]
  \centering
  \includegraphics[width=.35\linewidth]{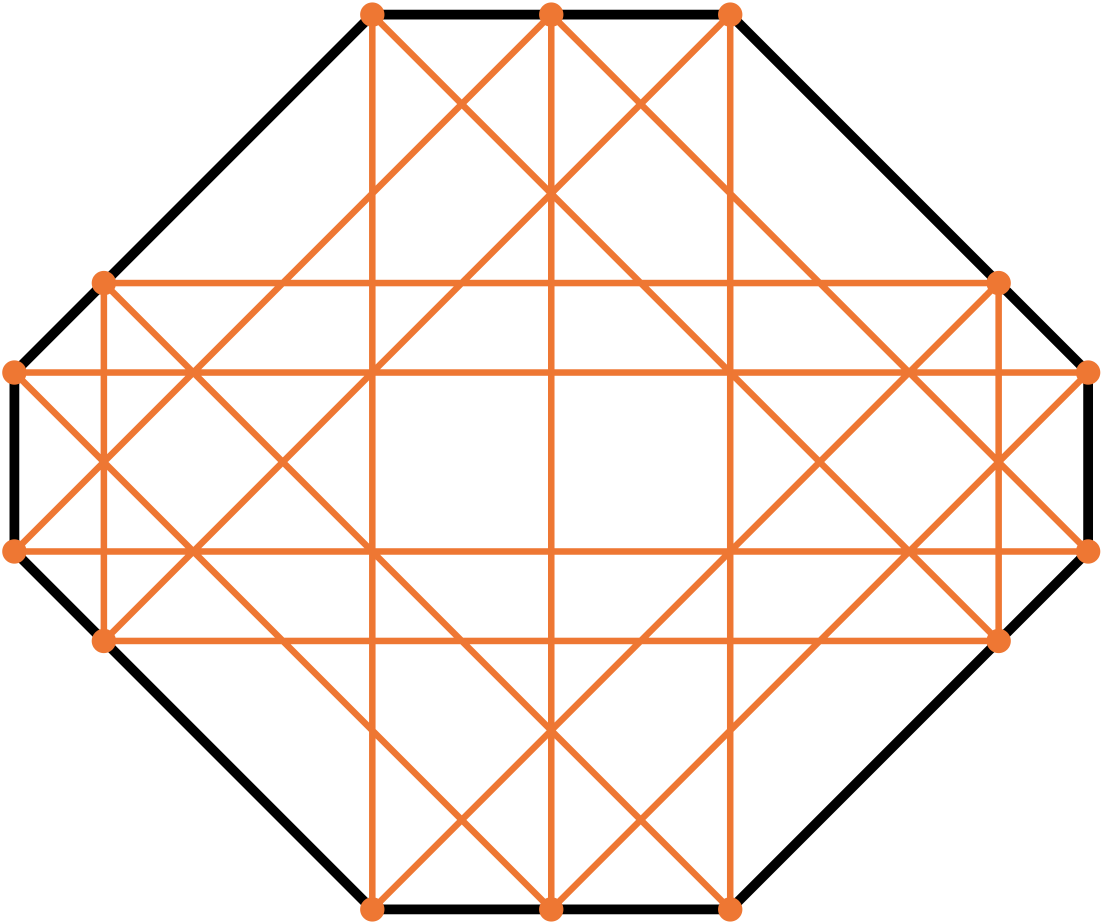}
  \caption{This octagon is in BP.}
    \label{fig:stop_sign_distorted}
\end{figure}

\begin{figure}[ht]
  \centering
  \includegraphics[width=.4\linewidth]{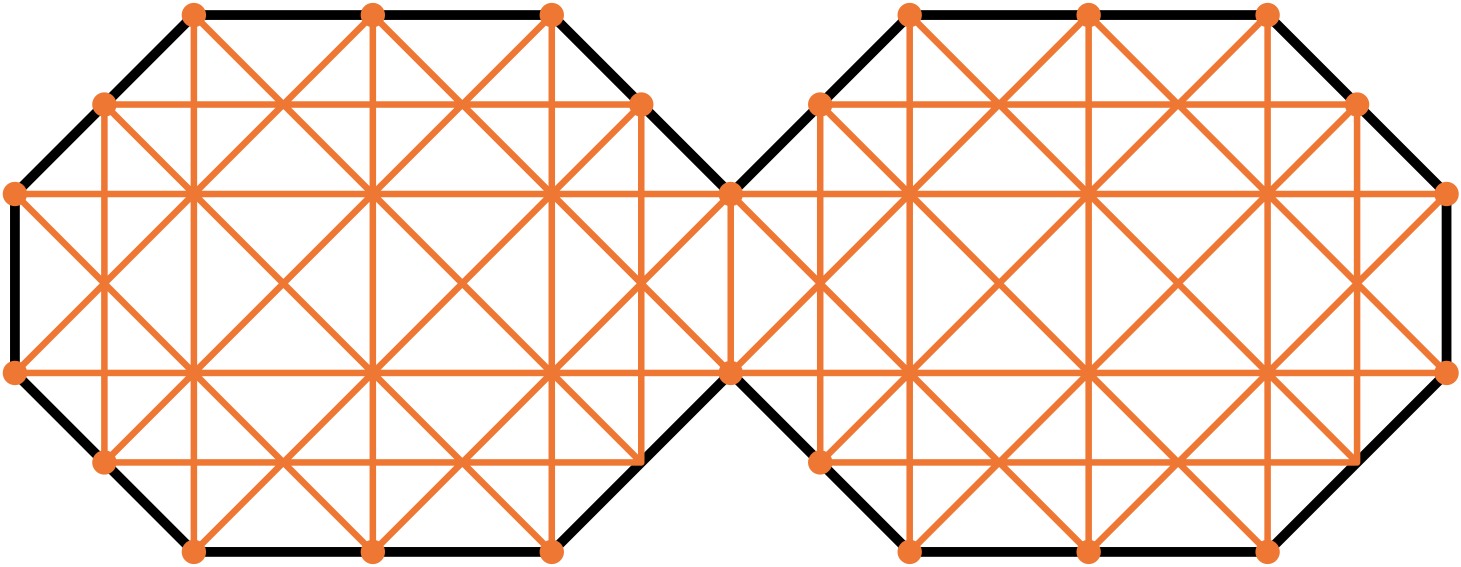}
  \caption{This polygon is in BP.}
    \label{fig:double_egg}
\end{figure}

While Theorem \ref{theorem:F_finite} allows us to easily create examples, it is still too limited. In fact, the polygons ``the Quad'' and ``the Penthouse'' from \cite{Albers_Banhatti_Sadlo_Schwartz_Tabachnikov_2025} have infinite filled sets of vertices, see Figure \ref{fig:C_grid_quad} for the Quad, but still every symplectic billiard orbit is periodic. However, for both examples the set of critical points $C$ is finite. See Definition \ref{def:critical_points} for the set $C$ of critical points and also the set $C^\#$. These examples also motivate the definition of the set $C$. Recall that in $C$ we only take directions of actual symplectic billiard trajectories into account, whereas in $F$ we simply go, starting from vertices, repeatedly in all possibly allowed directions regardless of whether they follow symplectic billiard trajectories or not. It turns out the requirement of $C$ being finite is still too strict, see Example \ref{example_unbounded_periodic}, but does guarantee a uniform period bound. 

\begin{theorem}[Periodicity criterion] \label{theorem:periodicity-criterion}
$ $
\begin{enumerate}  
\renewcommand{\theenumi}{\Roman{enumi}}\renewcommand{\labelenumi}{(\theenumi)}
\itemsep=1ex
\item If every limit point of the set of critical points $C$ is a vertex of $P_-\sqcup P_+$ then every symplectic billiard orbit on the phase space $\P$ is periodic.
\item If $C$ is closed then every orbit on $\P\setminus C^\#$ is periodic. Then, in particular, almost every orbit on $\P_{\max}$ is periodic. 
\item If $C$ is finite then every symplectic billiard orbit on the phase space $\P$ is periodic and the maximal period is bounded from above by $4|C\cap P_-||C\cap P_+|$. In the single table setting (counting only the critical points of $P = P_-=P_+$), we have the bound $2(|C|^2 - |C|)$. 
\end{enumerate}
\end{theorem}

\begin{proof}
We begin by noting that if $C$ is closed then for any point $z\in P_\pm$ that is not a limit point of $C$, we have a unique closest point $\underline z \in C$ to the left of $z$, according to the map $S^1 \to P_\pm$. The point $\underline z$ has positive distance to $z$. We point out that there is no vertex strictly in between $\underline z$ and $z$ since vertices are contained in $C$.

\medskip

\noindent \underline{Claim}: Let $(x_0,x_1) \in \P$ where neither $x_0$ nor $x_1$ is a limit point of $C$. If $C$ is closed then the symplectic billiard orbit of $(x_0,x_1)$ is periodic.

\smallskip

We will show that the tile containing $(x_0,x_1)$ has positive phase space area and then apply Lemma \ref{lemma:positive_tiles_are_periodic}.
Since $x_0$ and $x_1$ are not limit points of $C$  and $C$ is closed we have the corresponding closest points to the left, $\underline x_0, \underline x_1 \in C$, as defined above. Moreover, for every $(y_0,y_1) \in \underline x_0 x_0 \times \underline x_1 x_1$ the point $y_0$ lies in the same edge of $P_\pm$ as $x_0$ and, likewise, $y_1$ lies in the same edge of $P_\mp$ as $x_1$. We conclude  $\det(\nu_{y_0},\nu_{y_1}) = \det(\nu_{x_0},\nu_{x_1}) \neq 0$ and hence $(y_0,y_1) \in \P'$. Moreover, since $\underline x_0$ and $\underline x_1$ are the respective closest points in $C$, we have $(y_0,y_1) \notin C^\#$ and, in particular, $(y_0,y_1) \notin N\subset C^\#$, see Definition \ref{def:critical_points} for the set $C^\#$. Therefore, $(y_0,y_1)$ lies in the phase space $\P = \P'\setminus N$ and we conclude that $\underline x_0 x_0 \times \underline x_1 x_1 \subset \P$, i.e.~is a phase rectangle. By Lemma \ref{lemma:connected_is_contained_in_tile} each phase rectangle is contained in a tile. Since $\underline x_0$ resp.~$\underline x_1$ have positive distance to $x_0$ resp.~$x_1$ the corresponding tile has non-zero phase space area. Lemma \ref{lemma:positive_tiles_are_periodic} then implies that $(x_0,x_1)$ has a periodic symplectic billiard orbit, proving the Claim.

\medskip

We are now in the position to prove Theorem \ref{theorem:periodicity-criterion}. Assertion (I) follows from the Claim  since for any $(x_0,x_1) \in \P$ the points $x_0$ and $x_1$ are not vertices and thus, by assumption, not limit points of $C$.

Assertion (II) also follows from the Claim together with the fact that $C$ is closed if and only if it contains all of its limit points. Moreover, according to Lemma \ref{lemma:F_C_countable_N_null set} the set $C^\#$ is a null set.

For Assertion (III) we observe that the number of arc-wise connected components in $\big[(P_-\times P_+)\sqcup (P_+\times P_-)\big]\setminus C^\#$ is bounded by $2|C\cap P_-||C\cap P_+|<\infty$, see Figure \ref{fig:disjoint_C_grid_sub3} for an illustration. Using again that $C$ is finite, we conclude that every arc-wise connected component in $\big[(P_-\times P_+)\sqcup (P_+\times P_-)\big]\setminus C^\#$ is an open set. The phase space $\P$ has fewer (and potentially bigger) arc-wise connected components than $\big[(P_-\times P_+)\sqcup (P_+\times P_-)\big]\setminus C^\#$, see again Figure \ref{fig:disjoint_C_grid_sub3}. In particular, every arc-wise connected component of $\P$ has positive area. We recall from Corollary \ref{cor:connected_components_are_tiles} that arc-wise connected components of $\P$ are tiles.

By Lemma \ref{lemma:positive_tiles_are_periodic} the symplectic billiard map $\phi$ induces a return map on every tile of positive area. This return map has at most order 4. 
From this we obtain that every point in $\P$ is periodic with period bounded by $8|C\cap P_-||C\cap P_+|$. Taking into account that the phase space consists of the forward and backward phase space, see Definition \ref{definition:iterable_phase_space} and Lemma \ref{lemma:billard_map_maps_phase_space_into_itself}, improves this bound to $4|C\cap P_-||C\cap P_+|$. For the single table setting the analogous argument gives the bound $2|C|^2$ which can be improved to $2(|C|^2 - |C|)$ by observing that points on the diagonal are not in the phase space since the billiard map is not defined for points on parallel edges.
\end{proof}

\begin{remark}\label{rmk:after_Thm_periodicity}
In Example \ref{example:quad} we determined the set $C$ of critical points of the Quad which is, in particular, finite. Therefore, part (III) of Theorem \ref{theorem:periodicity-criterion} gives an alternative proof of the fact that the Quad is an element of $BP$, see \cite{Albers_Banhatti_Sadlo_Schwartz_Tabachnikov_2025}.    

Example \ref{example_unbounded_periodic}, a square and a rhombus, shows that $FP\setminus BP$ is non-empty. Note that in this example all limit points of $C$ are vertices. By part (I) of Theorem \ref{theorem:periodicity-criterion} this is an element of $FP$, i.e.~every orbit on the phase space is periodic. However, there is no uniform period bound, in fact, there are periodic orbits moving arbitrarily deep into the far corner of the rhombus, see Figure \ref{fig:C_infinite_explanation_sub1}. 

We also give an example of a single table element in $FP\setminus BP$, shown in Figure \ref{fig:star_single}. However, we do not give a full proof here. The regular pentagram has a filled set of vertices $F$ with limit points precisely in the outer vertices. Thus every limit point of $C$ is a vertex and Theorem \ref{theorem:periodicity-criterion} applies.
\end{remark}

In particular, we exhibited the following theorem.

\begin{theorem}\label{theorem:FP_but_not_BP}
There is a pair of polygons $P_-$ and $P_+$ for which every symplectic billiard orbit is periodic; however, their periods are not uniformly bounded. 
\end{theorem}

\begin{figure}[ht]
    \includegraphics[width=.4\linewidth]{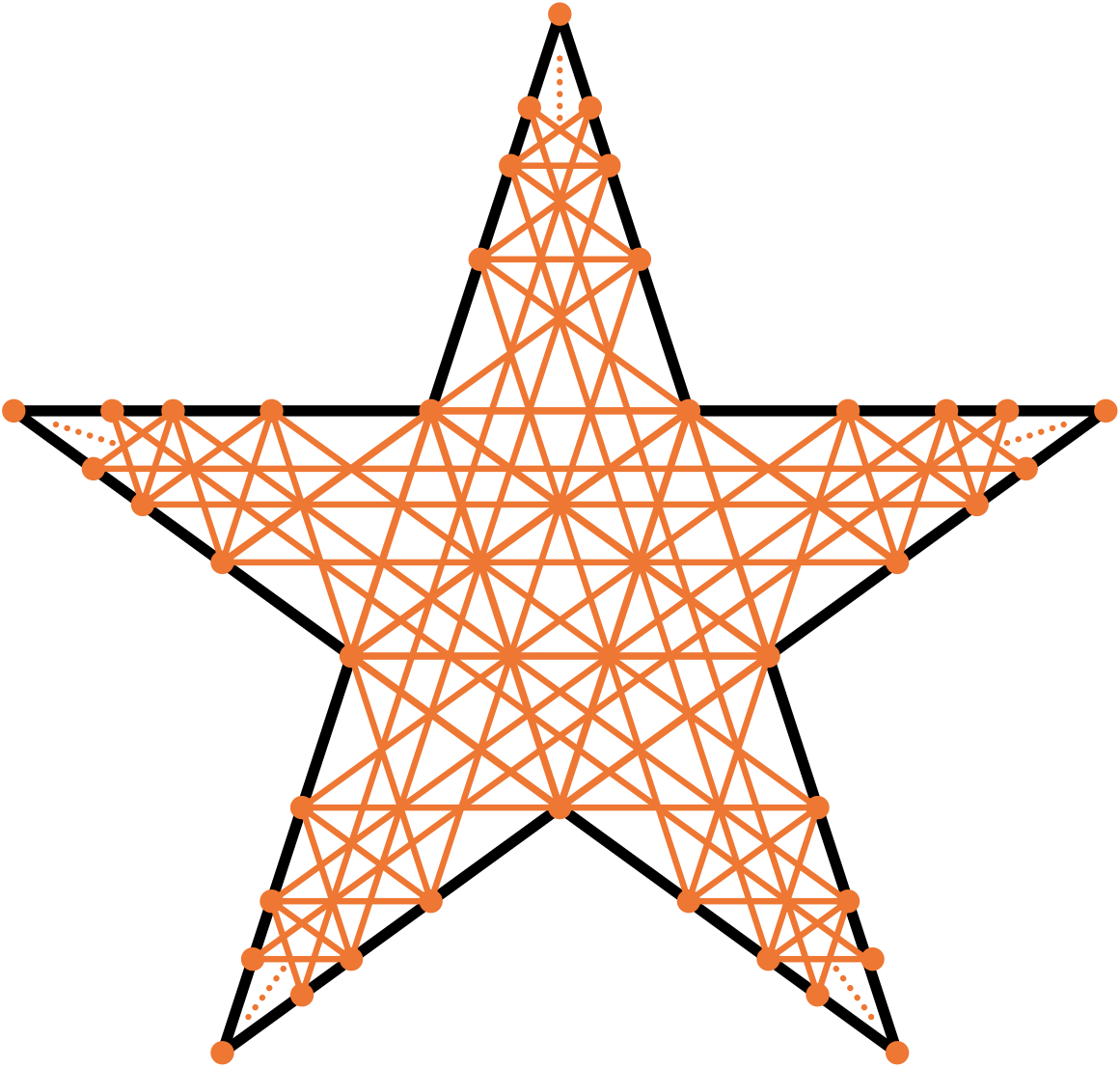} 
    \caption{The regular pentagram is in $FP\setminus BP$. The set $F$ has accumulation points precisely in the outer vertices of the pentagram.}
    \label{fig:star_single}
\end{figure}


\section{ne-quadrilaterals have isolated periodic orbits}\label{section:kite}

In this section we prove that every ne-quadrilateral (short for north-east quadrilateral) carries an isolated periodic symplectic billiard orbit of period 6.
By a ne-quadrilateral we mean a convex polygon with vertices $(0,1)$, $(0,0)$, $(1,0)$, $(X,Y)$ with $X > 1$, $Y > 1$ and $|X-Y| < 1$, see Figure \ref{fig:crooked_kite}. 

\begin{figure}[ht]
    \includegraphics[width=.8\linewidth]{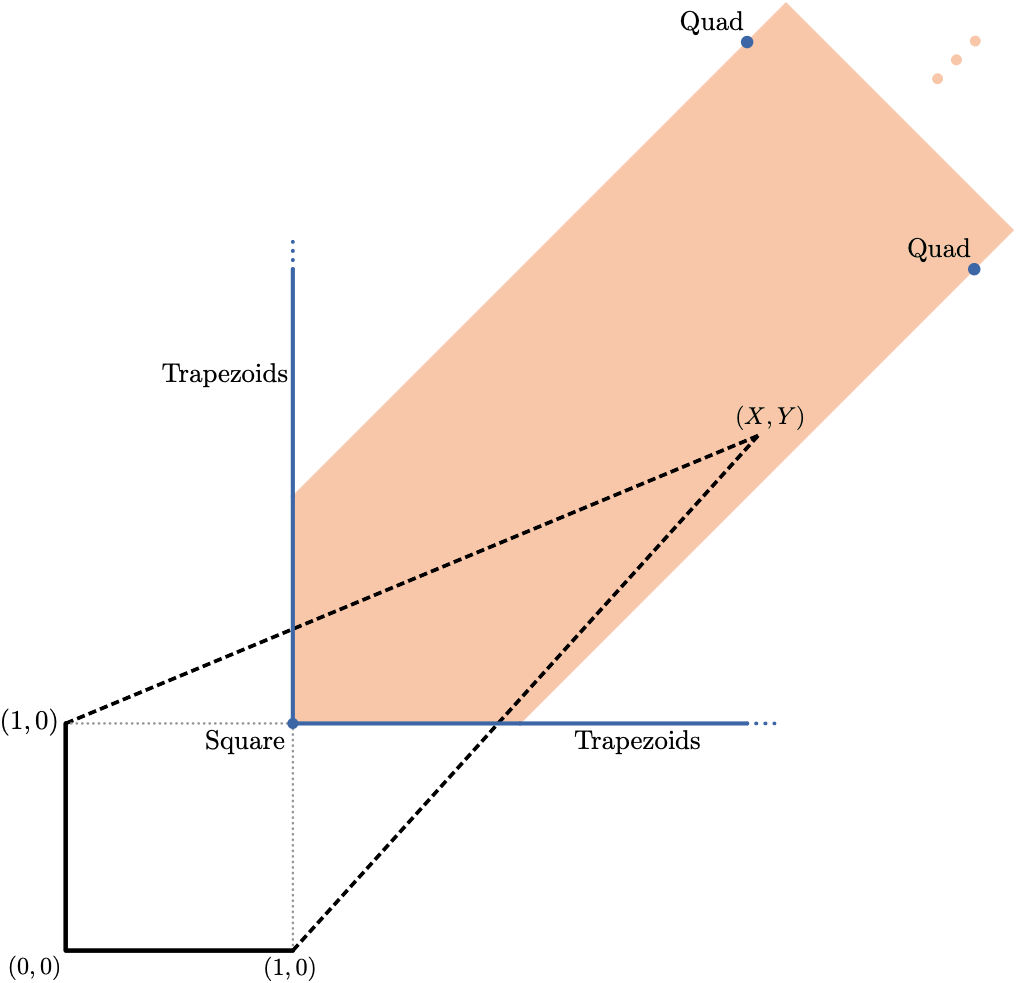} 
    \caption{The family of ne-quadrilaterals, i.e.~choose $(X,Y)$ in the orange region.}
    \label{fig:crooked_kite}
\end{figure}

These examples show that the set $IP$ is non-empty.
On the boundary of the family of ne-quadrilaterals, we have a square ($X = Y = 1$), and trapezoids ($X = 1$ or $Y = 1$). Both of these are in $BP$. Another point from $BP$ in the boundary is (an affine transformation of) the Quad, namely the polygon with vertices $(0,1)$, $(0,0)$, $(1,0)$, $(X=3,Y=4)$. However, computer experiments suggest that many quadrilaterals in the boundary are not in $FP$. Finally we point out that the space of quadrilaterals up to affine transformations is 2-dimensional and the space of ne-quadrilaterals forms an open subset in there.

Let us first illustrate how to find the isolated trajectory geometrically, see Figure \ref{fig:kite_period_6}.  The process is to cut off the red and blue triangles from the polygon. These triangles are then point-reflected in $\R^2$. After shrinking and translating they fit exactly into the polygon again and indeed form the even and odd parts of a 6-periodic symplectic billiard trajectory. Figure \ref{fig:kite_period_fail} shows how this construction fails on the boundary of the space of ne-quadrilaterals. Here, one of the slopes is that of the diagonal. 

\begin{figure}[ht]
    \centering
    \includegraphics[width=.6\linewidth]{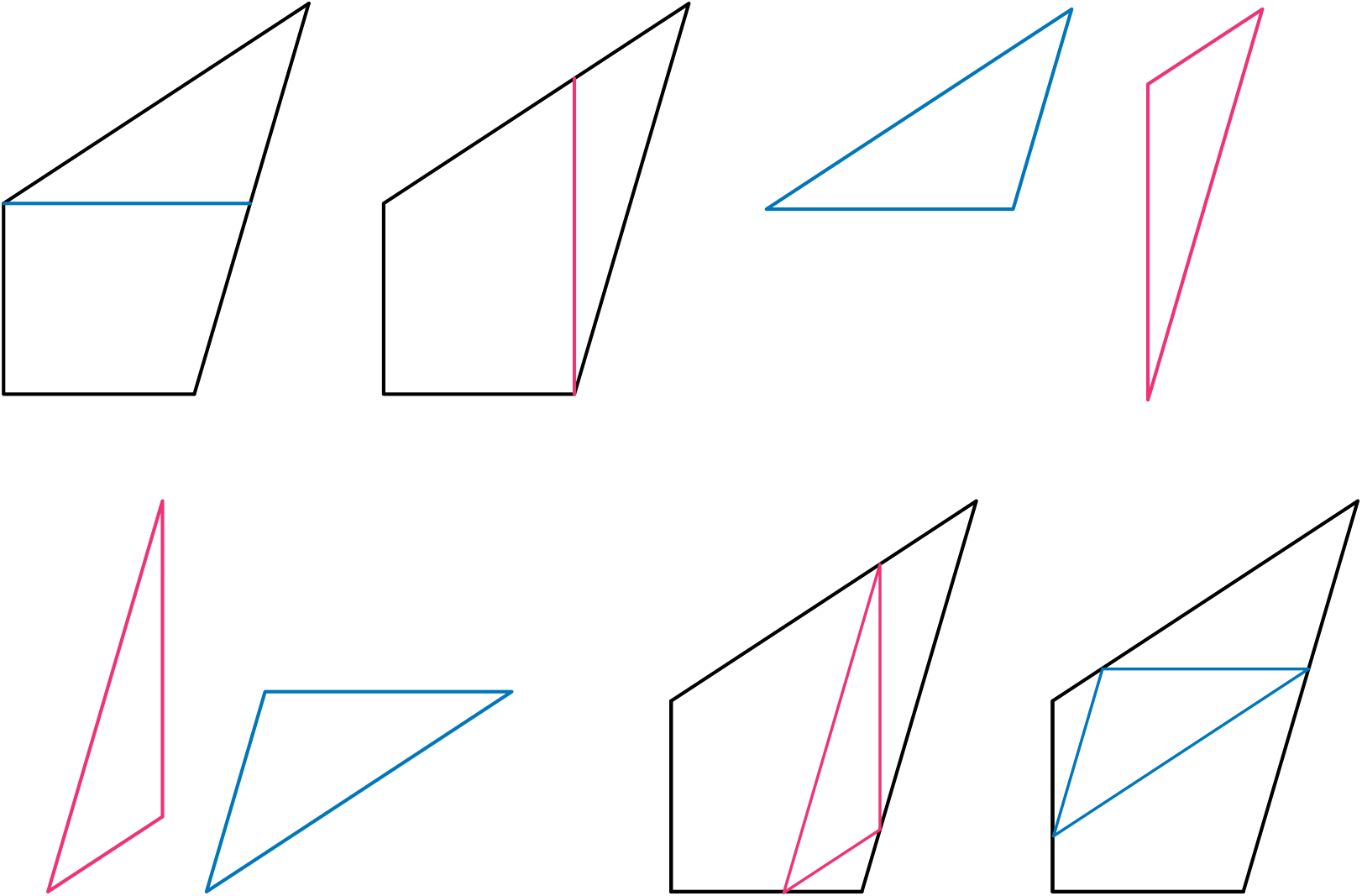}
    \caption{\small How to find a 6-periodic trajectory in any ne-quadrilateral. }
    \label{fig:kite_period_6}
\end{figure}

\begin{figure}[ht]
    \centering
    \includegraphics[width=.6\linewidth]{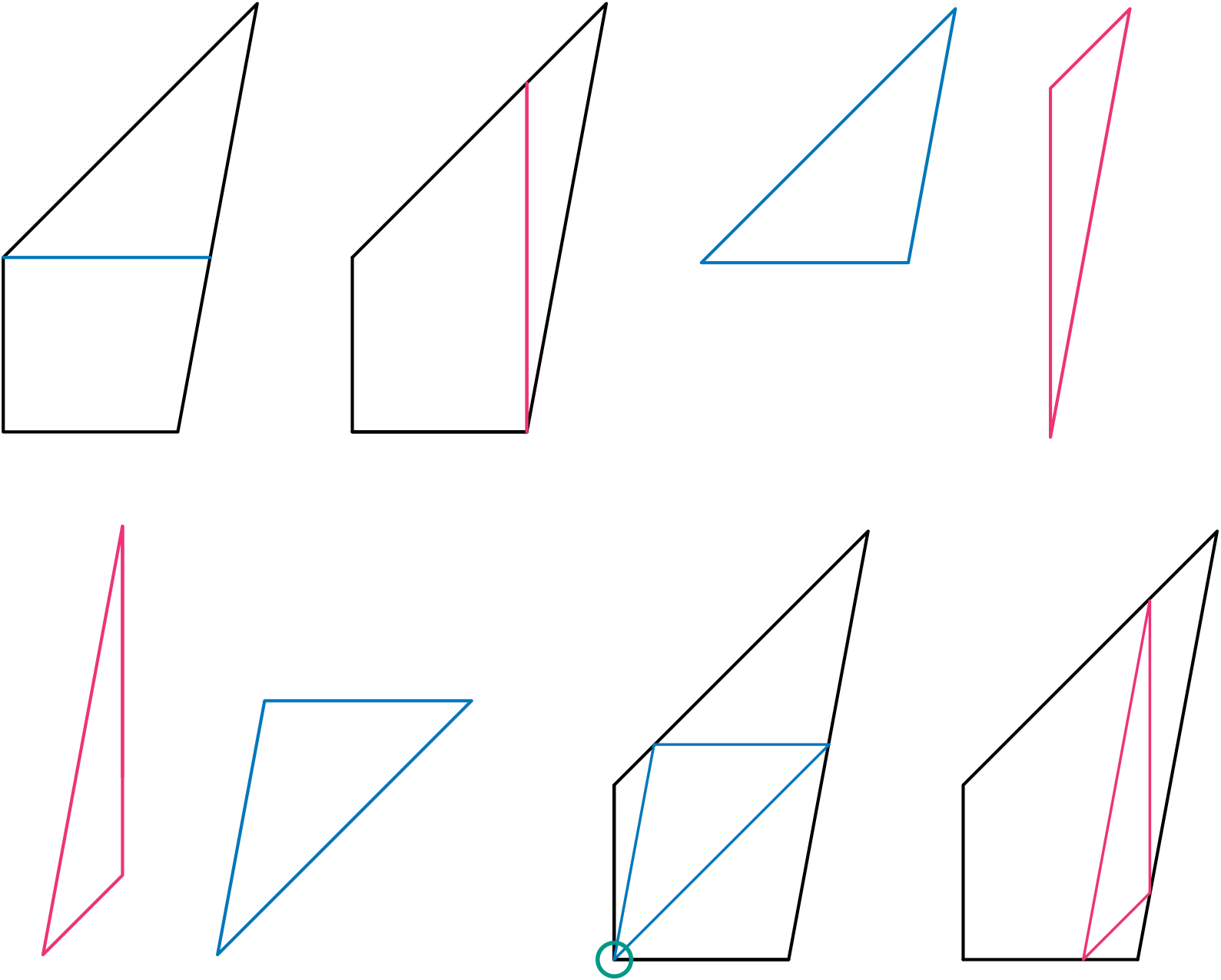}
    \caption{\small 
    This polygon has a diagonal slope and lies on the boundary of the family of ne-quadrilaterals. The 6-periodic orbits in nearby ne-quadrilaterals disappear in the displayed polygon as, in the limit, these run into the bottom-left vertex. }
    \label{fig:kite_period_fail}
\end{figure}

Let us now describe the periodic orbit in a ne-quadrilateral more analytically. The two relevant slopes of the sides of the ne-quadrilateral are $m_1=\frac{Y-1}{X}$ and $m_2=\frac{Y}{X-1}$. Now consider two families of rays starting on the horizontal side of the ne-quadrilateral with these two slopes, i.e.~$(s,0)+\R(X,Y-1)$ and $(s,0)+\R(X-1,Y)$ for $0\leq s<1$. In particular, each ray is parallel to one of the sides of the ne-quadrilateral. Since $X,Y > 1$, both rays intersect the ne-quadrilateral in two points, one being $(s,0)$ and the other is on a slanted side, see Figure \ref{fig:period_6_sub1}. We denote this other intersection point by $a(s)$ resp.~$b(s)$. 

\begin{figure}[ht]
        \includegraphics[width=0.65\linewidth]{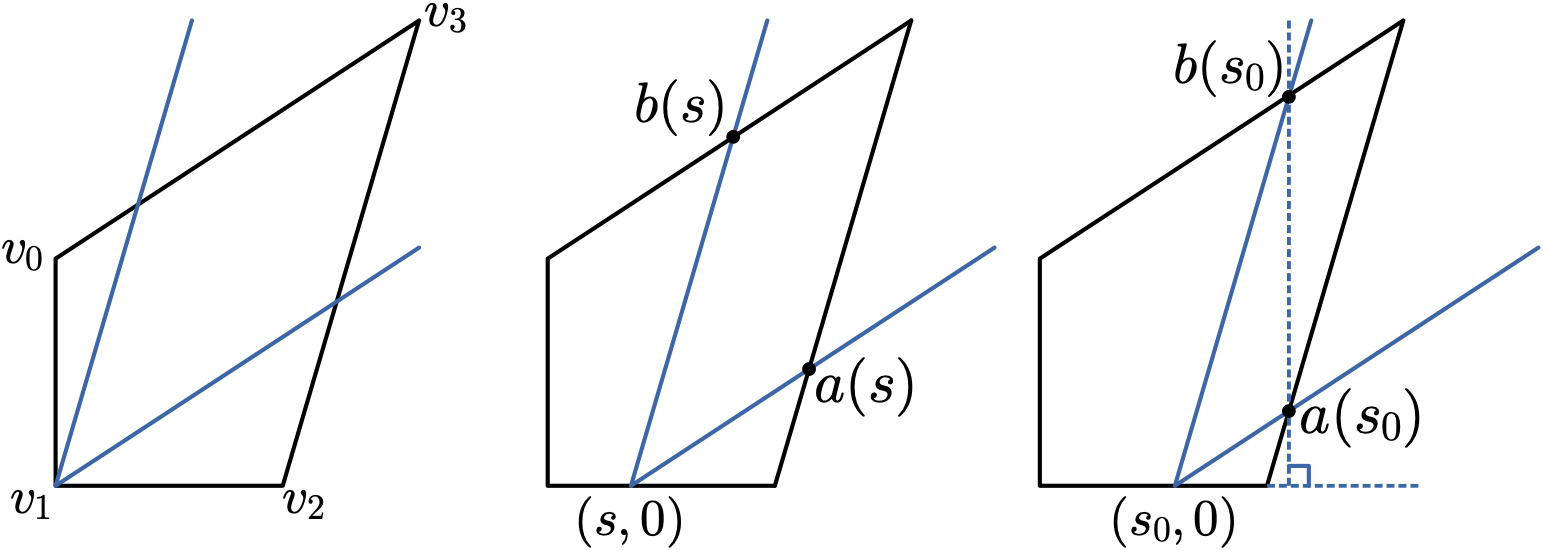}
        \caption{Defining the functions $a(s)$, $b(s)$ and determining the value $s_0$.}
        \label{fig:period_6_sub1}
\end{figure}

We claim that, due to the assumptions $X > 1$, $Y > 1$ and $|X-Y| < 1$, there is a unique $s_0\in(0,1)$ such that the $x$-coordinates of $a(s_0)$ and $b(s_0)$ agree, see again Figure \ref{fig:period_6_sub1}. Indeed, as a function of $s\in[0,1)$, the $x$-coordinate of $a(s)$ is strictly decreasing and that of $b(s)$ strictly increasing. The value of the $x$-coordinate of $a(0)$ equals $\frac{YX}{X+Y-1}$ and is strictly larger than that of $b(0)=\frac{(X-1)X}{X+Y-1}$. In the limit $s\to1$ this relation turns around; more precisely, the value of the $x$-coordinate of $a(1)$ equals $1$ which is strictly smaller than that of $b(1)=X$. The intermediate value theorem gives the desired $s_0$. 

We conclude that the triangle with vertices $(s_0,0)$, $a(s_0)$, $b(s_0)$ is inscribed in the ne-quadrilateral, has a vertical edge and two edges which are parallel to the slanted sides of the ne-quadrilateral. Analogously, considering two families of rays starting on the vertical side of the ne-quadrilateral instead we also find an inscribed triangle $(0,t_0)$, $d(t_0)$, $c(t_0)$ with one horizontal edge and two edges parallel to the slanted sides of the ne-quadrilateral. One of the coordinates of the latter triangle is of the form $(0,t_0)$. It can now easily be checked that the symplectic billiard orbit in the ne-quadrilateral with initial conditions $(x_0,x_1) = \big((0,t_0), b(s_0)\big)$ is 6-periodic, see Figure \ref{fig:period_6_sub2}.

\begin{figure}[ht]
        \includegraphics[width=0.6\linewidth]{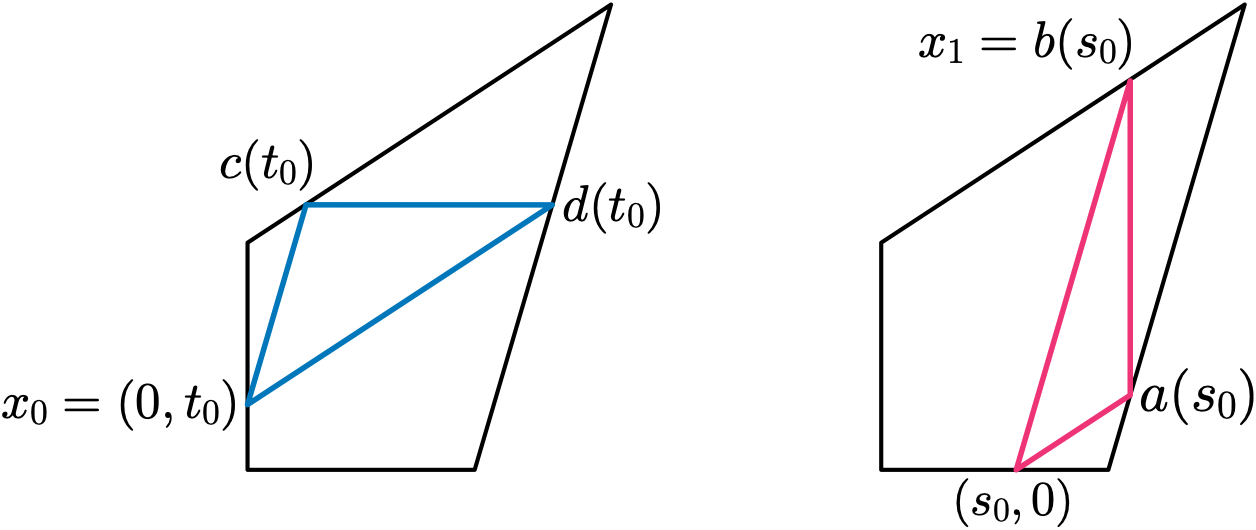}
        \caption{The symplectic billiard orbit in the ne-quadrilateral with initial conditions $(x_0,x_1) = \big((0,t_0), b(s_0)\big)$ is 6-periodic (and isolated).}
        \label{fig:period_6_sub2}
\end{figure}

It remains to show that this 6-periodic orbit is isolated. For that we consider symplectic billiard orbits with initial conditions $(x_0,x_1)$ satisfying $x_0=(0,t_0)$ and $|x_1-b(s_0)|<\delta$. For sufficiently small $\delta>0$ all these symplectic billiard orbits have (in forward time) the same even trajectory, see Figure \ref{fig:kite_contraction}. 

\begin{figure}[ht]
    \centering
    \includegraphics[width=0.7\linewidth]{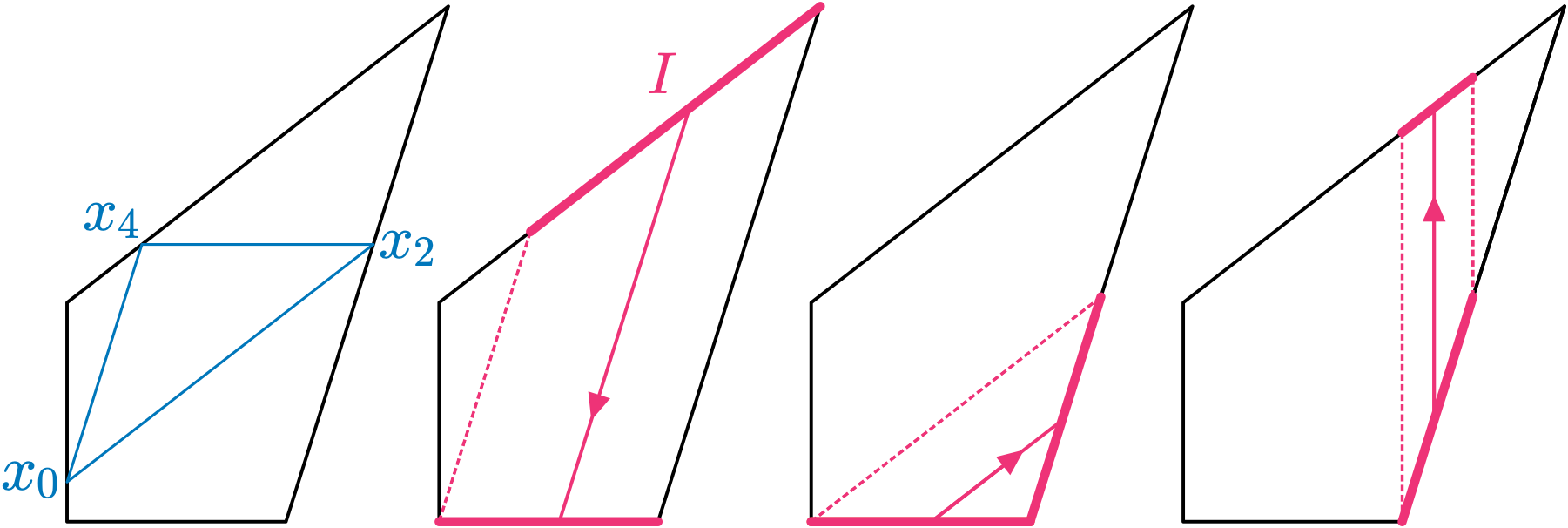}
    \caption{The return map of the set $\{x_0\}\times I$ is a contraction.}
    \label{fig:kite_contraction}
\end{figure}

In particular, we obtain a return map on the interval $(b(s_0)-\delta,b(s_0)+\delta)$. Due to the slopes of the sides of the ne-quadrilateral this return map is a contraction, see again Figure \ref{fig:kite_contraction}. In particular, this return map has a unique fixed point which necessarily is $b(s_0)$. Therefore, in any neighborhood of $\big((0,t_0), b(s_0)\big)$ in phase space there are non-periodic orbits and Lemma \ref{lemma:positive_tiles_are_periodic}  implies that the tile of $\big((0,t_0), b(s_0)\big)$ has zero area. Thus, by definition, this 6-periodic orbit is isolated. This proves that every ne-quadrilateral carries an isolated periodic symplectic billiard orbit of period 6.

\section{The necktie has no periodic orbits}\label{section:necktie}

In this section we present a pair of polygons, the necktie, see Figure \ref{fig:necktie}, for which all symplectic billiard orbits are non-periodic. For this, the first return map of a particular subset of the corresponding phase space plays a crucial role. It turns out that this return map is related to the so-called dyadic odometer resp.~the von Neumann-Kakutani transformation or the 2-adic adding machine. Before discussing symplectic billiards on the necktie we recall different representations of these maps. One is the one occurring as part of the first return map, the other clearly shows that this particular map has no periodic points.

\begin{figure}[ht]
    \centering
    \includegraphics[width=0.09\linewidth]{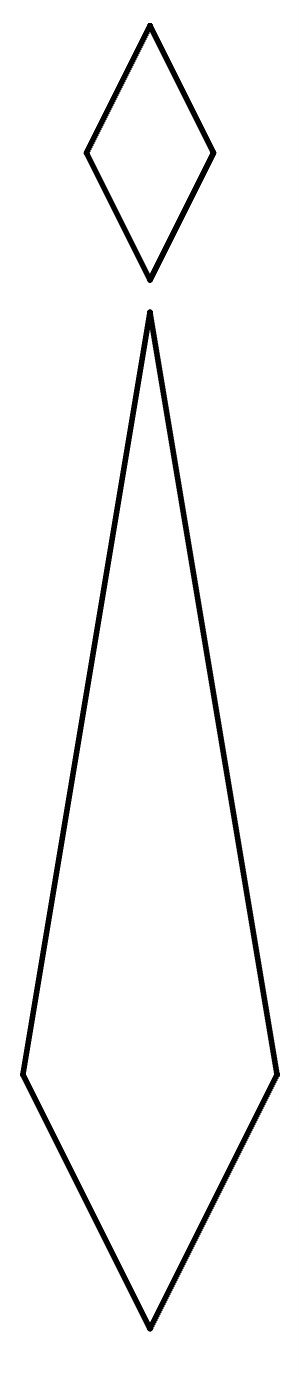}
    \caption{The necktie -- a pair of polygons for which the symplectic billiard map has no periodic orbits at all.}
    \label{fig:necktie}
\end{figure}

Consider the set $\{0,1\}^\N$ of sequences of $0$ and $1$ and the map
\begin{equation*}
    \begin{split}
        S:\{0,1\}^\N &\to \{0,1\}^\N\\
        (1,\ldots,1,0,a_k,a_{k+1},\ldots) &\mapsto (0,\ldots,0,1,a_k,a_{k+1},\ldots).
    \end{split}
\end{equation*}
That is, whenever a sequence starts with a number of 1's followed by a 0 it is replaced by an equal number of 0's followed by a 1. The remaining part of the sequence is unchanged. This map is called the dyadic odometer. Next we give two other representations of this map. In the last one it is easy to see that $S$ has no periodic orbits.

The first representation of $S$ is as an interval exchange transformation (von Neumann-Kakutani transformation) via the binary representation 
\begin{equation}\nonumber
\begin{aligned}
\varphi: \{0,1\}^\N &\to [0,1]\\ 
(a_1,a_2,\ldots) &\mapsto \sum_{i=1}^\infty a_i2^{-i}.
\end{aligned}
\end{equation}
Of course, $\varphi^{-1}$ is defined only outside the dyadic rationals, i.e.~exactly those where the binary representation is not unique. The map $\varphi\circ S\circ\varphi^{-1}$ is the interval exchange transformation depicted in Figure \ref{fig:dyadic_odometer}.

\begin{figure}[ht]
    \centering
    \includegraphics[width=.5\linewidth]{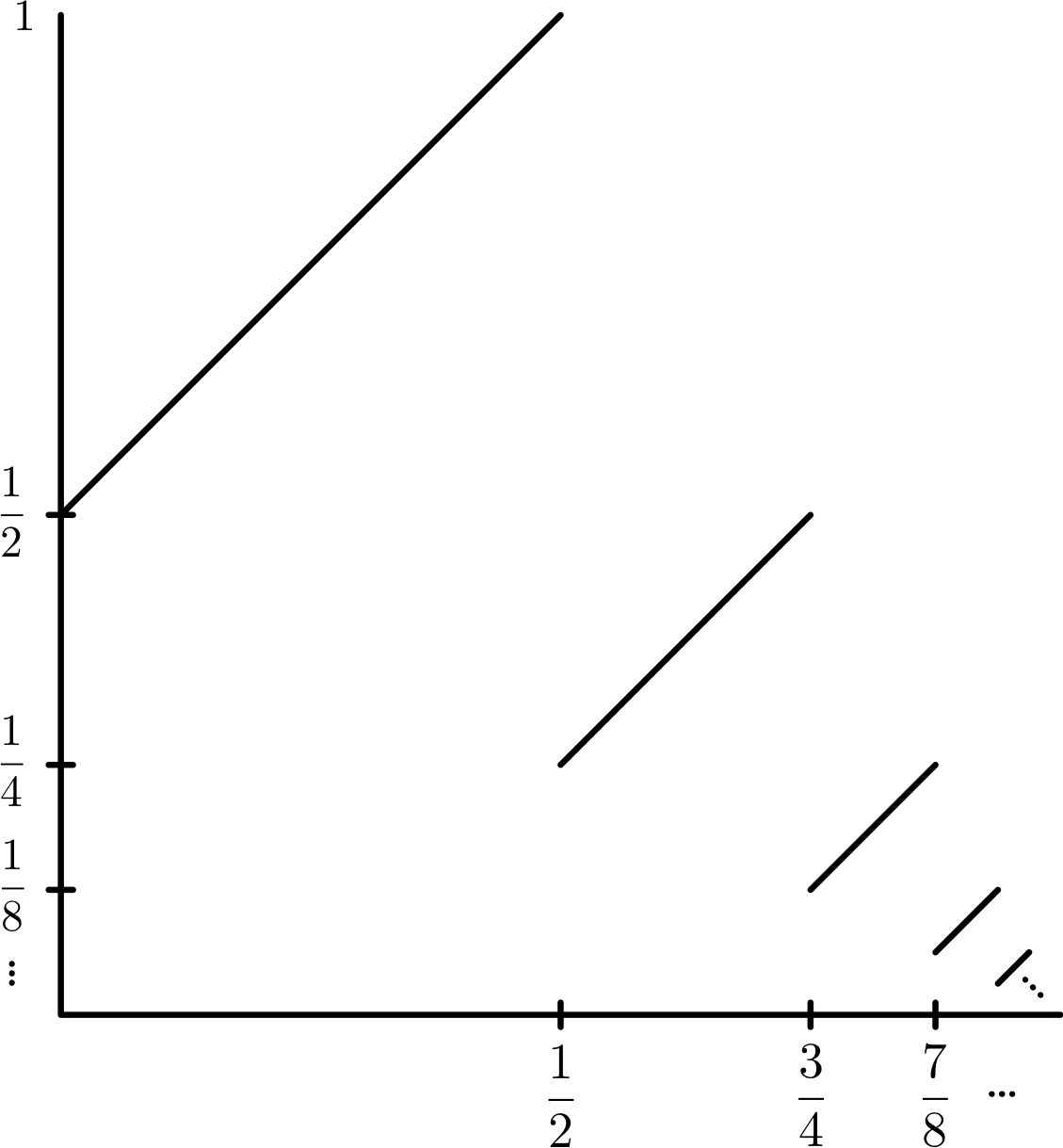}
    \caption{The von Neumann-Kakutani transformation.}
    \label{fig:dyadic_odometer}
\end{figure}

If we instead identify $\{0,1\}^\N$ with the 2-adic integers $\Z_2$, that is, if we consider the bijection
\begin{equation}\nonumber
\begin{aligned}
\psi: \{0,1\}^\N &\to \Z_2\\ 
(a_1,a_2,\ldots) &\mapsto \sum_{i=1}^\infty a_i 2^i,
\end{aligned}
\end{equation}
then the map $S$ becomes
\begin{equation}\nonumber
\begin{aligned}
\psi \circ  S \circ \psi^{-1}: \Z_2 &\to \Z_2\\
x &\mapsto x+1,
\end{aligned}
\end{equation}
the 2-adic adding machine. In the 2-adic integers the equation $x+n=x$, $n\in\N_0$, has only $n=0$ as a solution and thus the map $S$ does not have periodic points.

Now we describe a concrete realization of the necktie as the square and the kite with vertices $v_0 = (2,2)$, $v_1 = (2,0)$, $v_2 = (0,0)$ and $ v_3 = (0,2)$ resp.~$w_0 = (5,2)$, $w_1 = (4,2)$, $w_2 = (3,0)$ and $w_3 =(5,1)$, see Figure \ref{fig:square_kite}. 

\begin{figure}[ht]
    \centering
    \includegraphics[width=.55\linewidth]{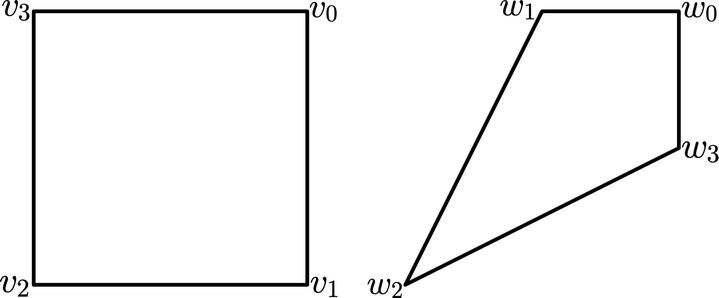}
    \caption{For convenience we rearrange the necktie from Figure \ref{fig:necktie}.}
    \label{fig:square_kite}
\end{figure}

To show that the symplectic billiard map has no periodic orbits we make the following observations:

\begin{enumerate}  
\itemsep=1.5ex
\item For every point in the phase space $\P$ the corresponding symplectic billiard trajectory eventually hits some point in $w_3w_0\cup w_0w_1$.
\item For $x\in w_3w_0$ the symplectic billiard map defines a map of  the 1-dimensional domain $\P \cap (\{x\}\times v_1v_2)$ resp.~$\P \cap (\{x\}\times v_3v_0)$. Similarly, for $x\in w_0w_1$ we obtain a map of $\P \cap (\{x\}\times v_0v_1)$ resp.~$\P \cap (\{x\}\times v_2v_3)$. After identifying the domain $\{x\}\times v_iv_{i+1}\cong(0,1)$ we obtain a measurable transformation of the segment $(0,1)$ and points in  $\{x\}\times v_iv_{i+1}$ that are \textit{not} in $\P$ correspond to the dyadic rationals in $(0,1)$. In this identification each map coincides with the von Neumann-Kakutani transformation. 
\end{enumerate}

Combining these two observations with the fact that $S$ has no periodic points immediately implies that the symplectic billiard map of the necktie cannot have periodic orbits. In the remainder of this section we prove the two observations.

\begin{proof}[Proof of Observation (i)]
Let  $(x_k)_{k\in\Z}$ be a symplectic billiard trajectory. Without loss of generality we may assume that the even trajectory $(x_{2k})_{k\in\Z}$ lies in the kite and the odd trajectory $(x_{2k+1})_{k\in\Z}$ in the square. We first observe that the even trajectory needs to follow a ``staircase'' pattern, see the blue part in Figure \ref{fig:observation_sub1}. 

\begin{figure}[ht]
        \centering
        \includegraphics[width=0.90\linewidth]{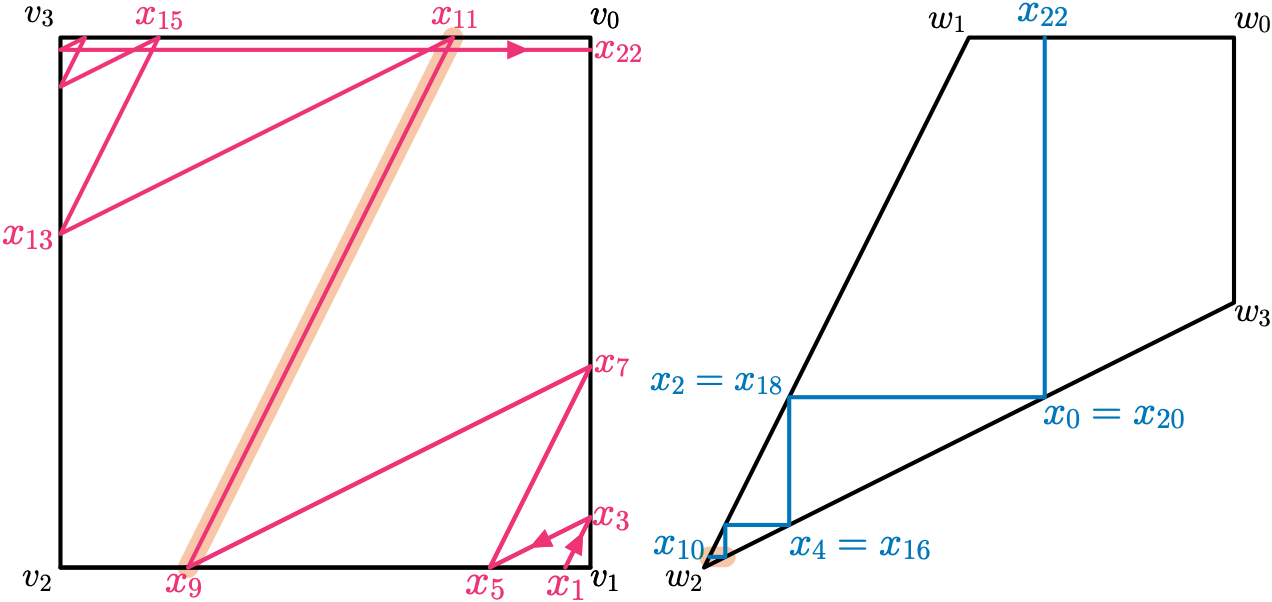}%
        \caption{A symplectic billiard trajectory starting at $(x_0,x_1)$ whose even trajectory eventually hits $w_0w_1$.\\[1ex]}
        \label{fig:observation_sub1}
\end{figure}
 
%
%
 
Indeed, $x_{2k+2}$ is obtained from $x_{2k}$ by moving vertically up/down or horizontally left/right, that is, parallel to the sides of the square. Assume now that the even trajectory descends along a staircase towards $w_2$ in the kite. We argue next that it has to ``turn around'' after a finite number of iterations, that is, it stops moving towards $w_2$ and starts moving away from $w_2$. This occurs at $x_{10}$ in Figure \ref{fig:observation_sub1}. Since the even trajectory in the kite moves towards $w_2$ the odd trajectory in the square moves up from the corner $v_1$ or down from $v_3$. Therefore, the odd trajectory necessarily will hit the vertical side or horizontal side of the square twice in a row after a finite number of iterations, see highlighted parts in Figure \ref{fig:observation_sub1}. Precisely then the even trajectory turns around, i.e.~retraces the previous steps and moves up the staircase again switching between horizontal and vertical direction while the odd trajectory continues to move up from $v_1$ towards $v_3$ or down from $v_3$ towards $v_1$. Therefore, the lines connecting consecutive elements of the odd trajectory become shorter and shorter. In particular, the odd trajectory needs to alternate between vertical and horizontal sides of the square. Thus, the even trajectory continues moving up on a staircase until it hits some point $x \in w_3w_0\cup w_0w_1$, as claimed. 
\end{proof}

\begin{remark}
We point out that when the even trajectory hits $w_3w_0\cup w_0w_1$ then the odd trajectory will next move either horizontally or vertically, that is, move from the vertical side to the vertical side of the square or from the horizontal side to the horizontal side. As a consequence the even trajectory will again ``turn around'' on the staircase. Also, while every even trajectory will at some point descend along a staircase towards $w_2$ in the kite and then necessarily turn around, the number of iterations until this turn around occurs very much depends on the value of the corresponding odd trajectory in the square. This leads to the highly non-trivial map on $\P \cap (\{x\}\times v_1v_2)$ and similar sets. 
\end{remark}

\begin{proof}[Proof of Observation (ii)]
By symmetry considerations it is enough to consider a point $x\in w_0w_1$ and describe the symplectic billiard trajectories starting in $\{x\}\times v_0v_1$. We identify $\{x\}\times v_0v_1$ with the interval $(0,1)$ via 
\begin{equation}\nonumber
\{x\}\times v_0v_1=\{(x,(1-t)v_0 + tv_1)\mid t\in(0,1)\}\cong(0,1).
\end{equation}
With this identification we claim that for every $\ell\in\N$ the subinterval $(\frac{2^{\ell-1}-1}{2^{\ell-1}},\frac{2^\ell-1}{2^\ell})$ of $\{x\}\times v_0v_1$ returns (after several applications of the symplectic billiard map) to $\{x\}\times v_0v_1$ and is mapped to the subinterval $(1-\frac{2^\ell-1}{2^\ell}, 1-\frac{2^{\ell-1}-1}{2^{\ell-1}})$. Moreover, this map is simply the translation
\begin{equation}\nonumber
\begin{aligned}
\left(\frac{2^{\ell-1}-1}{2^{\ell-1}},\frac{2^\ell-1}{2^\ell}\right) &\longrightarrow \left(1-\frac{2^\ell-1}{2^\ell}, 1-\frac{2^{\ell-1}-1}{2^{\ell-1}}\right)\\[1ex]
s&\mapsto \quad s+1-\frac{2^{\ell-1}-1}{2^{\ell-1}}-\frac{2^\ell-1}{2^\ell}.
\end{aligned}
\end{equation}
More precisely, if the initial value of a symplectic orbit is $(x_0,x_1)\in \{x\}\times v_0v_1\cong(0,1)$ with $(x_0,x_1)\in(\frac{2^{\ell-1}-1}{2^{\ell-1}},\frac{2^\ell-1}{2^\ell})$ then 
\begin{equation}\nonumber
x_{4\ell}=x_0 \quad\text{and}\quad x_{4\ell+1}=x_1+1-\frac{2^{\ell-1}-1}{2^{\ell-1}}-\frac{2^\ell-1}{2^\ell}.
\end{equation}
Moreover, $x_{2k}\neq x_0$ for all $k=1,\ldots,2\ell-1$. Thus, the self-map of $\{x\}\times v_0v_1$ is, after the identification $\{x\}\times v_0v_1\cong(0,1)$, precisely the above interval exchange map, the von Neumann-Kakutani transformation. In particular, the dyadic rationals are, after some iterations of the symplectic billiard map, mapped to points of the form $\frac{2^{k-1}-1}{2^{k-1}}$, $k\in\N$. These, in turn, then are mapped eventually into vertices, see our description below. Hence the dyadic rationals are not part of the phase space.  In fact, this identifies the dyadic rationals with points in the discontinuity set $N$.

To understand the above claim let us describe in more detail the odd part $(x_{2k+1})_{k\in\Z}$ of the symplectic billiard trajectory with initial values $(x_0,x_1)\in(\frac{2^{\ell-1}-1}{2^{\ell-1}},\frac{2^\ell-1}{2^\ell})\subset (0,1)\cong \{x\}\times v_0v_1$. The cases $\ell=1,2,3,4$ are shown in Figure \ref{fig:square_kite_map_description_without_2_adics}, where we write $x_1\in(0,\tfrac12)$ instead of $(x_0,x_1)\in(0,\tfrac12)$ etc.~for convenience.
 We point out that for $\ell\geq2$ the shift amount $1-\frac{2^{\ell-1}-1}{2^{\ell-1}}-\frac{2^\ell-1}{2^\ell}$ is negative while for $\ell=1$ it is $+\frac12$.

\begin{figure}[ht]

    \begin{subfigure}{\linewidth}
        \centering
        \includegraphics[width=0.65\linewidth]{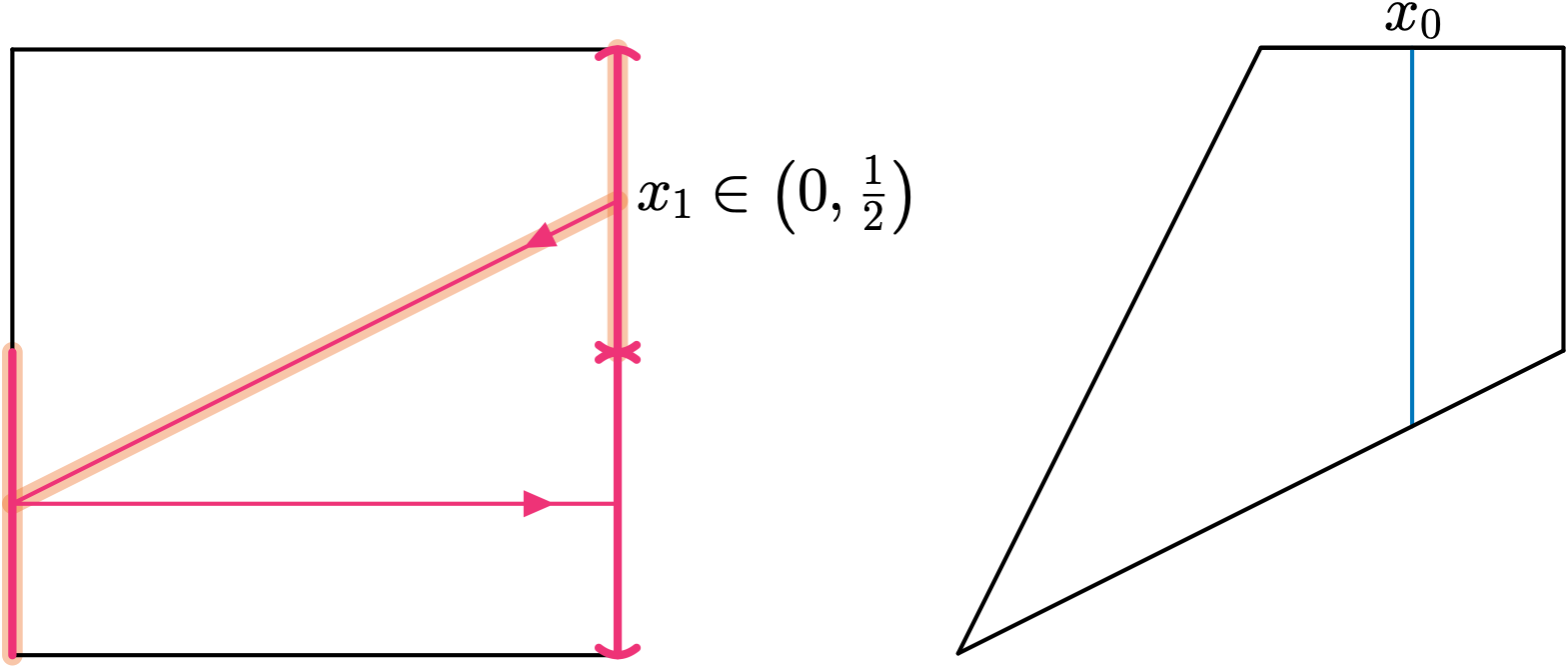}
        \label{fig:sub1}
    \end{subfigure}\\[1ex]
    
    \begin{subfigure}{\linewidth}
        \centering
        \includegraphics[width=0.65\linewidth]{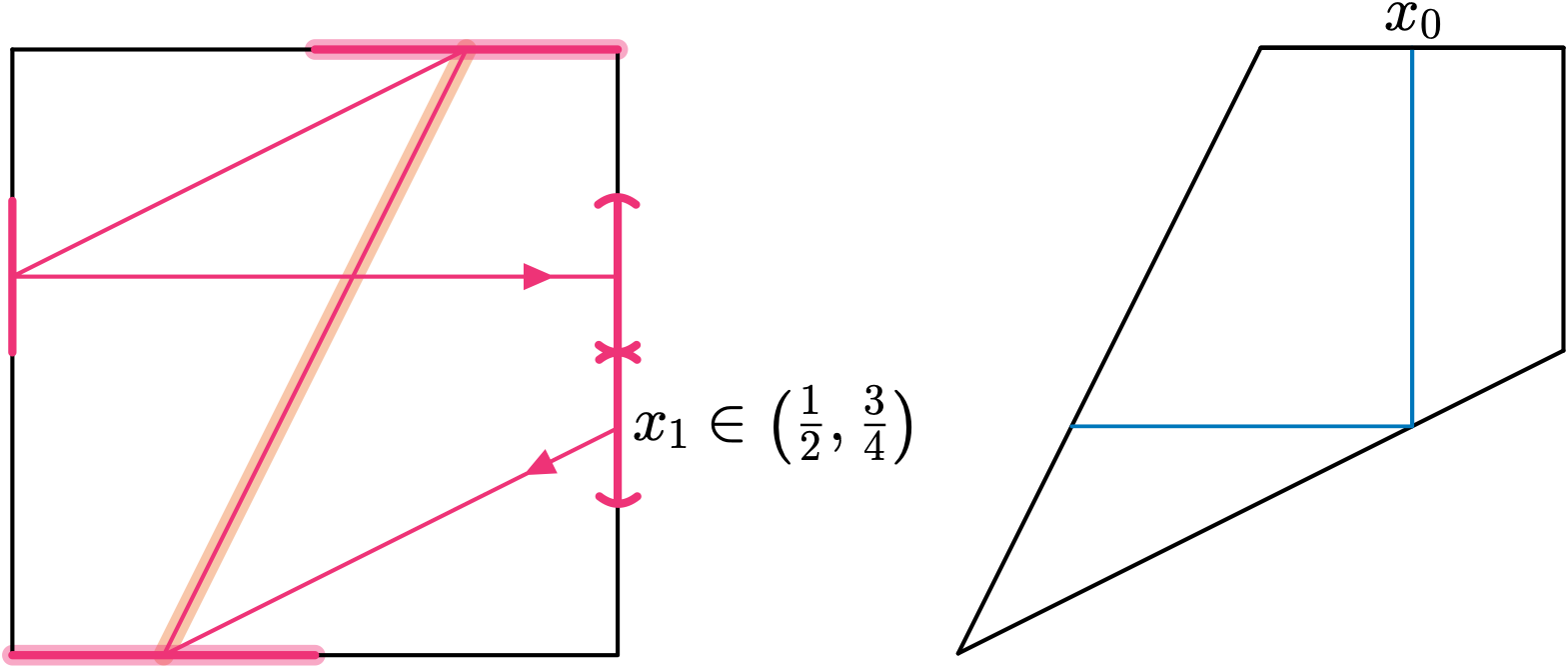}
        \label{fig:sub2}
    \end{subfigure}\\[1ex]

    \begin{subfigure}{\linewidth}
        \centering
        \includegraphics[width=0.65\linewidth]{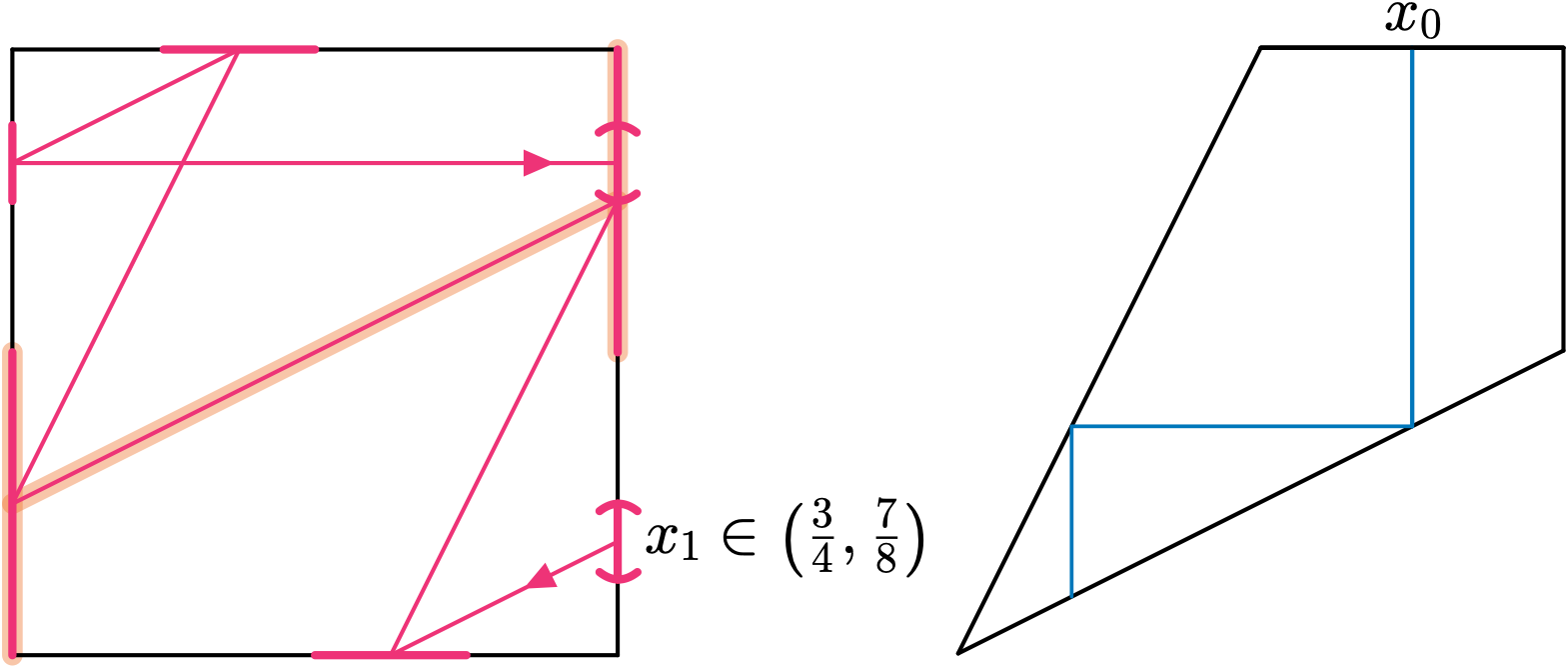}
        \label{fig:sub3}
    \end{subfigure}\\[1ex]
    
    \begin{subfigure}{\linewidth}
        \centering
        \includegraphics[width=0.65\linewidth]{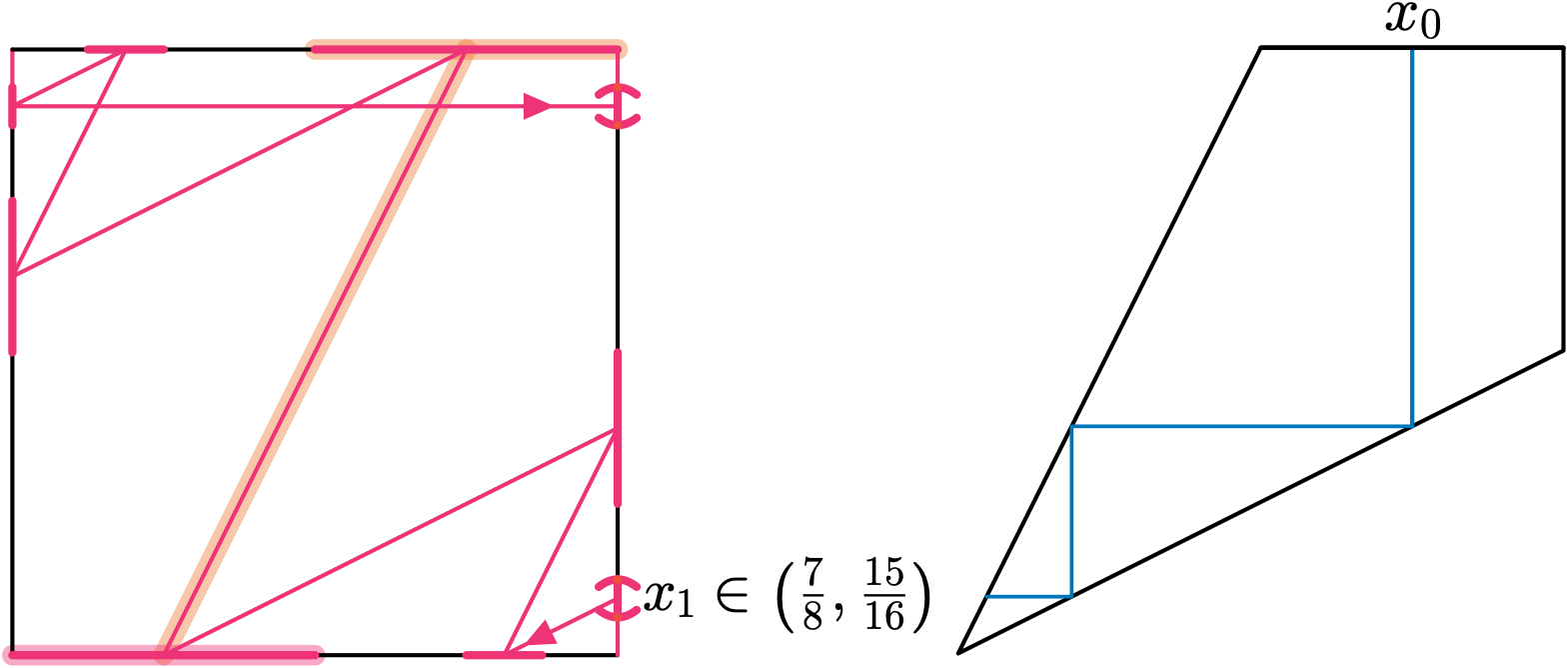}
        \label{fig:sub4}
    \end{subfigure}
    
    \caption{The map is the von Neumann-Kakutani transformation.}
    \label{fig:square_kite_map_description_without_2_adics}
\end{figure}

As explained above the even trajectory $(x_{2k})$ will move along a staircase (starting going down). With the given initial conditions the even trajectory will ``turn around'' precisely at $x_{2\ell}$ and move up until $x_{4\ell}=x_0$. Also as explained above, the points $x_{2\ell-1}$ and $x_{2\ell+1}$ are both on a vertical side (for $\ell$ odd) or a horizontal side (for $\ell$ even) of the square. The points $x_1, x_3, \ldots, x_{2\ell-1}$ and $x_{2\ell+1},\ldots, x_{4\ell-1}$ alternate between vertical and horizontal sides. Finally $x_{4\ell-1}$ and $x_{4\ell+1}$ are both on the vertical side and the map is $(x_0=x,x_1)\mapsto (x_{4\ell}=x,x_{4\ell+1})$.

As for the odd trajectories, the interval $(\frac{2^{\ell-1}-1}{2^{\ell-1}},\frac{2^\ell-1}{2^\ell})$ containing $(x_0,x_1)$ is mapped by the symplectic billiard map alternatingly $(\ell-1)$-times from a vertical to a horizontal resp.~a horizontal to a vertical side of the square. These respective maps are projections along lines with slope $\frac12$ resp.~$2$ (the slopes of the sides of the kite). Each projection is an orientation reversing map of the form $s\mapsto -2s+s_0$.
In fact, the initial interval $(\frac{2^{\ell-1}-1}{2^{\ell-1}},\frac{2^\ell-1}{2^\ell})$ is, after $(\ell-1)$ steps, enlarged to an interval of length $\frac12$  by this process. This interval of length $\tfrac12$ is either contained in a horizontal or vertical side of the square, depending on the parity of $\ell$. Next, this interval is mapped to the opposite side simply by ``a flip'', see the highlighted part in Figure \ref{fig:square_kite_map_description_without_2_adics}. This map is an orientation reversing isometry. This is the iteration step where the even trajectory ``turns around'', see above. Now, the entire process is inverted, that is, the interval of length $\frac12$ is mapped alternatingly $(\ell-1)$-times from a vertical to a horizontal resp.~a horizontal to a vertical side of the square by projections but now with the slopes interchanged. I.e.~the interval is shrinking (and is being shifted). After another $(\ell-1)$ steps the interval is shrunk back to size $\frac{1}{2^\ell}$ and lies on the vertical side $v_2v_3$ of the square. Finally, the even trajectory returns to $x_{4\ell}=x_0$ and the interval is mapped over horizontally to the vertical side $v_0v_1$. This is again an orientation reversing isometry. Therefore, the map of $(\frac{2^{\ell-1}-1}{2^{\ell-1}},\frac{2^\ell-1}{2^\ell})\subset \{x\}\times v_0v_1$ is the composition of $2(\ell-1)+2=2\ell$ 
 orientation reversing maps, forms an isometry and is thus of the form $s\mapsto s+s_\ell$. Following for instance the mid-point in $(\frac{2^{\ell-1}-1}{2^{\ell-1}},\frac{2^\ell-1}{2^\ell})$ through this process shows 
\begin{equation}\nonumber
s_\ell=1-\frac{2^{\ell-1}-1}{2^{\ell-1}}-\frac{2^\ell-1}{2^\ell}.
\end{equation}
This completes the argument.
\end{proof}

\section{Pairs of smooth, strictly convex tables}\label{section:smooth_existence}

In the previous section we exhibited an example of two polygonal convex tables so that the symplectic billiard map has no periodic orbits at all. In contrast, in this section we will show that for two smooth, strictly convex tables, there always exists a periodic orbit on the corresponding symplectic billiard map. In fact, for each $k \ge 2 $ there exists a $2k$-periodic orbit. However, we do not exclude that these orbits are multiple covers, i.e.~if $k$ is divisible by $p$ then the orbit might be the $\tfrac{k}{p}$-fold iteration of some $p$-periodic orbit. We also remind the reader that in the two table setting periodic orbits necessarily have even period.

Let us describe the set-up. We consider two smooth embedded loops $\gamma_-$ and $\gamma_+$ in the plane each bounding a strictly convex domain. It is convenient but not necessary for visualization to assume that these domains are disjoint. For a pair of points $(x_0,x_1) \in \gamma_\pm \times \gamma_\mp$ with distinct tangent directions, i.e.~$\det(\nu_{x_0},\nu_{x_1}) \neq 0$, we define a new point $x_2$ as the intersection of the curve $\gamma_\pm$ with the line $x_0 + T_{x_1}\gamma_\mp$ other than $x_0$. Since $\gamma_\pm$ are strictly convex the point $x_2$ is uniquely defined. The  symplectic billiard map is then by definition $\phi(x_0,x_1) = (x_1,x_2)$. Then $\phi$ maps the phase space
$$
\P(\gamma_-,\gamma_+) := \{(x,y)\in (\gamma_-\times \gamma_+) \sqcup (\gamma_+ \times \gamma_-) | \det(\nu_x,\nu_y) \neq 0\} 
$$
into itself. Thus, we may iterate $\phi$, speak of orbits, trajectories etc., as in the polygonal case. Moreover, due to convexity the symplectic billiard map is reversible on the entire phase space. We are now ready to state the analog of \cite[Theorem 11]{Albers_Tabachnikov_Introducing_symplectic_billiards} concerning the existence of periodic orbits. 

\begin{theorem}\label{theorem:smooth_existence}
    Let $\gamma_-$ and $\gamma_+$ be strictly convex, smooth loops in the plane. Then for each $k \ge 2$ there exists a $2k$-periodic symplectic billiard trajectory on the phase space $\P(\gamma_-,\gamma_+)$.
\end{theorem}

\begin{remark}
As mentioned above we do not exclude that the $2k$-periodic orbits are multiple covers. Restricting to $2p$, $p$ prime, still leads to infinitely many periodic orbits.

We point out that if we find a $2k$-periodic orbit starting in $\P(\gamma_-,\gamma_+)\cap(\gamma_- \times \gamma_+)$ then after shifting indices by one along this periodic orbit we obtain a $2k$-periodic orbit starting in $\P(\gamma_-,\gamma_+)\cap(\gamma_+ \times \gamma_-)$. 
\end{remark}

The proof of Theorem \ref{theorem:smooth_existence} is basically the same as the one given in \cite{Albers_Tabachnikov_Introducing_symplectic_billiards}. For the reader's convenience we provide the argument here. Consider the function 
\begin{equation}\nonumber
\begin{aligned}
f_k : (\gamma_- \times \gamma_+)^k &\to \R\\
(z_1,\ldots,z_{2k}) &\mapsto \sum_{i=1}^{2k} \omega(z_i,z_{i+1})
\end{aligned}
\end{equation}
where we still read indices cyclically and $\omega$ is the standard symplectic form on the plane $\R^2$. The idea is that critical points of $f_k$ are periodic orbits up to ``back-tracking''. It turns out that there is no back-tracking for a maximum / minimum of $f_k$ and this gives us the sought after periodic orbit. 

\begin{lemma}\label{lemma:critical_point_shoelace_formula}
A point $(z_1,\ldots,z_{2k})\in(\gamma_- \times \gamma_+)^k$ is a critical point of $f_k$ if and only if
\begin{equation}\nonumber
z_{j+1}-z_{j-1} \in T_{z_j}\gamma_\pm
\end{equation}
holds for all $j=1,\ldots, 2k$. 
\end{lemma}

\begin{proof}
The map $f_k$ extends (by the same formula) to a map $F_k:\R^{4k}\to\R$.  Computing $DF_k(z_1,\ldots,z_{2k})$ and restricting to the tangent space $T_{(z_1,\ldots,z_{2k})}(\gamma_- \times \gamma_+)^k$ immediately gives the  critical point equation
$$
\omega(v,z_{j+1}-z_{j-1})=0\quad \forall v \in T_{z_j}\gamma_\pm\quad\forall j=1,\ldots, 2k
$$
which is, since $\gamma_\pm\subset\R^2$, equivalent to 
$$z_{j+1}-z_{j-1} \in T_{z_j}\gamma_\pm$$ for all $j=1,\ldots, 2k$ .
\end{proof}

We consider the maximal value of $f_k$ attained on the compact set $(\gamma_- \times \gamma_+)^k$.

\begin{lemma}\label{lemma:maximal_value_shoelace_formula}
    The maximal value of $f_k$ is strictly larger than that of $f_{k-1}$.
\end{lemma}

\begin{proof}
Let $f_{k-1}$ attain its maximal value at $(w_1,\ldots,w_{2k-2})\in(\gamma_- \times \gamma_+)^{k-1}$. We claim that we can find a point $(c,d) \in \gamma_- \times \gamma_+$ such that $f_k(w_1,\ldots,w_{2k-2},c,d) > f_{k-1} (w_1,\ldots,w_{2k-2})$. This clearly will prove the Lemma.

For notational convenience we set $a = w_1$ and $ b = w_{2k-2}$. Rearranging the desired inequality gives
\begin{equation}\nonumber
\begin{aligned}
0 &< f_k(a,w_2,\ldots,w_{2k-3},b,c,d)- f_{k-1} (a,w_2,\ldots,w_{2k-3},b) \\
&=\omega(b,c) + \omega(c,d) + \omega(d,a) - \omega(b,a).
\end{aligned}
\end{equation}
We point out that changing the relative position of $\gamma_-$ and $\gamma_+$ does not change the value of $f_k$. Indeed, moving $\gamma_-$ to $\gamma_-+v$ for some $v\in\R^2$ and keeping $\gamma_+$ fixed leads to
\begin{equation}\nonumber
\begin{aligned}
f_k(w_1+v,&w_2,w_3+v,w_4,\ldots,w_{2k-1}+v,w_{2k})\\
&=f_k(w_1,\ldots,w_{2k})+\omega(v,w_2)+\omega(w_2,v)+\ldots+\omega(v,w_{2k})+\omega(w_{2k},v)\\
&=f_k(w_1,\ldots,w_{2k}).
\end{aligned}
\end{equation}
In particular, we may move $\gamma_-$ such that $\gamma_-\ni a=w_1=w_{2k-2}=b\in\gamma_+$. Thus, now the task is to find $c,d$ such that
\begin{equation}\nonumber
\omega(b,c) + \omega(c,d) + \omega(d,b)>0,
\end{equation}
that is, we reduced the problem to finding $(c,d) \in \gamma_- \times \gamma_+$ such that the oriented area of the triangle $bcd$ is positive where $a=b\in\gamma_-\cap\gamma_+$ is given.

Since $\gamma_-$ is strictly convex it is certainly not contained in $T_b\gamma_+$. In particular, we may choose $c\in\gamma_-\setminus T_b\gamma_+$ such that the line through $a=b$ and $c$ is transverse to $\gamma_+$ at $b=a$, i.e.
$$
a + \R(c-a) \pitchfork T_b\gamma_+.
$$
Now choose $d\in\gamma_+$ in the half-space w.r.t.~the line $a+\R(c-a)$ making the oriented area of the triangle $bcd$ positive. This finishes the proof of the Lemma.
\end{proof}

Now, we are in the position to prove Theorem \ref{theorem:smooth_existence}.
    
\begin{proof}[Proof of Theorem \ref{theorem:smooth_existence}]
We begin by rephrasing the condition of being a symplectic billiard trajectory in terms of the function $f_k$. We claim that a point $(z_1,\ldots,z_{2k}) \in (\gamma_- \times \gamma_+)^k$ corresponds to a $2k$-periodic symplectic billiard trajectory on the phase space if and only if $z_{i+1} - z_{i-1} \in T_{z_i}\gamma_\pm$ and $z_{i+1} \neq z_{i-1}$ for all $i$ where we again read indices cyclically. Indeed, using strict convexity of $\gamma_-$ and $\gamma_+$ these conditions imply $T_{z_i}\gamma_\pm \neq T_{z_{i+1}}\gamma_\mp$, i.e.~$\det(\nu_{z_i},\nu_{z_{i+1}}) \neq 0$, in particular, $(z_i,z_{i+1})\in\P(\gamma_-,\gamma_+)$. Thus, Lemma \ref{lemma:critical_point_shoelace_formula} implies that the $2k$-periodic symplectic billiard trajectories are precisely the critical points of the function $f_k : (\gamma_- \times \gamma_+)^k \to \R$ which satisfy the additional condition $z_{i+1} \neq z_{i-1}$ for all $i$. 

Now, let $(w_1,\ldots,w_{2k})$ be a point at which $f_k$ attains its maximal value on the smooth closed manifold $(\gamma_- \times \gamma_+)^k$. In particular, $(w_1,\ldots,w_{2k})$ is a critical point of $f_k$. We claim that $w_{i+1} \neq w_{i-1}$ for all $i$. Assume for a contradiction that $w_{j+1} = w_{j-1}$ for some $j$. For notational convenience we assume $j\neq1,2k-1,2k$. Then \begin{equation}\nonumber
\begin{aligned}
f_k(w_1,\ldots,w_{2k}) &= f_{k-1}(w_1,\ldots, w_{j-1}, w_{j+2},\ldots,w_{2k}) +\omega(w_{j-1},w_j)+\omega(w_j,w_{j+1})\\
&= f_{k-1}(w_1,\ldots, w_{j-1}, w_{j+2},\ldots,w_{2k}),
\end{aligned}
\end{equation}
that is, the terms containing $w_j$ cancel each other. This implies that the maximal value of $f_{k-1}$ is larger than or equal to that of $f_k$ which directly contradicts Lemma \ref{lemma:maximal_value_shoelace_formula}. Therefore, the maximum $(w_1,\ldots,w_{2k})$ is indeed a critical point of $f_k$ with $w_{i+1} \neq w_{i-1}$ for all $i$ and thus a $2k$-periodic symplectic billiard trajectory. This proves the Theorem.
\end{proof}

\begin{remark}
The above proof barely used that we study symplectic billiards in the plane and should look almost identical for two smooth, strictly convex, closed hypersurfaces in $\R^{2n}$, as in \cite{Albers_Tabachnikov_Introducing_symplectic_billiards}. We leave the two table perspective on symplectic billiards in higher dimensions for future work. 
\end{remark}

\begin{figure}[ht]
    \centering
    \includegraphics[width=.6\linewidth]{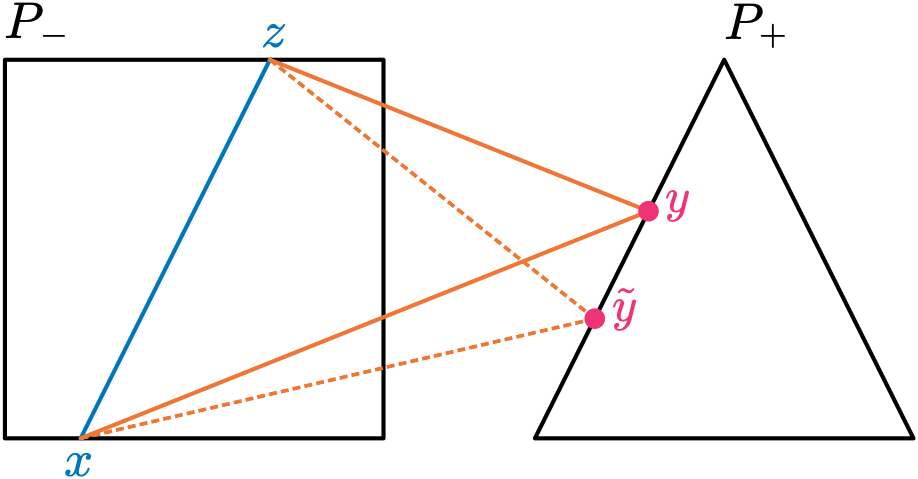}
    \caption{The triangles $xyz$ and $x\tilde yz$ have the same area since $x-z$ is parallel to the side on which $y$ and $\tilde y$ lie.}
    \label{fig:variational_principle}
\end{figure}

\begin{remark} \label{remark:variational_principle_and_invariant_area_form}
Similar to \cite{Albers_Tabachnikov_Introducing_symplectic_billiards} we give a variational characterization of the symplectic billiard rule. This also informs the definition of the functions $f_k$ from above. We recall this variational characterization here and denote for notational convenience by $S:=\omega:\R^2\times\R^2\to\R$ the standard symplectic form $\omega$ on $\R^2$.  For $(x,y)\in \P_\mathrm{max}$ we claim that $\phi(x,y)=(y,z)$ is equivalent to
$$
\frac{d}{dy}\Big|_{y\in P_\mp}\Big(S(x,y) + S(y,z)\Big) = 0.
$$
Indeed, $S(x,\tilde y) + S(\tilde y,z)$ equals twice the area of the triangle $x\tilde yz$, see Figure \ref{fig:variational_principle}. 

I.e.~the gradient with respect to $\tilde y$ of this area is perpendicular to $z-x$. By the principle of Lagrange multipliers we conclude that $\frac{d}{dy}\Big|_{y\in P_\mp}\Big(S(x,y) + S(y,z)\Big) = 0$ if and only if $z-x$ is parallel to the edge of $P_\mp$ containing $y$. Hence, we may call $S$ a generating function for the symplectic billiard map $\phi$. The same argument works in the smooth case.
\end{remark}

\bibliographystyle{amsalpha}
\bibliography{bibliography.bib}

\end{document}